\documentclass{amsart}
\usepackage{amssymb,latexsym}
\usepackage{amscd,amsthm}
\usepackage{tikz-cd}

\usepackage{amsmath}
\usepackage[all]{xy}
\usepackage{tikz}

\newtheorem{theorem}{Theorem}[section]

\newtheorem{them}{Theorem}

\newtheorem{lemma}[theorem]{Lemma}
\newtheorem{proposition}[theorem]{Proposition}
\newtheorem{corollary}[theorem]{Corollary}

\theoremstyle{definition}
\newtheorem{definition}[theorem]{Definition}
\newtheorem{fact}[theorem]{Fact}
\newtheorem{remark}{Remark}
\newtheorem*{note}{Note}
\newtheorem*{notation}{Notation}
\newtheorem{example}[theorem]{Example}

\newtheorem{setup}[theorem]{Set-Up}

\DeclareMathOperator{\Ext}{Ext}
\DeclareMathOperator{\Hom}{Hom}
\DeclareMathOperator{\Tor}{Tor}

\DeclareMathOperator{\im}{Im}

\newcommand{\cat}[1]{\mathcal{#1}}           

\newcommand{\class}[1]{\mathcal{#1}}   

\newcommand{\Z}{\mathbb{Z}}
\newcommand{\Q}{\mathbb{Q/Z}}
\newcommand{\mathcolon}{\colon\,} 

\newcommand{\ch}{\textnormal{Ch}(R)}

\newcommand{\dwclass}[1]{dw\widetilde{\class{#1}}}
\newcommand{\exclass}[1]{ex\widetilde{\class{#1}}}

\newcommand{\tilclass}[1]{\widetilde{\class{#1}}}
\newcommand{\dgclass}[1]{dg\widetilde{\class{#1}}}

\newcommand{\rightperp}[1]{#1^{\perp}}
\newcommand{\leftperp}[1]{{}^\perp #1}

\newcommand{\homcomplex}{\mathit{Hom}}

\newcommand{\Ho}[1]{{\textnormal{Ho}(#1)}}

\newcommand{\arr}{\xrightarrow{}}

\begin{document}

\title{Gorenstein complexes and recollements from cotorsion pairs}

\author{James Gillespie}
\address{Ramapo College of New Jersey \\
         School of Theoretical and Applied Science \\
         505 Ramapo Valley Road \\
         Mahwah, NJ 07430}
\email[Jim Gillespie]{jgillesp@ramapo.edu}
\urladdr{http://pages.ramapo.edu/~jgillesp/}

\dedicatory{Dedicated with gratitude to Mark Hovey}

\date{\today}

\begin{abstract}
We describe a general correspondence between injective (resp. projective) recollements of triangulated categories and injective (resp. projective) cotorsion pairs. This provides a model category description of these recollement situations. Our applications focus on displaying several recollements
that glue together various full subcategories of $K(R)$, the homotopy category of chain complexes of modules over a general ring $R$.
When $R$ is (left) Noetherian ring, these recollements involve complexes built from the Gorenstein injective modules. When $R$ is a (left) coherent ring for which all flat modules have finite projective dimension we obtain the duals. These results extend to a general ring $R$ by replacing the Gorenstein modules with the Gorenstein AC-modules introduced in~\cite{bravo-gillespie-hovey}.  We also see that in any abelian category with enough injectives, the Gorenstein injective objects enjoy a maximality property in that they contain every other class making up the right half of an injective cotorsion pair.
\end{abstract}

\maketitle

\section{introduction}

It has become clear that adjoint functors between homotopy categories of chain complexes are strongly connected with complete cotorsion pairs. For example, suppose $(\class{F},\class{C})$ is a cotorsion pair in the category of chain complexes of modules over a ring $R$ and assume $\class{F}$ is closed under suspensions. It is shown in~\cite{enochs-iacob-jenda-bravo-rada-adjoints and Ch(R)} that the full subcategory $K(\class{F})$ of the usual homotopy category $K(R)$ of chain complexes is right admissible, meaning the inclusion has a right adjoint. On the other hand, $K(\class{C})$  is left admissible. Knowing that the abstract categorical language for homotopy theory is that of model categories, and knowing that in algebraic settings model structures translate to cotorsion pairs, this paper aims to show there is a much deeper connection. In particular, we show how a recollement situation, which is a sort of elaborate gluing of triangulated categories with several adjoint pairs, may be obtained model categorically, using just the cotorsion pairs.

The idea relies on Becker's work in~\cite{becker}, where he finds and describes a beautiful way to localize two ``injective'' abelian model structures to obtain a third abelian model structure. The homotopy categories of the three involved model structures are then linked by a colocalization sequence and the localized model structure is in fact the right Bousfield localization of the first, killing the fibrant objects of the second. Becker goes on to recover Krause's recollement $K_{ex}(Inj) \xrightarrow{} K(Inj) \xrightarrow{} \class{D}(R)$ from~\cite{krause-stable derived cat of a Noetherian scheme} using the theory of abelian model categories. Here, $K(Inj)$ is the homotopy category of all complexes of injective modules, $K_{ex}(Inj)$ is the full subcategory of exact complexes of injectives, and $\class{D}(R)$ is the derived category of $R$. After reading~\cite{becker} the author recalled seeing in his work instances of injective cotorsion pairs that were ``linked'' in the same way that allowed Becker to obtain Krause's recollement. One of those situations were the three cotorsion pairs used by Becker to reconstruct the Krause recollement. But another two were ``Gorenstein injective'' versions of this. This inspired the author to look for general conditions that allowed for three injective cotorsion pairs to give a recollement. This appears as Theorem~\ref{them-finding recollements}, but we describe it in more detail below. The theorem appears quite useful as it allowed the author to not only put together the two Gorenstein injective versions of the recollement alluded to above, but to find an unexpected third variation of the recollement and then without much effort to spot two other interesting recollements involving the categorical Gorenstein injective complexes. A converse to Theorem~\ref{them-finding recollements} appears as Corollary~\ref{cor-bijective correspondences between admissible subcats and injective cot pairs} and indicates that from the model category perspective, the conditions relating the three cotorsion pairs in Theorem~\ref{them-finding recollements} characterize (injective) recollement situations. So clearly many more recollement situations, perhaps even all of those that arise in practice, ought to be describable via cotorsion pairs. In fact, several more applications involving exact categories will appear in a sequel to the current paper. We now describe in more detail the main results in this paper.

\subsection{Injective model structures and recollements from cotorsion pairs}

The theory of abelian model categories began in~\cite{hovey} where it was shown that an abelian category with a compatible model structure is equivalent to two complete cotorsion pairs. The definition of cotorsion pair appears as Definition~\ref{def-cotorsion pair} and Hovey's correspondence appears in Proposition~\ref{prop-Hovey's theorem}. Our main goal in Sections~\ref{sec-injective and weak injective hovey triples} and~\ref{sec-localization sequences and recollements from cotorsion pairs} is to understand localization theory, in particular recollement situations, in terms of Hovey's correspondence.

So let $\class{A}$ be an abelian category with enough injectives. We will call a complete cotorsion pair $(\class{W},\class{F})$ an \emph{injective cotorsion pair} whenever $\class{W}$ is \emph{thick} (Definition~\ref{def-thick}) and $\class{W} \cap \class{F}$ is the class of injective objects in $\class{A}$. As we will make clear, such a cotorsion pair is equivalent to an injective model structure on $\cat{A}$, which as defined in~\cite{gillespie-exact model structures} is an abelian model structure on $\cat{A}$ in which all objects are cofibrant.  The canonical example of an injective cotorsion pair would have to be $(\class{E},\dgclass{I})$ in the category of chain complexes $\ch$, where here $\class{E}$ is the class of all exact complexes and $\dgclass{I}$ is the class of DG-injective complexes. This cotorsion pair completely determines the standard injective model structure on $\ch$ for the derived category $\class{D}(R)$. Focussing on the injective cotorsion pairs themselves is convenient because it puts all of the essential information for an injective model structure in one small package. In particular, $\class{F}$ is the class of fibrant objects, $\class{W}$ are the trivial objects, $\class{W} \cap \class{F}$ are the injectives (trivially fibrant objects), and the entire model structure is determined from this information. This decluttering allows one to then focus on the relationships between the classes appearing in different cotorsion pairs, which turns out to be useful for spotting recollement situations.

Becker too considers injective cotorsion pairs in~\cite{becker}, although he did not call them this. Becker shows that if $\class{M}_1 = (\class{W}_1,\class{F}_1)$ and $\class{M}_2 = (\class{W}_2,\class{F}_2)$ are injective cotorsion pairs (so model structures) with $\class{F}_2 \subseteq \class{F}_1$ then there is a colocalization sequence associated to the homotopy categories $\Ho{\class{M}_2} \xrightarrow{\textnormal{R}\,\textnormal{id}} \Ho{\class{M}_1} \xrightarrow{\textnormal{R}\,\textnormal{id}} \Ho{\class{M}}$ where $\class{M} = \class{M}_1/\class{M}_2$ is the right Bousfield localization of $\class{M}_1$ with respect to killing the class $\class{F}_2$ of fibrant objects in $\class{M}_2$. This is central to this paper and the exact statement reappears in Proposition~\ref{prop-Beckers theorem}. We note two additions in how we state Becker's result. First we define in Section~\ref{sec-injective and weak injective hovey triples} the notion of a \emph{weak injective} model structure. We point out that Becker's theorem has a converse in that every hereditary weak injective model structure is the right localization of two injective ones. We also add a uniqueness property that the class of trivial objects satisfy. These two additions are not hard but help give a bigger picture to the theory. More importantly, the uniqueness condition is quite useful for actually \emph{spotting} a localization, or even potential recollement situation.

In light of the above, the very rough idea then behind spotting a recollement is to find three injective cotorsion pairs $\class{M}_1 = (\class{W}_1,\class{F}_1)$, $\class{M}_2 = (\class{W}_2,\class{F}_2)$, and $\class{M}_3 = (\class{W}_3,\class{F}_3)$ for which $\class{M}_1/\class{M}_2$ is Quillen equivalent to $\class{M}_3$ and vice versa $\class{M}_1/\class{M}_3$ is Quillen equivalent to $\class{M}_2$. But to get a recollement one needs to ``glue'' together a localization sequence and a colocalization sequence. The exact condition allowing such a ``gluing'' is this: $\class{F}_3 \subseteq \class{F}_1$ and $\class{W}_3 \cap \class{F}_1 = \class{F}_2$. A precise statement of the main theorem follows.

\begin{them}\label{them-A}
Let $\cat{A}$ be an abelian category with enough injectives and suppose we have three injective cotorsion pairs $$\class{M}_1 = (\class{W}_1, \class{F}_1) , \ \ \ \class{M}_2 = (\class{W}_2, \class{F}_2) , \ \ \ \class{M}_3 = (\class{W}_3, \class{F}_3)$$ such that $\class{F}_3 \subseteq  \class{F}_1$ and $\class{W}_3 \cap \class{F}_1 = \class{F}_2$. Then
\begin{enumerate}
\item The localization $\class{M}_1/\class{M}_2$ is Quillen equivalent to $\class{M}_3$ while $\class{M}_1/\class{M}_3$ is Quillen equivalent to $\class{M}_2$.
\item There is a recollement
\[
\xy
(-30,0)*+{\class{F}_2/\sim };
(0,0)*+{\class{F}_1/\sim};
(25,0)*+{\class{F}_3/\sim};
{(-19,0) \ar (-10,0)};
{(-10,0) \ar@<0.5em> (-19,0)};
{(-10,0) \ar@<-0.5em> (-19,0)};
{(10,0) \ar (19,0)};
{(19,0) \ar@<0.5em> (10,0)};
{(19,0) \ar@<-0.5em> (10,0)};
\endxy
.\]  where here $f \sim g$ if and only if $g-f$ factors through an injective object.
\item There is a converse to \textnormal{(2)} above.
\end{enumerate}
\end{them}
Parts (1) and (2) appear in Theorem~\ref{them-finding recollements} while the converse appears in Corollary~\ref{cor-bijective correspondences between admissible subcats and injective cot pairs}.

A few comments on Theorem~\ref{them-A} are in order. First, we point out that all the functors in the recollement are easily described. They are all just left or right derived identity functors between the three model structures, and in practice they correspond to taking precovers or preenvelopes using completeness of the cotorsion pairs.

Second, the author wishes to point out that the converse (3) and its proof were pointed out to him by the anonymous referee of an earlier version of this paper! In short, the converse says that starting with an ambient injective cotorsion pair $(\class{W},\class{F})$ for which $\class{F}/\sim$ sits in the center of a recollement, then there must exist two injective cotorsion pairs $(\class{W}',\class{F}')$, and $(\class{W}'',\class{F}'')$ with $\class{F}'' \subseteq \class{F}$ and $\class{W}'' \cap \class{F} = \class{F}'$, which recover the recollement. Though this converse is not used in the rest of the paper it clearly makes the theory more interesting as it indicates how (injective) recollements appear model categorically.

Third, while the applications of the above theorem in the current paper are to triangulated categories associated to $\ch$, the author wishes to stress that Theorem~\ref{them-finding recollements} holds in the full generality of abelian categories with enough injectives and should prove useful in more general settings.

Last, the dual notion of a \emph{projective cotorsion pair} is equally useful. There are versions of all our results for projective cotorsion pairs in abelian categories with enough projectives.

\subsection{Gorenstein complexes and Gorenstein derived categories}

The remainder of this paper, beginning in Section~\ref{sec-the Gorenstein injective cotorsion pair} is concerned with applications of Theorem~\ref{them-A} to Gorenstein homological algebra and its extension to arbitrary rings in~\cite{bravo-gillespie-hovey}. In Gorenstein homological algebra we essentially replace injective, projective, and flat modules with the Gorenstein projective, Gorenstein injective, and Gorenstein flat modules.
We call an object $M$ in an abelian category with enough injectives \emph{Gorenstein injective} if $M = Z_0J$ for some exact complex of injectives $J$ which remains exact after applying $\Hom(I,-)$ for any other injective object $I$. The Gorenstein projectives are the dual which we define in an abelian category with enough projectives.
Ideally we would like classical results concerning the projectives, injectives and flats to have analogs in Gorenstein homological algebra. In particular, given a ring $R$, the most fundamental question in Gorenstein homological algebra is whether or not Gorenstein injective preenvelopes (resp. Gorenstein projective precovers) exist in $R$-Mod. Said another way, we wish to know when the Gorenstein injectives makes up the right half of a complete cotorsion pair. This first of all allows for the construction of Gorenstein derived functors which is the starting point of the theory; see~\cite{enochs-jenda-book}. Second it leads to a generalized Tate cohomology theory which has been pursued in different ways by many authors. See for example~\cite{benson-infinite}, \cite{jorgensen}, \cite{iacob-generalized tate} and~\cite{martsinkovsky-avramov}. In light of this and Theorem~\ref{them-A}, it is natural to consider Gorenstein homological algebra in this paper because of the following result appearing as Theorem~\ref{them-Gorenstein projectives are at top of lattice}.

\begin{them}\label{them-B}
Let $\cat{A}$ be any abelian category with enough injectives. Let $\class{I}$ denote the class of all injective objects in $\class{A}$ and $\class{G}\class{I}$ denote the class of Gorenstein injectives.
\begin{enumerate}
\item For any injective cotorsion pair $(\class{W},\class{F})$, we have $\class{I} \subseteq \class{F} \subseteq \class{GI}$.
\item Whenever $(\leftperp{\class{G}\class{I}}, \class{G}\class{I})$ is a complete cotorsion pair it is automatically an injective cotorsion pair too.
\end{enumerate}
\end{them}
In other words, the Gorenstein injectives are the largest possible class that can appear as the right half of an injective cotorsion pair. Therefore we obtain a recollement from Theorem~\ref{them-A} whenever $\class{W}_3 \cap \class{GI} = \class{F}_2$, and this is the key to obtaining up to 10 recollements involving complexes of Gorenstein modules as we describe more below in Theorems~\ref{them-C} and~\ref{them-D}. This also leads us to investigate in Sections~\ref{sec-lattice} and~\ref{sec-lifting from modules to chain complexes} a possible lattice structure on the class of all injective cotorsion pairs. For the Gorenstein injectives are the maximal element in the lattice while the usual injectives are the minimal element.

Categorically speaking, the abstract home for Tate cohomology is the \emph{stable module category} which was defined in~\cite{hovey} for Gorenstein rings by putting a model structure on the category of $R$-modules whose fibrant objects are the Gorenstein injectives. The existence of a nice stable module category boils down, again, to the existence of Gorenstein injective preenvelopes. In~\cite{bravo-gillespie-hovey} it is shown that the Gorenstein injective model structure on $R$-Mod from~\cite{hovey} extends to any (left) Noetherian ring $R$, and that the Gorenstein projective model structure extends to any (left) coherent ring $R$ for which all flat modules have finite projective dimension. But the most interesting aspect of the work in~\cite{bravo-gillespie-hovey} is that it provides a way to extend Gorenstein homological algebra to arbitrary rings. It introduces two new classes of modules we call \emph{absolutely clean} modules and \emph{level} modules and proves that a perfect duality exists between these classes with respect to taking character modules. When $R$ is Noetherian, the absolutely clean modules are nothing more than the usual injective modules while the level modules are exactly the flat modules. When $R$ is coherent, the absolutely clean modules are the absolutely pure (or FP-injective) modules while still the level modules are the flat modules. This duality between the absolutely clean and level modules is key to extending Gorenstein homological algebra. For a general ring $R$, we then call a module $M$ \emph{Gorenstein AC-injective} if $M = Z_0J$ for some exact complex of injectives $J$ for which $\Hom_R(I,J)$ is also exact for any absolutely clean module $I$. To emphasize, when $R$ is Noetherian, the Gorenstein AC-injectives are just the usual Gorenstein injectives. There is the dual notion of \emph{Gorenstein AC-projective} modules and these coincide with the usual Gorenstein projectives for sufficiently nice rings. It is shown in~\cite{bravo-gillespie-hovey} that for any ring $R$, $(\class{W},\class{GI})$ is an injective cotorsion pair, where $\class{GI}$ is the class of Gorenstein AC-injectives. This provides a generalization of the stable module category to any ring $R$. On the other hand, we show $(\class{GP},\class{V})$ is a projective cotorsion pair where $\class{GP}$ is the class of Gorenstein AC-projectives. This provides another generalization of the stable module category, but for many interesting rings they will coincide.

In the current paper we first of all lift these results to the level of chain complexes by showing that each chain complex $X$ of $R$-modules has a Gorenstein injective preenvelope and an exact Gorenstein injective preenvelope when $R$ is Noetherian. It follows that $K(GInj)$ and $K_{ex}(GInj)$ each appear as the homotopy category of an injective model structure on $\ch$. Here $K(GInj)$ denotes the homotopy category of all (categorical) Gorenstein injective chain complexes and $K_{ex}(GInj)$ denotes homotopy category of all exact Gorenstein injective complexes. But the more interesting thing is that we obtain several new recollements by using the correspondence between cotorsion pairs and recollements in Theorem~\ref{them-A}. For a general ring $R$, there are potentially 10 recollements we see (five injective ones and five projective ones) involving complexes of Gorenstein modules. We summarize the most interesting ones below in Theorems~\ref{them-C} and~\ref{them-D}.

\begin{them}\label{them-C}
Let $\class{D}(R)$ denote the derived category of $R$. Then the following hold.
\begin{enumerate}
\item If $R$ is (left) Noetherian ring, then there is a recollement
\[
\xy
(-30,0)*+{K_{ex}(GInj) };
(0,0)*+{K(GInj)};
(25,0)*+{\class{D}(R)};
{(-19,0) \ar (-10,0)};
{(-10,0) \ar@<0.5em> (-19,0)};
{(-10,0) \ar@<-0.5em> (-19,0)};
{(10,0) \ar (19,0)};
{(19,0) \ar@<0.5em> (10,0)};
{(19,0) \ar@<-0.5em> (10,0)};
\endxy
.\]
This says that the homotopy category of Gorenstein injective complexes is obtained by gluing the derived category together with the homotopy category of exact Gorenstein injective complexes.

\item If $R$ is (left) coherent ring and all flat modules have finite projective dimension, then the dual Gorenstein projective version of the above recollement in \textnormal{(1)} exists. More generally, for any ring $R$ in which all level modules have finite projective dimension we have the same result.

\item Both the injective recollement of \textnormal{(1)} and the projective recollement of \textnormal{(2)} extend to arbitrary rings by replacing the Gorenstein injective complexes (resp. Gorenstein projective complexes) with the complexes of Gorenstein AC-injectives (resp. Gorenstein AC-projectives).

\end{enumerate}
\end{them}

These results appear in Theorem~\ref{them-three Gorenstein version of Krause recollement}, Theorem~\ref{them-three Gorenstein projective versions of Krause recollement}, and Theorem~\ref{them-Gorenstein AC-injective recollements} where some other variations of these recollements also appear. As we explore in Sections~\ref{sec-lattice} and~\ref{sec-lifting from modules to chain complexes} these other variations may or may not be distinct from the above recollements, nor need they be nontrivial. The farther away from Gorenstein the ring is, the more recollements we see.

Finally, we now describe one of two beautiful recollements that appear near the end of the paper. For any ring $R$, let $(\class{W},\class{GI})$ denote the injective cotorsion pair where $\class{GI}$ are the Gorenstein AC-injectives and $\class{W}$ are the trivial objects in the corresponding stable module category. Now let $\exclass{W}$ denote the class of all exact complexes $W$ for which each $W_n \in \class{W}$. We will see that there is an injective model structure on $\ch$ with $\exclass{W}$ as its class of trivial objects. Let  $\class{D}(\exclass{W})$ denote the associated triangulated homotopy category.

\begin{them}\label{them-D}
In the above set-up we have the following.
\begin{enumerate}
\item If $R$ is Noetherian then there is a recollement
\[
\xy
(-30,0)*+{K_{ex}(Inj) };
(0,0)*+{K(GInj)};
(29,0)*+{\class{D}(\exclass{W})};
{(-19,0) \ar (-10,0)};
{(-10,0) \ar@<0.5em> (-19,0)};
{(-10,0) \ar@<-0.5em> (-19,0)};
{(10,0) \ar (19,0)};
{(19,0) \ar@<0.5em> (10,0)};
{(19,0) \ar@<-0.5em> (10,0)};
\endxy
.\]
This says that the homotopy category of Gorenstein injective complexes is obtained by gluing $\class{D}(\exclass{W})$ together with the stable derived category $S(R) = K_{ex}(Inj)$.

\item The injective recollement of \textnormal{(1)} still holds for an arbitrary ring $R$ by replacing the Gorenstein injective complexes by the complexes of Gorenstein AC-injectives.

\item The projective dual to both \textnormal{(1)} and \textnormal{(2)} above hold. In particular, the projective dual of \textnormal{(1)} holds if $R$ is coherent and all flat modules have finite projective dimension.
\end{enumerate}

\end{them}

\vspace{.1in}

The outline of the paper is as follows. Preliminary notions and notation are laid out in Section~\ref{sec-preliminaries}. In Section~\ref{sec-injective and weak injective hovey triples} we look at the necessary notions of Hovey triples, injective cotorsion pairs, weak injective cotorsion pairs, and Becker's localization theorem. In Section~\ref{sec-localization sequences and recollements from cotorsion pairs} we discuss recollements and the main Theorem~\ref{them-A} relating them to injective cotorsion pairs. In Section~\ref{sec-the Gorenstein injective cotorsion pair} we prove the maximality property of the Gorenstein injective cotorsion pair in Theorem~\ref{them-B}. In Section~\ref{sec-lattice} we look at the semilattice of injective cotorsion pairs. Then in Section~\ref{sec-lifting from modules to chain complexes} we show how to lift an injective cotorsion pair from $R$-Mod to six in $\ch$. Finally in Section~\ref{sec-Gorenstein projective and injective models for the derived category} we look exclusively at the Gorenstein complexes and the Krause-like recollements of Theorem~\ref{them-C}, parts~(1) and~(2). In Section~\ref{sec-more recollements} we conclude Theorem~\ref{them-D}, parts~(1) and~(2), as well as an equally interesting variant of this recollement. In Section~\ref{sec-Gorenstein AC recollements} we explain a little more detail on the work of~\cite{bravo-gillespie-hovey} which results in the extensions of our recollements to arbitrary rings $R$. That is, part~(3) of Theorems~\ref{them-C} and~\ref{them-D}.

\subsection{Acknowledgements}
The author thanks the anonymous referee of an earlier version of this paper. In particular, the contributions in Subsection~\ref{subsec-converse of main theorem} are due to the referee. The author also wishes to thank Daniel Bravo for offering much help in typesetting this paper with \LaTeX, Mark Hovey for suggestions on how to make this paper more readable, and Henrik Holm for enjoyable discussions on Gorenstein modules. Finally, the author acknowledges the inspiration he received from reading the paper~\cite{becker}.

\section{preliminaries}\label{sec-preliminaries}

Here we either recall or define fundamental ideas and set some notation which we
use throughout the paper. By a ring $R$, we always mean a ring with 1. Central to this paper is Quillen's notion of a model category~\cite{quillen}. Our basic reference is~\cite{hovey-model-categories} and we are concerned with abelian model structures. The basic theory of abelian model structures is in~\cite{hovey}, but we will recall the correspondence with cotorsion pairs in this section below.

\subsection{Chain complexes}

We denote by $R$-Mod the category of (left) $R$-modules and by $\ch$ the category
of chain complexes of (left) $R$-modules. We write chain complexes as $\cdots
\xrightarrow{} X_{n+1} \xrightarrow{d_{n+1}} X_{n} \xrightarrow{d_n}
X_{n-1} \xrightarrow{} \cdots$ so that the differentials lower the degree.

We let $S^{n}(M)$ denote the chain
complex with all entries 0 except $M$ in degree $n$. We let $D^{n}(M)$
denote the chain complex $X$ with $X_{n} = X_{n-1} = M$ and all other
entries 0.  All maps are 0 except $d_{n} = 1_{M}$. Given $X$, the
\emph{suspension of $X$}, denoted $\Sigma X$, is the complex given by
$(\Sigma X)_{n} = X_{n-1}$ and $(d_{\Sigma X})_{n} = -d_{n}$.  The
complex $\Sigma (\Sigma X)$ is denoted $\Sigma^{2} X$ and inductively
we define $\Sigma^{n} X$ for all $n \in \Z$.

Given two chain complexes $X$ and $Y$ we define $\homcomplex(X,Y)$ to
be the complex of abelian groups $ \cdots \xrightarrow{} \prod_{k \in
\Z} \Hom(X_{k},Y_{k+n}) \xrightarrow{\delta_{n}} \prod_{k \in \Z}
\Hom(X_{k},Y_{k+n-1}) \xrightarrow{} \cdots$, where $(\delta_{n}f)_{k}
= d_{k+n}f_{k} - (-1)^n f_{k-1}d_{k}$.  This gives a functor
$\homcomplex(X,-) \mathcolon \ch \xrightarrow{} \textnormal{Ch}(\Z)$
which is left exact, and exact if $X_{n}$ is projective for all $n$.
Similarly the contravariant functor $\homcomplex(-,Y)$ sends right
exact sequences to left exact sequences and is exact if $Y_{n}$ is
injective for all $n$.

Recall that $\Ext^1_{\ch}(X,Y)$ is the group of (equivalence classes)
of short exact sequences $0 \xrightarrow{} Y \xrightarrow{} Z \xrightarrow{} X \xrightarrow{} 0$ under the Baer
sum. We let $\Ext^1_{dw}(X,Y)$ be the subgroup of $\Ext^1_{\ch}(X,Y)$
consisting of those short exact sequences which are split in each
degree. We often make use of the following standard fact.

\begin{lemma}\label{lemma-homcomplex-basic-lemma}
For chain complexes $X$ and $Y$, we have
$$\Ext^1_{dw}(X,\Sigma^{(-n-1)}Y) \cong H_n \homcomplex(X,Y) =
Ch(R)(X,\Sigma^{-n} Y)/\sim \ ,
$$ where $\sim$ is chain homotopy.
\end{lemma}

In particular, for chain complexes $X$ and $Y$, $\homcomplex(X,Y)$ is
exact iff for any $n \in \Z$, any $f \mathcolon \Sigma^nX \xrightarrow{} Y$ is
homotopic to 0 (or iff any $f \mathcolon X \xrightarrow{} \Sigma^nY$ is homotopic
to 0).

\subsection{Cotorsion pairs} The most important concept in this paper is that of a cotorsion pair in an abelian category $\cat{A}$. A standard reference is~\cite{enochs-jenda-book}.

\begin{definition}\label{def-cotorsion pair}
A pair of classes $(\class{F},\class{C})$ in an abelian category
$\cat{A}$ is a cotorsion pair if the following conditions hold:
\begin{enumerate}
    \item $\Ext^1_{\cat{A}}(F,C) = 0$ for all $F \in \class{F}$ and $C \in
    \class{C}$.

    \item If $\Ext^1_{\cat{A}}(F,X) = 0$ for all $F \in \class{F}$, then $X
    \in \class{C}$.

    \item If $\Ext^1_{\cat{A}}(X,C) = 0$ for all $C \in \class{C}$, then $X \in
    \class{F}$.
\end{enumerate}
\end{definition}

A cotorsion pair is said to have \emph{enough projectives} if for
any $X \in \cat{A}$ there is a short exact sequence $0 \xrightarrow{} C \xrightarrow{} F \xrightarrow{} X \xrightarrow{} 0$ where $C \in
\class{C}$ and $F \in \class{F}$. We say it has \emph{enough injectives} if it satisfies the dual statement. If both of these hold
we say the cotorsion pair is \emph{complete}. Whenever the category $\cat{A}$ has enough injectives and projectives then a cotorsion pair $(\class{F},\class{C})$ is complete iff $(\class{F},\class{C})$ has enough injectives iff $(\class{F},\class{C})$ has enough projectives.~\cite[Proposition~7.1.7]{enochs-jenda-book}.

In $R$-Mod, the class of projectives is the left half of an obvious complete cotorsion pair while the class of injectives is the right half of an obvious complete cotorsion pair. The most well-known nontrivial example of a complete cotorsion pair is
$(\class{F},\class{C})$ where $\class{F}$ is the class of flat
modules and $\class{C}$ are the cotorsion modules. A proof that this is a complete cotorsion theory can be found in~\cite{enochs-jenda-book}. We will be concerned with cotorsion pairs of chain complexes in this paper and will use results from~\cite{gillespie} and~\cite{gillespie-degreewise-model-strucs}. These results describe several cotorsion pairs in $\ch$ that can be associated to a single given cotorsion pair in $R$-Mod. This will be encountered in Section~\ref{sec-lifting from modules to chain complexes} where we will recall the basic definitions.

\subsection{Hereditary cotorsion pairs} Let $\cat{A}$ be an abelian category. We say a cotorsion pair $(\class{F},\class{C})$ in $\cat{A}$ is \emph{resolving} if $\class{F}$ is closed under taking kernels of epimorphisms between objects of $\class{F}$. That is, if for any short exact sequence $0 \xrightarrow{} X \xrightarrow{} Y \xrightarrow{} Z \xrightarrow{} 0$, we have $X$ in $\class{F}$ whenever $Y$ and $Z$ are in $\class{F}$. We say  $(\class{F},\class{C})$ is \emph{coresolving} if the right hand class $\class{C}$ satisfies the dual. Finally, we say that  $(\class{F},\class{C})$ is \emph{hereditary} if it is both resolving and coresolving. The following is a standard test for checking to see if a given cotorsion pair is hereditary.

\begin{lemma}[Hereditary Test]\label{lemma-hereditary test}
Assume $(\class{F},\class{C})$ is a cotorsion pair in an abelian category $\cat{A}$. Consider the statements below.
\begin{enumerate}
\item  $\class{F}$ is closed under taking kernels of epimorphisms between objects of $\class{F}$.
\item $\class{F}$ is syzygy closed, meaning $X \in \class{F}$ whenever we have an exact $0 \xrightarrow{} X \xrightarrow{} P \xrightarrow{} Z \xrightarrow{} 0$ with $P$ projective and $Z \in \class{F}$.
\item  $\class{C}$ is closed under taking cokernels of monomorphisms between objects of $\class{C}$.
\item $\class{C}$ is cosyzygy closed, meaning $Z \in \class{C}$ whenever we have an exact $0 \xrightarrow{} X \xrightarrow{} I \xrightarrow{} Z \xrightarrow{} 0$ with $I$ injective and $X \in \class{C}$.
\item $\Ext^2_{\cat{A}}(F,C) = 0$ whenever $F \in \class{F}$ and $C \in \class{C}$.
\end{enumerate}
Then we have the following implications depending on whether or not $\class{A}$ has enough injectives or projectives.
\begin{itemize}
\item If $\cat{A}$ has enough injectives then we have (3) implies (4) implies (5) implies (1) implies (2).
\item If $\cat{A}$ has enough projective then we have (1) implies (2) implies (5) implies (3) implies (4).
\item (5) implies (1) -- (4) without any assumption on enough injectives or projectives. (Just using the long exact sequence in $\Ext$.)
\end{itemize}
In particular, when $\cat{A}$ is a category with enough projectives and injectives then all of the conditions (1) -- (5) are equivalent.
\end{lemma}

\begin{proof}
See~\cite{garcia-rozas}.
\end{proof}

So suppose $\cat{A}$ has enough injectives and you have a cotorsion pair $(\class{F},\class{C})$. Then from Lemma~\ref{lemma-hereditary test} the cotorsion pair is hereditary whenever it is known to just be coresolving, or even if $\class{C}$ is just cosyzygy closed. However it can't be used to conclude hereditary in the case that you only know it is resolving. But for a complete cotorsion pair, the following is a wonderfully convenient result that appears as Corollary~1.1.13 of~\cite{becker}.

\begin{lemma}[Becker's Lemma]\label{lemma-Beckers lemma}
Let $\cat{A}$ be an abelian category and let $(\class{F},\class{C})$ be a complete cotorsion pair. Then the following are equivalent.
\begin{enumerate}
\item $(\class{F},\class{C})$ is hereditary.
\item $(\class{F},\class{C})$ is resolving.
\item $(\class{F},\class{C})$ is coresolving.
\end{enumerate}

\end{lemma}

\subsection{Hovey's correspondence}

Hovey defines abelian model categories in~\cite{hovey}. He then characterizes them in terms of cotorsion pairs as we now describe. So in fact one could even take the cotorsion pairs given in the correspondence below as the definition of an abelian model category. First, we need the definition of a thick subcategory.

\begin{definition}\label{def-thick}
By a \emph{thick subcategory} of an abelian $\cat{A}$ we mean a class of objects $\class{W}$ which is closed under direct summands and such that if two out of three of the terms in a short exact sequence are in $\class{W}$, then so is the third.
\end{definition}

\begin{proposition}\label{prop-Hovey's theorem}
Let $\class{A}$ be an abelian category with an abelian model structure. Let $\class{Q}$ be the class of cofibrant objects, $\class{R}$ the class of fibrant objects and $\class{W}$ the class of trivial objects. Then $\class{W}$ is a thick subcategory of $\cat{A}$ and both $(\class{Q},\class{W} \cap \class{R})$ and $(\class{Q} \cap \class{W} , \class{R})$ are complete cotorsion pairs in $\cat{A}$. Conversely, given a thick subcategory $\class{W}$ and classes $\class{Q}$ and $\class{R}$ making $(\class{Q},\class{W} \cap \class{R})$ and $(\class{Q} \cap \class{W} , \class{R})$ each complete cotorsion pairs, then there is an abelian model structure on $\cat{A}$ where $\class{Q}$ are the cofibrant objects, $\class{R}$ are the fibrant objects and $\class{W}$ are the trivial objects.
\end{proposition}

We point out that the abelian model structure on $\cat{A}$ is then completely determined by the classes $\class{Q}$, $\class{W}$ and $\class{R}$. Indeed the cofibrations (resp. trivial cofibrations) are the monomorphisms with cokernel in $\class{Q}$ (resp. $\class{Q} \cap \class{W}$) and the fibrations are the epimorphisms with kernel in $\class{R}$ (resp. $\class{W} \cap \class{R}$). The weak equivalences are the maps which factor as a trivial cofibration followed by a trivial fibration. However, the description of the weak equivalences given by the Lemma below is sometimes more convenient. It follows from~\cite[Lemma~5.8]{hovey} and an application of the two out of three axiom.

\begin{lemma}\label{lemma-characterization of weak equivalences in exact categories}
Say $\cat{A}$ is an abelian category with an abelian model structure. Let $\class{W}$ denote the class of trivial objects. Then a map $f$ is a weak equivalence if and only if it factors as a monomorphism with cokernel in $\class{W}$ followed by an epimorphism with kernel in $\class{W}$.
\end{lemma}

\begin{remark}\label{remark-bicompleteness}
We note that one normally assumes the category $\cat{A}$ is bicomplete for it to qualify as a model category. However, as explained in~\cite{gillespie-exact model structures} an abelian category with an abelian model structure already has enough colimits and limits to get the basic first results of homotopy theory. So we can always discuss model \emph{structures} on abelian categories $\cat{A}$, but we will deliberately reserve the term model \emph{category} when $\cat{A}$ is known to be bicomplete.

\end{remark}

\section{Injective and weak injective Hovey triples}\label{sec-injective and weak injective hovey triples}

It was reading~\cite{becker} that led the author to denote abelian model structures as triples and this eventually led to the definition of a weak injective model structure as below. Throughout this entire section assume $\cat{A}$ is an abelian category.

\subsection{Hovey triples} We start with a convenient definition.

\begin{definition}\label{def-hovey triple}
Let $\class{Q}$, $\class{W}$, and $\class{R}$ be three classes in $\cat{A}$ as in Hovey's correspondence. Then we call $(\class{Q},\class{W},\class{R})$ a \emph{Hovey triple}. By a \emph{hereditary} Hovey triple we mean that the two corresponding cotorsion pairs $(\class{Q} , \class{W} \cap \class{R})$ and $(\class{Q} \cap \class{W} , \class{R})$ are each hereditary.
\end{definition}

\begin{notation}
Due to Hovey's one-to-one correspondence between Hovey triples in $\cat{A}$ and abelian model structures on $\cat{A}$ we often will not distinguish between the Hovey triple and the actual model structure. For example, we may say $\class{M} = (\class{Q},\class{W},\class{R})$ is an abelian model structure and understand this to mean the model structure associated to the Hovey triple $(\class{Q},\class{W},\class{R})$. On the other hand, we may say that an abelian model structure is \emph{hereditary} and by this we mean that the Hovey triple is hereditary.
\end{notation}

The following is an easy but useful fact about Hovey triples.

\begin{proposition}[Characterization of trivial objects]\label{prop-characterization of trivial objects}
Suppose $(\class{Q}, \class{W}, \class{R})$ is a Hovey triple in $\class{A}$.
Then the thick class $\class{W}$ is characterized two ways as follows:
  \begin{align*}
   \class{W}  &= \{\, A \in \class{A} \, | \, \exists \, \text{s.e.s.} \, 0 \arr A \arr W_2 \arr W_1 \arr 0 \, \text{ with} \, W_2 \in \class{W} \cap \class{R} \, , W_1 \in \class{Q} \cap \class{W} \,\} \\
           &= \{\, A \in \class{A} \, | \, \exists \, \text{s.e.s.} \, 0 \arr W_2 \arr W_1 \arr A \arr 0 \, \text{ with} \,  W_2 \in \class{W} \cap \class{R} \, , W_1 \in \class{Q} \cap \class{W} \,\}.
          \end{align*}  Consequently, whenever $(\class{Q}, \class{V}, \class{R})$ is a Hovey triple, then necessarily $\class{V} = \class{W}$.

\end{proposition}

\begin{proof}
($\supseteq$) It is clear that each class described is contained in $\class{W}$ since $\class{W}$ is handed to us thick.
($\subseteq$) Let $W \in \class{W}$. Then apply enough injectives of $(\class{Q}, \class{W} \cap \class{R})$ to get a short exact sequence $0 \xrightarrow{} W \xrightarrow{} W_2 \xrightarrow{} W_1 \xrightarrow{} 0$ where $W_2 \in \class{W} \cap \class{R}$ and $W_1 \in \class{Q}$. But indeed $W_1 \in \class{Q} \cap \class{W}$ since $\class{W}$ is thick. So $W$ is in the top class described. On the other hand, $W$ is also in the bottom class described by a similar argument using enough projectives of the cotorsion pair $(\class{Q} \cap \class{W}, \class{R})$.

Now if $(\class{Q}, \class{V}, \class{R})$ is any other Hovey triple we must have $\class{V} \cap \class{R} = \class{W} \cap \class{R}$ (since each equal $\rightperp{\class{Q}}$) and $\class{Q} \cap \class{V} = \class{Q} \cap \class{W}$ (since each equal $\leftperp{\class{R}}$). So by what we just proved in the last paragraph it immediately follows that $\class{V} = \class{W}$.

\end{proof}

\subsection{Injective and weak injective Hovey triples}

\begin{definition}\label{def-weak injective Hovey triples}
We call a Hovey triple  $(\class{Q}, \class{W}, \class{R})$ \emph{injective} if $\class{Q} = \class{A}$. That is, if every object of $\class{A}$ is cofibrant.
By a \emph{weak injective} Hovey triple we mean a Hovey triple $(\class{Q},\class{W},\class{R})$ for which $\class{Q} \cap \class{W} \cap \class{R}$ is exactly the class of injective objects in $\class{A}$. We have the obvious dual notions of \emph{projective} Hovey triples and \emph{weak projective} Hovey triples.

We apply these same phrases to the actual model structures induced as well. So we speak of \emph{injective}, \emph{weak injective}, \emph{projective}, and \emph{weak projective} model structures on $\cat{A}$.
 \end{definition}

\begin{remark}
It is immediate that an injective Hovey triple is a weak injective Hovey triple. Also an injective Hovey triple is automatically hereditary by Becker's Lemma~\ref{lemma-Beckers lemma} because $\class{W}$ is thick.

\end{remark}

Note that whenever $\class{M} = (\class{Q},\class{W},\class{R})$ is an injective Hovey triple then the class $\class{Q}$ is redundant and need not be mentioned. Indeed as long as $\class{A}$ has enough injectives, then an injective model structure on $\cat{A}$ is equivalent to a single complete cotorsion pair $\class{M} = (\class{W},\class{R})$ where $\class{W}$ is thick and $\class{W} \cap \class{R}$ is the class of injective objects. We make this precise in the next definition.

\begin{definition}\label{def-injective cotorsion pairs}
Let $\cat{A}$ be an abelian model category with enough injectives. We call a complete cotorsion pair $(\class{W},\class{F})$ an \emph{injective cotorsion pair} if $\class{W}$ is thick and $\class{W} \cap \class{F}$ coincides with the class of injective objects. In this case $\class{M} = (\class{W},\class{F})$ determines an injective model structure on $\cat{A}$ with $\class{F}$ the class of fibrant objects and $\class{W}$ the trivial objects. On the other hand, if $\cat{A}$ has enough projectives we define the dual notion of a \emph{projective cotorsion pair} $(\class{C},\class{W})$.

\end{definition}

\begin{remark}
We emphasize that we don't use the term \emph{injective cotorsion pair} unless $\cat{A}$ has enough injectives and we don't use the phrase \emph{projective cotorsion pair} unless $\cat{A}$ has enough projectives. The point is that a projective cotorsion pair IS a model structure on $\cat{A}$ but you don't have a model structure if $\cat{A}$ lacks projectives. The definitions are useful though because now all of the essential information needed for such a model structure is packed into one simple idea - a single cotorsion pair.

\end{remark}

The definition of an injective cotorsion pair is stronger that what we need. This isn't just interesting but it is relevant for stating a converse to Becker's localization theorem~(Proposition~\ref{prop-Beckers theorem}). The author learned the Lemma below from Henrik Holm.

\begin{lemma}\label{lem-Henriks lemma}
Assume $(\class{W},\class{F})$ is an hereditary cotorsion pair in $\class{A}$. Suppose that for any $F \in \class{F}$ we can find a short exact sequence $0 \xrightarrow{} F' \xrightarrow{} I \xrightarrow{} F \xrightarrow{} 0$ where $I$ is injective and $F' \in \class{F}$. Then $\class{W}$ is thick.
\end{lemma}

\begin{proof}
The assumptions imply that $\class{W}$ is already closed under retracts and extensions and is resolving. So it remains to show that if
  \begin{displaymath}
    \tag{\text{$*$}}
0 \xrightarrow{}   V  \xrightarrow{} W \xrightarrow{}  X \xrightarrow{} 0
  \end{displaymath}
  is an exact sequence with $V, W \in \mathcal{W}$
  then $X \in \mathcal{W}$ too. Applying $\Hom_{\cat{A}}(-,F)$ to
  ($*$), for any $F \in \mathcal{F}$, it follows that $\Ext^{\geqslant
    2}_{\cat{A}}(X,F)=0$. To see that $\Ext^1_{\cat{A}}(X,F)=0$ for every $F \in
  \mathcal{F}$, pick a short exact sequence $0 \to F' \to I \to F \to
  0$, where $I$ is injective and $F' \in \mathcal{F}$. Applying
  $\Hom_{\cat{A}}(X,-)$ to this sequence gives $\Ext^1_{\cat{A}}(X,F) \cong
  \Ext^2_{\cat{A}}(F',X)$, which is zero by what we just proved.
\end{proof}

\begin{proposition}[Characterizations of injective cotorsion pairs]\label{prop-characterizing injective cotorsion pairs}
Suppose $(\mathcal{W}, \class{F})$ is a complete cotorsion pair in an abelian category $\mathcal{A}$ with enough injectives. Then each of the following statements are equivalent:
\begin{enumerate}
\item $(\mathcal{W}, \class{F})$ is an injective cotorsion pair.
\item $(\mathcal{W}, \class{F})$ is hereditary and $\class{W} \cap \class{F}$ equals the class of injective objects.
\item $\mathcal{W}$ is thick and contains the injective objects.
\end{enumerate}
\end{proposition}

\begin{proof}
First, (1) implies (2) by definition and Becker's Lemma~\ref{lemma-Beckers lemma}. For (2) implies (3) note that for any $F \in \class{F}$ using that the cotorsion pair $(\mathcal{W}, \class{F})$ has enough projectives we can write $0 \xrightarrow{} F' \xrightarrow{} W \xrightarrow{} F \xrightarrow{} 0$ where $W \in \class{W}$ and $F' \in \class{F}$. Then since $\class{F}$ is closed under extensions we have $W \in \class{F}$. So by hypothesis we have $W$ is injective. Therefore (2) implies (3) follows from Lemma~\ref{lem-Henriks lemma}.

For (3) implies (1) first note that the hypothesis makes clear that the injective objects are in $\mathcal{W} \cap \class{F}$. So we just wish to show that everything in $\class{W} \cap \class{F}$ is injective. So suppose $X \in \class{W} \cap
\class{F}$. Then using that $\cat{A}$ has enough injectives find a short exact sequence  $0 \xrightarrow{} X \xrightarrow{} I \xrightarrow{} I/X \xrightarrow{} 0$ where $I$ is injective. By hypothesis $I \in \class{W}$ and also
$\class{W}$ is assumed to be thick, which means $I/X \in
\class{W}$. But now since $(\class{W},\class{F})$ is a cotorsion pair
the exact sequence splits. Therefore $X$ is a direct summand of
$I$, proving $X \in \class{I}$.
\end{proof}

By duality we have the following as well.

\begin{proposition}[Characterizations of projective cotorsion pairs]\label{prop-characterizing projective cotorsion pairs}
Suppose $(\mathcal{C}, \class{W})$ is a complete cotorsion pair in an abelian category $\mathcal{A}$ with enough projectives. Then each of the following statements are equivalent:
\begin{enumerate}
\item $(\mathcal{C}, \class{W})$ is a projective cotorsion pair.
\item $(\mathcal{C}, \class{W})$ is hereditary and $\class{C} \cap \class{W}$ equals the class of projective objects.
\item $\mathcal{W}$ is thick and contains the projective objects.
\end{enumerate}
\end{proposition}

For a general Grothendieck category $\cat{A}$ it becomes important to know whether or not the left side of a cotorsion pair contains a set of generators for $\cat{A}$. We point out the following property of injective cotorsion pairs which could be of use in this setting. It won't be used in this paper however.

\begin{proposition}\label{prop-finite injective and projective dimension objects are trivial}
Let $\cat{A}$ be an abelian category with enough injectives. If $(\class{W}, \class{F})$ is any injective cotorsion pair, then $\class{W}$ contains all objects of finite injective dimension. If $\class{A}$ also has enough projectives then $\class{W}$ also contains all objects of finite projective dimension.

We have the dual for projective cotorsion pairs $(\class{C},\class{W})$.

\end{proposition}

\begin{proof}
The point is that $\class{W}$ is \emph{thick} and contains all injective and projective objects. So let $M$ be of finite injective dimension. Then there exists an exact sequence $$0 \xrightarrow{} M \xrightarrow{} I^0 \xrightarrow{} I^1 \xrightarrow{} \cdots \xrightarrow{} I^n \xrightarrow{} 0$$ with each $I^n$ injective. Since each $I^i \in \class{W}$ and $\class{W}$ thick we conclude $M \in \class{I}$. If $\cat{A}$ has enough projectives, then the same argument works for projective dimension since projectives are always in the left side of a cotorsion pair.

\end{proof}

Note that Proposition~\ref{prop-finite injective and projective dimension objects are trivial} is saying that any object of finite injective dimension is trivial in the model structure associated to any injective cotorsion pair.

\subsection{The homotopy category of a weak injective model structure}
 It is convenient for us to explicitly define weak injective model structures for two reasons. First, the hereditary ones are precisely the right Bousfield localizations of two injective ones as we point out in Proposition~\ref{prop-Beckers theorem}. Second the homotopy relation on the subcategory of cofibrant-fibrant objects is characterized in a very nice way which we explain below. This will prove important for this paper. For example, it is automatic that for any weak injective (or weak projective) model structure on $\ch$, two maps from a cofibrant complex to a fibrant complex are formally homotopic if and only if they are chain homotopic in the usual sense.

First we point out that if $(\class{Q}, \class{W}, \class{R})$ is any Hovey triple in $\class{A}$ then the fully exact subcategory $\class{Q} \cap \class{R}$ is a Frobenius category with $\class{Q} \cap \class{W} \cap \class{R}$ being precisely the class of projective-injective objects. Therefore the stable category $\class{Q} \cap \class{R}/\sim$ is naturally a triangulated category. Since $\class{A}$ is abelian, it is a general fact that $\Ho{\class{A}}$ is also a triangulated category and is triangle equivalent to $\class{Q} \cap \class{R}/\sim$. See~\cite[Section~5]{gillespie-exact model structures} and~\cite[Chapters~6 and~7]{hovey-model-categories} for more on all this. We now go on to make this a little more precise while focusing only on injective and weak injective model structures.

Recall that a \emph{thick} subcategory of a triangulated category is a triangulated subcategory that is closed under retracts.

\begin{proposition}\label{prop-classification of homotopic maps in injective models}
Suppose $\class{M} = (\class{W},\class{F})$ is an injective cotorsion pair in $\class{A}$.
\begin{enumerate}
\item If $Y$ is fibrant then two maps $f,g \mathcolon X \xrightarrow{} Y$ in $\class{A}$ are homotopic, written $f \sim g$, if and only if $g-f$ factors through an injective object.
\item The fully exact subcategory $\class{F}$ of fibrant objects is a Frobenius category with its projective-injective objects being precisely the injective objects of $\class{A}$.
\item The inclusion functor $i : \class{F} \xrightarrow{} \cat{A}$ induces an inclusion of triangulated categories $\textnormal{Ho}\,i : \class{F}/\sim \ \xrightarrow{} \Ho{\cat{A}}$. It displays $\class{F}/\sim$ as an equivalent thick subcategory of $\Ho{\cat{A}}$.
\item The inverse of $\textnormal{Ho}\,i$ is the functor $\textnormal{Ho}\,R : \Ho{\cat{A}} \xrightarrow{} \class{F}/\sim$ and this is the fibrant replacement functor.
\end{enumerate}
\end{proposition}

\begin{proposition}\label{prop-classification of homotopic maps in weak injective models}
Suppose $\class{M} = (\class{Q},\class{W},\class{R})$ is an hereditary weak injective model structure in $\cat{A}$.
\begin{enumerate}
\item If $X$ is cofibrant and $Y$ is fibrant then two maps $f,g \mathcolon X \xrightarrow{} Y$ in $\class{A}$ are homotopic, written $f \sim g$, if and only if $g-f$ factors through an injective object.
\item The fully exact subcategory $\class{Q} \cap \class{R}$ of cofibrant-fibrant objects is a Frobenius category with its projective-injective objects being precisely the injective objects of $\class{A}$.
\item The inclusion functor $i : \class{Q} \cap \class{R} \xrightarrow{} \cat{A}$ induces an inclusion of triangulated categories $\textnormal{Ho}\,i : \class{Q} \cap \class{R} /\sim \ \xrightarrow{} \Ho{\cat{A}}$. It displays $\class{Q} \cap \class{R} /\sim$ as an equivalent thick subcategory of $\Ho{\cat{A}}$.

\item The inverse of $\textnormal{Ho}\,i$ is the functor $\textnormal{Ho}\,Q \circ \textnormal{Ho}\,R : \Ho{\cat{A}} \xrightarrow{} \class{Q} \cap \class{R} /\sim$ and this is fibrant replacement followed by cofibrant replacement.
\end{enumerate}
\end{proposition}

\begin{proof}
Both follow from general results that can be found in~\cite[Proposition~4.4 and Section~5]{gillespie-exact model structures} and Chapters~6 an~7 of~\cite{hovey-model-categories}. To show the triangulated subcategory $\class{Q} \cap \class{R}/\sim$ is closed under retracts use the fact that $\class{Q} \cap \class{R}$ contains the injective objects and is itself is closed under direct sums and retracts.

\end{proof}

\subsection{Becker's Theorem and a converse}\label{subsec-Beckers theorem and a converse}

We now look at a beautiful result from~\cite[Proposition~1.4.2]{becker} giving a simple description, in terms of Hovey triples, of the right Bousfield localization of an injective model structure with respect to another. In our statement here we add a uniqueness property and a converse. These last two things are not hard at all but they give a complete picture and are in fact useful for spotting localizations. In fact the author had already witnessed on more than one occasion model structures on $\ch$ arise in the following way. First, let $(\class{W}_1 , \class{F}_1)$ and $\class{M}_2 = (\class{W}_2 , \class{F}_2)$ be two injective model structures on $\cat{A}$. Suppose you also have a thick subcategory $\class{W}$ for which $\class{M} = (\class{W}_2 , \class{W} , \class{F}_1)$ is a Hovey triple. Then these three model structures are obviously linked in some way. It was not until the result of Becker that the author learned the formal connection. It turns out that $\class{M} = (\class{W}_2 , \class{W} , \class{F}_1)$ is the right Bousfield localization of $\class{M}_1$ with respect to the fibrant objects in $\class{M}_2$.

\begin{proposition}[Characterization of Becker Localizations]\label{prop-Beckers theorem}
Let $\cat{A}$ be an abelian category with enough injectives. Let $\class{M}_1 = (\class{W}_1 , \class{F}_1)$ and $\class{M}_2 = (\class{W}_2 , \class{F}_2)$ be two injective cotorsion pairs on $\cat{A}$ with $\class{F}_2 \subseteq \class{F}_1$. Then there exists a weak injective hereditary Hovey triple $\class{M}_1/\class{M}_2 = (\class{W}_2, \class{W}, \class{F}_1)$ on $\cat{A}$ where the thick class $\class{W}$ is
  \begin{align*}
   \class{W}  &= \{\, X \in \class{A} \, | \, \exists \, \text{s.e.s.} \, 0 \arr X \arr F_2 \arr W_1 \arr 0 \, \text{ with} \, F_2 \in \class{F}_2 \, , W_1 \in \class{W}_1 \,\} \\
           &= \{\, X \in \class{A} \, | \, \exists \, \text{s.e.s.} \, 0 \arr F_2 \arr W_1 \arr X \arr 0 \, \text{ with} \,  F_2 \in \class{F}_2 \, , W_1 \in \class{W}_1 \,\} \\
          \end{align*} We call $\class{M}_1/\class{M}_2$ the \emph{right localization of $\class{M}_1$ with respect to $\class{M}_2$} and we note the following uniqueness and converse:
\begin{enumerate}
\item~\textnormal{(Uniqueness of Trivial Objects)}  Suppose $\class{V}$ is a thick subcategory for which $\class{M} = (\class{W}_2 , \class{V} , \class{F}_1)$ is a Hovey triple. Then $\class{M} = \class{M}_1/\class{M}_2$.

\item~\textnormal{(Converse)}  Let $\class{M} = (\class{Q},\class{W},\class{R})$ be a weak injective hereditary Hovey triple in $\cat{A}$. Then setting $\class{M}_1 = (\class{Q} \cap \class{W},\class{R})$ and $\class{M}_2 = (\class{Q},\class{W} \cap \class{R})$, these are each injective cotorsion pairs and $\class{M} = \class{M}_1/\class{M}_2$.

\end{enumerate}
\end{proposition}

\begin{remark}
Becker shows that his right localization $\class{M}_1/\class{M}_2$ is in fact the right Bousfield localization with respect to the maps $0 \to F$ where $F \in \class{F}_2$.
\end{remark}

\begin{proof}
This is Becker's result. See~\cite[Proposition~1.4.2]{becker}. We are simply noting the uniqueness and converse statements.

The uniqueness of the class of trivial objects is a corollary to Proposition~\ref{prop-characterization of trivial objects}. To prove the converse statement, let $\class{M} = (\class{Q},\class{W},\class{R})$ be any weak injective and hereditary Hovey triple in $\cat{A}$. By definition of $\class{M} = (\class{Q},\class{W},\class{R})$ being hereditary we have that each of the cotorsion pairs $(\class{Q} \cap \class{W},\class{R})$ and $(\class{Q},\class{W} \cap \class{R})$ are also hereditary. Setting $\class{M}_1 = (\class{W}_1 , \class{F}_1) = (\class{Q} \cap \class{W},\class{R})$
and $\class{M}_2 = (\class{W}_2 , \class{F}_2) =  (\class{Q},\class{W} \cap \class{R})$ we get that $\class{M}_1$ and $\class{M}_2$ are each injective cotorsion pairs by condition~(2) of Proposition~\ref{prop-characterizing injective cotorsion pairs}. It now follows that $\class{M} = \class{M}_1/\class{M}_2$ by the uniqueness property just mentioned above.

\end{proof}

\begin{example}
It is easy to see that $\class{M} = (\class{W},\class{F})$ is an injective model structure if and only if it is the right localization of itself by the trivial model structure induced by the categorical injective cotorsion pair $(\class{A},\class{I})$.

\end{example}

\begin{example}\label{example-degreewise model structures}
Let $\cat{A} = \ch$. Let $\class{M}_1 = (\class{W}_1, \class{F}_1)$ be the Inj model structure from~\cite{bravo-gillespie-hovey} which has as the fibrant complexes $\class{F}_1 = \dwclass{I}$, the class of all complexes of injective $R$-modules.  Also let $\class{M}_2 = (\class{W}_2, \class{F}_2)$ be the exact injective model structure from~\cite{bravo-gillespie-hovey} which has as the fibrant complexes $\class{F}_2 = \exclass{I}$, the class of all exact complexes of injective $R$-modules. Denote the class of exact complexes by $\class{E}$. Then it follows from Theorem~4.7 of~\cite{gillespie-degreewise-model-strucs} that there is a hereditary weak injective Hovey triple $\class{M} = (\class{W}_2, \class{E}, \class{F}_1)$. It follows that $\class{M} = \class{M}_1/\class{M}_2$. Its homotopy category is $\class{D}(R)$ because the trivial complexes are precisely the exact complexes (and so it follows from Lemma~\ref{lemma-characterization of weak equivalences in exact categories} that the homology isomorphisms are the weak equivalences). This example is relevant to Becker's approach to recovering Krause's recollement. We see Gorenstein injective analogs in Section~\ref{sec-Gorenstein projective and injective models for the derived category}.

\end{example}

\begin{remark}
We have worked with injective cotorsion pairs and their right localizations in this section. We note that the dual statements concerning left localizations of projective cotorsion pairs also hold. We have omitted the projective statements in this section.

\end{remark}

\section{localization sequences and recollements from cotorsion pairs}\label{sec-localization sequences and recollements from cotorsion pairs}

The goal here is to describe how recollements are related to injective cotorsion pairs. This is Theorem~\ref{them-finding recollements} and its converse in Corollary~\ref{cor-bijective correspondences between admissible subcats and injective cot pairs}. Since injective cotorsion pairs are particular types of model structures the functors involved in the recollements are actually derived functors between simple Quillen adjunctions as we will describe. The method is a generalization of Becker's approach from~\cite{becker} where he obtained Krause's recollement from~\cite{krause-stable derived cat of a Noetherian scheme}, for a general ring $R$, using the theory of model categories.

Recall that the homotopy category of an abelian model category is always a triangulated category and has a set of weak generators whenever the model structure is cofibrantly generated. See Hovey~\cite[Section~7]{hovey} for more details. We start with the definition of a recollement. Loosely, a recollement is an ``attachment'' of two triangulated categories. The standard reference is~\cite{BBD-perverse sheaves}. We give the definition that appeared in~\cite{krause-stable derived cat of a Noetherian scheme} based on localization and colocalization sequences.

\begin{definition}\label{def-localization sequence}
Let $\class{T}' \xrightarrow{F} \class{T} \xrightarrow{G} \class{T}''$ be a sequence of exact functors between triangulated categories. We say it is a \emph{localization sequence} when there exists right adjoints $F_{\rho}$ and $G_{\rho}$ giving a diagram of functors as below with the listed properties.
$$\begin{tikzcd}
\class{T}'
\rar[to-,
to path={
([yshift=0.5ex]\tikztotarget.west) --
([yshift=0.5ex]\tikztostart.east) \tikztonodes}][swap]{F}
\rar[to-,
to path={
([yshift=-0.5ex]\tikztostart.east) --
([yshift=-0.5ex]\tikztotarget.west) \tikztonodes}][swap]{F_{\rho}}
& \class{T}
\rar[to-,
to path={
([yshift=0.5ex]\tikztotarget.west) --
([yshift=0.5ex]\tikztostart.east) \tikztonodes}][swap]{G}
\rar[to-,
to path={
([yshift=-0.5ex]\tikztostart.east) --
([yshift=-0.5ex]\tikztotarget.west) \tikztonodes}][swap]{G_{\rho}}
& \class{T}'' \\
\end{tikzcd}$$
\begin{enumerate}
\item The right adjoint $F_{\rho}$ of $F$ satisfies $F_{\rho} \circ F \cong \text{id}_{\class{T}'}$.
\item The right adjoint $G_{\rho}$ of $G$ satisfies $G \circ G_{\rho} \cong \text{id}_{\class{T}''}$.
\item For any object $X \in \class{T}$, we have $GX = 0$ iff $X \cong FX'$ for some $X' \in \class{T}'$.
\end{enumerate}
A \emph{colocalization sequence} is the dual. That is, there must exist left adjoints $F_{\lambda}$ and $G_{\lambda}$ with the analogous properties.
\end{definition}

It is fair to say that a localization sequence is a sequence of left adjoints which in some sense ``splits'' at the level of triangulated categories. See~\cite[Section~3]{krause-localization theory for triangulated categories} for the first properties of localization sequences which reflect this statement. Similarly, a colocalization sequence is a sequence of right adjoints with this property. It is true that if  $\class{T}' \xrightarrow{F} \class{T} \xrightarrow{G} \class{T}''$ is a localization sequence then  $\class{T}'' \xrightarrow{G_{\rho}} \class{T} \xrightarrow{F_{\rho}} \class{T}'$ is a colocalization sequence and if  $\class{T}' \xrightarrow{F} \class{T} \xrightarrow{G} \class{T}''$ is a colocalization sequence then  $\class{T}'' \xrightarrow{G_{\lambda}} \class{T} \xrightarrow{F_{\lambda}} \class{T}'$ is a localization sequence. This brings us to the definition of a recollement where the sequence of functors  $\class{T}' \xrightarrow{F} \class{T} \xrightarrow{G} \class{T}''$ is both a localization sequence and a colocalization sequence.

\begin{definition}\label{def-recollement}
Let $\class{T}' \xrightarrow{F} \class{T} \xrightarrow{G} \class{T}''$ be a sequence of exact functors between triangulated categories. We say $\class{T}' \xrightarrow{F} \class{T} \xrightarrow{G} \class{T}''$ induces a \emph{recollement} if it is both a localization sequence and a colocalization sequence as shown in the picture.
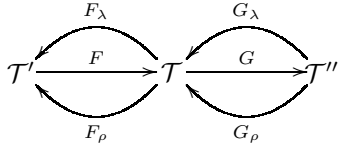
\[
\xy
(-20,0)*+{\class{T}'};
(0,0)*+{\class{T}};
(20,0)*+{\class{T}''};
{(-18,0) \ar^{F} (-2,0)};
{(-2,0) \ar@/^1pc/@<0.5em>^{F_{\rho}} (-18,0)};
{(-2,0) \ar@/_1pc/@<-0.5em>_{F_{\lambda}} (-18,0)};
{(2,0) \ar^{G} (18,0)};
{(18,0) \ar@/^1pc/@<0.5em>^{G_{\rho}} (2,0)};
{(18,0) \ar@/_1pc/@<-0.5em>_{G_{\lambda}} (2,0)};
\endxy
\]
\end{definition}

So the idea will be to ``glue'' a colocalization sequence to a localization sequence to get the diagram of functors in the recollement diagram.

\subsection{Colocalization sequences from weak injective model structures}\label{subsec-colocalization sequences} Lets look again at the setup to Proposition~\ref{prop-Beckers theorem}. We are working in an abelian category $\cat{A}$ with enough injectives. We have two injective cotorsion pairs $\class{M}_1 = (\class{W}_1,\class{F}_1)$ and $\class{M}_2 = (\class{W}_2, \class{F}_2)$ with $\class{F}_2 \subseteq \class{F}_1$. Since the identity functor is exact it follows immediately that $\class{M}_2 \xrightarrow{\text{id}} \class{M}_1 \xrightarrow{\text{id}} \class{M}_1/\class{M}_2$ are right Quillen functors.  That is, we have Quillen adjunctions  $\class{M}_1/\class{M}_2 \rightleftarrows \class{M}_1 \rightleftarrows \class{M}_2$ consisting entirely of identity functors. The following comes from~\cite[Corollary~1.4.5]{becker}.

\begin{proposition}\label{prop-colocalizations sequences from injective cot pairs}
Let $\cat{A}$ be an abelian category with enough injectives and let $\class{M}_1 = (\class{W}_1 , \class{F}_1)$ and $\class{M}_2 = (\class{W}_2 , \class{F}_2)$ be two injective cotorsion pairs on $\cat{A}$ with $\class{F}_2 \subseteq \class{F}_1$. Then the identity Quillen adjunctions  $\class{M}_1/\class{M}_2 \rightleftarrows \class{M}_1 \rightleftarrows \class{M}_2$ descend to a colocalization sequence $\Ho{\class{M}_2} \xrightarrow{\textnormal{R}\,\textnormal{id}} \Ho{\class{M}_1} \xrightarrow{\textnormal{R}\,\textnormal{id}} \Ho{\class{M}_1/\class{M}_2}$ with left adjoints $\textnormal{L}\,\textnormal{id}$. In particular, on the level of the full subcategory of cofibrant-fibrant subobjects we have the colocalization sequence:
\[
\begin{tikzpicture}[node distance=2.75 cm, auto]
\node (A)  {$\class{F}_2/\sim$};
\node (B) [right of=A] {$\class{F}_1/\sim$};
\node (C) [right of=B] {$(\class{F}_1 \cap \class{W}_2) /\sim$};

%
%
\draw[<-] (A.10) to node {E$(\class{M}_2)$} (B.170);
\draw[->] (A.350) to node [swap] {$I$} (B.190);
\draw[<-] (B.10) to node {$I$} (C.175);
\draw[->] (B.350) to node [swap] {C$(\class{M}_2)$} (C.185);
\end{tikzpicture}
\]
Here the functors $I$ are each inclusion, the functor E$(\class{M}_2)$ represents using enough injectives of the cotorsion pair $\class{M}_2 = (\class{W}_2 , \class{F}_2)$ (to get what would often be called a \emph{special $\class{F}_2$-preenvelope}), and the functor C$(\class{M}_2)$  represents using enough projectives of the cotorsion pair $\class{M}_2 = (\class{W}_2 , \class{F}_2)$ (to get what would often be called a \emph{special-$\class{W}_2$ precover}).

\end{proposition}

\begin{proof}
See~\cite[Corollary~1.4.5]{becker} to see that we have the colocalization sequence $\Ho{\class{M}_2} \xrightarrow{\textnormal{R}\,\textnormal{id}} \Ho{\class{M}_1} \xrightarrow{\textnormal{R}\,\textnormal{id}} \Ho{\class{M}_1/\class{M}_2}$. In general, the right derived functor is defined on objects by first taking a fibrant replacement and than applying the functor, here the identity. Similarly, the left derived functor is defined by first taking a cofibrant replacement and then applying the functor. In any abelian model category $\class{M} = (\class{Q}, \class{W}, \class{R})$, taking cofibrant replacements corresponds to using enough projectives of the cotorsion pair $(\class{Q},\class{W} \cap \class{R})$ and taking fibrant replacements corresponds to using enough injectives of the cotorsion pair $(\class{Q} \cap \class{W}, \class{R})$. Also as in Propositions~\ref{prop-classification of homotopic maps in injective models} and~\ref{prop-classification of homotopic maps in weak injective models} one uses cofibrant-fibrant replacement and inclusion when translating between the homotopy categories of $\class{M}$ and the full subcategory of cofibrant-fibrant objects. We conclude that the functors work as stated.
\end{proof}

\subsection{Finding a recollement from three cotorsion pairs}\label{subsec-finding a recollement from three cot pairs}

With the same setup as in Subsection~\ref{subsec-colocalization sequences}, now suppose we have three injective model structures $\class{M}_1 = (\class{W}_1, \class{F}_1)$ , $\class{M}_2 = (\class{W}_2, \class{F}_2)$ , and $\class{M}_3 = (\class{W}_3, \class{F}_3)$ having $\class{F}_2 , \class{F}_3 \subseteq  \class{F}_1$. Then for $i = 1,2$ we have from Proposition~\ref{prop-colocalizations sequences from injective cot pairs} that the identity functors give Quillen adjunctions  $\class{M}_1/\class{M}_i \rightleftarrows \class{M}_1 \rightleftarrows \class{M}_i$ descending to give two colocalization sequences
\[
\tag{\text{$*$}}
\Ho{\class{M}_2} \xrightarrow{\textnormal{R}\,\textnormal{id}} \Ho{\class{M}_1} \xrightarrow{\textnormal{R}\,\textnormal{id}} \Ho{\class{M}_1/\class{M}_2}\]
 and
\[
\tag{\text{$**$}}
\Ho{\class{M}_3} \xrightarrow{\textnormal{R}\,\textnormal{id}} \Ho{\class{M}_1} \xrightarrow{\textnormal{R}\,\textnormal{id}} \Ho{\class{M}_1/\class{M}_3}.\]
However, we also have the following lemma.

\begin{lemma}
The identity adjunctions shown below are in fact Quillen adjunction:
$$\class{M}_3 \rightleftarrows \class{M}_3 \ , \ \ \class{M}_1/\class{M}_2 \rightleftarrows \class{M}_3 \ , \ \ \class{M}_1/\class{M}_3 \rightleftarrows \class{M}_2$$
\end{lemma}

\begin{proof}
The first is obvious and the second two are symmetric, so we show that $\class{M}_1/\class{M}_2 \rightleftarrows \class{M}_3$ is a Quillen adjunction. That is, the identity from $\class{M}_3$ to $\class{M}_1/\class{M}_2$ is right Quillen. For this just recall $\class{M}_1/\class{M}_2 = (\class{W}_2,\class{W},\class{F}_1)$. Since the identity functor is exact it just boils down to noting that we are given $\class{F}_3 \subseteq \class{F}_1$ and that the injective objects are contained in $\class{W} \cap \class{F}_1 = \class{F}_2$.

\end{proof}

\

So reversing the direction of $(**)$ above, and using the above lemma we get the following diagram of functors on the level of homotopy categories.

\[
\tag{\text{$***$}}
\begin{tikzpicture}[node distance=3.5 cm, auto]
\node (A)  {$\textnormal{Ho}(\class{M}_2)$};
\node (D) [below of=A] {$\textnormal{Ho}(\class{M}_1/\class{M}_3)$};
\node (B) [right of=A] {$\textnormal{Ho}(\class{M}_1)$};
\node (C) [right of=B] {$\textnormal{Ho}(\class{M}_1/\class{M}_2)$};
\node (E) [right of=D] {$\textnormal{Ho}(\class{M}_1)$};
\node (F) [right of=E] {$\textnormal{Ho}(\class{M}_3)$};

%
%
\draw[<-] (A.10) to node {\textnormal{L}\,\textnormal{id}} (B.170);
\draw[->] (A.350) to node [swap] {\textnormal{R}\,\textnormal{id}} (B.190);
\draw[->] (D.6) to node {\textnormal{L}\,\textnormal{id}} (E.170);
\draw[<-] (D.353) to node [swap] {\textnormal{R}\,\textnormal{id}} (E.190);
\draw[->] (E.10) to node {\textnormal{L}\,\textnormal{id}} (F.170);
\draw[<-] (E.350) to node [swap] {\textnormal{R}\,\textnormal{id}} (F.190);
\draw[<-] (B.10) to node {\textnormal{L}\,\textnormal{id}} (C.173);
\draw[->] (B.350) to node [swap] {\textnormal{R}\,\textnormal{id}} (C.187);
%
%
\draw[->] (B.290) to node {\textnormal{R}\,\textnormal{id}} (E.70);
\draw[<-] (B.250) to node [swap] {\textnormal{L}\,\textnormal{id}} (E.110);
\draw[->] (A.290) to node {\textnormal{R}\,\textnormal{id}} (D.70);
\draw[<-] (A.250) to node [swap] {\textnormal{L}\,\textnormal{id}} (D.110);
\draw[<-] (C.290) to node {\textnormal{R}\,\textnormal{id}} (F.70);
\draw[->] (C.250) to node [swap] {\textnormal{L}\,\textnormal{id}} (F.110);
\end{tikzpicture}
\]

This is a general form of Becker's butterfly diagram~\cite[page 28]{becker}. We now make a few notes about the diagram before proceeding.
\begin{itemize}
\item We write all left adjoints on the top or on the left and we write all right adjoints on the bottom or right.
\item The top row is a colocalization sequence.
\item The bottom row is a localization sequence.
\item The butterfly is pictured by identifying the two occurrences of $\Ho{\class{M}_1}$ along the the middle vertical arrows. The middle vertical arrows are not literally the identity but these functors are canonical equivalences of $\Ho{\class{M}_1}$. It is the identity however if you restrict to the full subcategory of cofibrant-fibrant objects.
\end{itemize}

Note that the condition $\class{F}_2, \class{F}_3 \subseteq \class{F}_1$ is all that is required for the butterfly diagram setup to exist. We now give our main construction theorem, Theorem~\ref{them-finding recollements}, giving simple criteria for an induced recollement situation. The proof will make use of the following proposition which relies on the uniqueness condition of Proposition~\ref{prop-Beckers theorem}. It can be useful on its own for spotting localizations of two injective model structures.

\begin{proposition}\label{prop-spotting localizations}
Let $\class{W}$ be a thick class in $\cat{A}$. Suppose $\class{M}_1 = (\class{W}_1 , \class{F}_1)$ and $\class{M}_2 = (\class{W}_2 , \class{F}_2)$ are injective model structures.
\begin{enumerate}
\item If $\class{W} \cap \class{F}_1 = \class{F}_2$ and $\class{W}_1 \subseteq \class{W}$. Then $\class{M}_1/\class{M}_2 = (\class{W}_2 , \class{W} , \class{F}_1)$.
\item If $\class{W}_2 \cap \class{W} = \class{W}_1$ and $\class{F}_2 \subseteq \class{W}$. Then $\class{M}_1/\class{M}_2 = (\class{W}_2 , \class{W} , \class{F}_1)$.

\end{enumerate}

\end{proposition}

\begin{proof}
First, note that $\class{F}_2 \subseteq \class{F}_1$ by the given. So taking left perps we get $\class{W}_1 \subseteq \class{W}_2$. Also we are given that $\class{W}_1 \subseteq \class{W}$, and so $\class{W}_1 \subseteq \class{W}_2 \cap \class{W}$ is automatic.

On the other hand, say $X \in \class{W}_2
\cap \class{W}$. Since $(\class{W}_1,\class{F}_1)$ is complete we
can find a short exact sequence $0 \xrightarrow{} B \xrightarrow{} A \xrightarrow{} X \xrightarrow{} 0$ where
$B \in \class{F}_1$ and $A \in \class{W}_1$. Since $A$ and $X$ are
in the thick class $\class{W}$ we get that $B$ is too. So $B \in \class{F}_1 \cap \class{W} =
\class{F}_2$. Since $(\class{W}_2,\class{F}_2)$ is a cotorsion
pair the sequence $0 \xrightarrow{} B \xrightarrow{} A \xrightarrow{} X \xrightarrow{} 0$ must split, making
$X$ a direct summand of $A$. But then $X$ must belong to
$\class{W}_1$ since the left side of a cotorsion pair is always
closed under retracts. This shows $\class{W}_1
= \class{W}_2 \cap \class{W}$ and it follows that $(\class{W}_2 , \class{W} , \class{F}_1)$ is an injective Hovey triple and proves (1) because of the uniqueness part of Proposition~\ref{prop-Beckers theorem}. The proof of~(2) is similar. The assumptions imply $\class{F}_2 \subseteq \class{W} \cap \class{F}_1$. On the other hand, for $X \in \class{F}_1 \cap \class{W}$ use completeness of $(\class{W}_2 , \class{F}_2)$ to find a s.e.s. $0 \xrightarrow{} X \xrightarrow{} F_2 \xrightarrow{} W_2 \xrightarrow{} 0$. But given that $F_2 \in \class{W}$, thickness of $\class{W}$ implies $W_2 \in \class{W} \cap \class{W}_2 = \class{W}_1$. So the sequence splits.

\end{proof}

\begin{theorem}\label{them-finding recollements}
Let $\cat{A}$ be an abelian category with enough injectives and suppose we have three injective cotorsion pairs $$\class{M}_1 = (\class{W}_1, \class{F}_1) , \ \ \ \class{M}_2 = (\class{W}_2, \class{F}_2) , \ \ \ \class{M}_3 = (\class{W}_3, \class{F}_3)$$ such that $\class{F}_2 , \class{F}_3 \subseteq  \class{F}_1$. Then if $\class{W}_3 \cap \class{F}_1 = \class{F}_2$ (or equivalently, $\class{W}_2 \cap \class{W}_3 = \class{W}_1$ and $\class{F}_2 \subseteq \class{W}_3$), then $\class{M}_1/\class{M}_2$ is Quillen equivalent to $\class{M}_3$ and $\class{M}_1/\class{M}_3$ is Quillen equivalent to $\class{M}_2$. In fact, there is a recollement
\[
\begin{tikzpicture}[node distance=3.5cm]
\node (A) {$\mathcal{F}_2/\sim$};
\node (B) [right of=A] {$\mathcal{F}_1/\sim$};
\node (C) [right of=B] {$\mathcal{F}_3/\sim$};
\draw[<-,bend left=40] (A.20) to node[above]{\small E$(\class{W}_2, \class{F}_2)$} (B.160);
\draw[->] (A) to node[above]{\small $I$} (B);
\draw[<-,bend right=40] (A.340) to node [below]{\small C$(\class{W}_3, \class{F}_3)$} (B.200);
\draw[<-,bend left] (B.20) to node[above]{\small $\lambda$=C$(\class{W}_2, \class{F}_2)$} (C.160);
\draw[->] (B) to node[above]{\tiny E$(\class{W}_3, \class{F}_3)$} (C);
\draw[<-,bend right] (B.340) to node [below]{\small $I$} (C.200);
\end{tikzpicture}
\]
Here, the notation such as E$(\class{W}_3, \class{F}_3)$ means to take a special $\class{F}_3$-preenvelope by using enough injectives of the cotorsion pair $(\class{W}_3, \class{F}_3)$. Similarly the notation C$(\class{W}_3,\class{F}_3)$ means to take a special $\class{W}_3$-precover. Moreover the left adjoint $\lambda$ has essential image $(\class{W}_2 \cap \class{F}_1)/\sim$ yielding an equivalence $(\class{W}_2 \cap \class{F}_1)/\sim \ \cong \ \class{F}_3/\sim$.
\end{theorem}

\begin{proof}
First we show that the two conditions are equivalent. That is, $\class{W}_3 \cap \class{F}_1 = \class{F}_2$ if and only if $\class{W}_2 \cap \class{W}_3 = \class{W}_1$ and $\class{F}_2 \subseteq \class{W}_3$. For the ``only if'' part, the only part that is not clear is $\class{W}_2 \cap \class{W}_3 \subseteq \class{W}_1$. So assume $W \in \class{W}_2 \cap \class{W}_3$. Use enough projectives of $(\class{W}_1, \class{F}_1)$ to find a short exact sequence $0 \xrightarrow{} F_1 \xrightarrow{} W_1 \xrightarrow{} W \xrightarrow{} 0$ with $F_1 \in \class{F}_1$, and $W_1 \in \class{W}_1 \subseteq \class{W}_3$. Since $\class{W}_3$ is thick we see that $F_1 \in \class{W}_3 \cap \class{F}_1 = \class{F}_2$. So with $F_1 \in \class{F}_2$ and $W \in \class{W}_2$, the short exact sequence must split, making $W$ a retract of $W_1$. Hence $W$ must be in $\class{W}_1$. For the ``if'' part, the analogous part to show is $\class{W}_3 \cap \class{F}_1 \subseteq \class{F}_2$. This follows by a similar argument which starts by finding a short exact sequence $0 \xrightarrow{} W \xrightarrow{} F_2 \xrightarrow{} W_2 \xrightarrow{} 0$ with $F_2 \in \class{F}_2$, and $W_2 \in \class{W}_2$.

Having shown the two conditions are equivalent we point out that if either holds, Proposition~\ref{prop-spotting localizations} applies (with $\class{W}_3$ in place of $\class{W}$) to conclude $\class{M}_1/\class{M}_2 = (\class{W}_2, \class{W}_3, \class{F}_1)$. We now consider again the butterfly diagram $(***)$ above consisting of all adjoint pairs.

First, note that $\class{M}_1/\class{M}_2$ and $\class{M}_3$ have the same trivial objects. It follows then from Lemma~\ref{lemma-characterization of weak equivalences in exact categories} that the weak equivalences in $\class{M}_1/\class{M}_2$ coincide with those in $\class{M}_3$. So by definition, the identity adjunction $\class{M}_1/\class{M}_2 \rightleftarrows \class{M}_3$ is a Quillen equivalence. This means the vertical maps on the far right of $(***)$ are equivalences. As we noted above the middle vertical maps in $(***)$ are already canonical equivalences.

Second, note that the full subcategory of $\class{M}_1/\class{M}_3 = (\class{W}_3, \class{V}, \class{F}_1)$ consisting of the cofibrant-fibrant subobjects is $\class{W}_3 \cap \class{F}_1 = \class{F}_2$. So we have $\Ho{\class{M}_1/\class{M}_3} \cong (\class{W}_3 \cap \class{F}_1)/\sim \ = \class{F}_2/\sim$ through the canonical equivalence. Thus on the level of the the homotopy category associated to the full subcategory of cofibrant-fibrant subobjects, the diagram $(***)$ becomes the following:

\[
\begin{tikzpicture}[node distance=3.5 cm, auto]
\node (A)  {$\class{F}_2/\sim$};
\node (D) [below of=A] {$\class{F}_2/\sim$};
\node (B) [right of=A] {$\class{F}_1/\sim$};
\node (C) [right of=B] {$(\class{W}_2 \cap \class{F}_1)/\sim$};
\node (E) [right of=D] {$\class{F}_1/\sim$};
\node (F) [right of=E] {$\class{F}_3/\sim$};

%
%
\draw[<-] (A.17) to node {E$(\class{W}_2,\class{F}_2)$} (B.163);
\draw[->] (A.343) to node [swap] {Inclusion} (B.197);
\draw[->] (D.17) to node {Inclusion} (E.163);
\draw[<-] (D.343) to node [swap] {C$(\class{W}_3,\class{F}_3)$} (E.197);
\draw[->] (E.17) to node {E$(\class{W}_3,\class{F}_3)$} (F.163);
\draw[<-] (E.343) to node [swap] {Inclusion} (F.197);
\draw[<-] (B.17) to node {Inclusion} (C.171);
\draw[->] (B.343) to node [swap] {C$(\class{W}_2,\class{F}_2)$} (C.189);
%
%
\draw[->] (B.290) to node {\textnormal{id}} (E.70);
\draw[<-] (B.250) to node [swap] {\textnormal{id}} (E.110);
\draw[->] (A.290) to node {\textnormal{id}} (D.70);
\draw[<-] (A.250) to node [swap] {\textnormal{id}} (D.110);
\draw[<-] (C.290) to node {C$(\class{W}_2,\class{F}_2)$} (F.70);
\draw[->] (C.250) to node [swap] {E$(\class{W}_3,\class{F}_3)$} (F.110);
\end{tikzpicture}
\]
In particular, the left hand square perfectly ``glues'' together along the inclusion $I : \class{F}_2/\sim \ \xrightarrow{} \class{F}_1/\sim$ showing $I$ also has a left adjoint. It follows formally from ~\cite[Proposition~4.13.1]{krause-localization theory for triangulated categories} that the localization sequence which makes up the bottom row induces the recollement. However, we claim $\lambda = \text{C}(\class{W}_2,\class{F}_2)$. To see this, we point out that while the right square doesn't commute on the nose, it is actually ``glued'' together by an isomorphism of functors. Indeed one can argue that there is a natural isomorphism from the composite functor $\class{F}_1/\sim \ \xrightarrow{\text{E}(\class{W}_3,\class{F}_3)} \class{F}_3/\sim \ \xrightarrow{\text{C}(\class{W}_2,\class{F}_2)} (\class{W}_2 \cap \class{F}_1)/\sim$ to the functor $\class{F}_1/\sim \ \xrightarrow{\text{C}(\class{W}_2,\class{F}_2)} (\class{W}_2 \cap \class{F}_1)/\sim$. From this it follows rather easily that the composite $ \class{F}_3/\sim \ \xrightarrow{\text{C}(\class{W}_2,\class{F}_2)} (\class{W}_2 \cap \class{F}_1)/\sim \ \hookrightarrow \class{F}_1/\sim$ is indeed the left adjoint to $\class{F}_1/\sim \ \xrightarrow{\text{E}(\class{W}_3,\class{F}_3)} \class{F}_3/\sim$. That is, $\lambda = \text{C}(\class{W}_2,\class{F}_2)$.
\end{proof}

\begin{remark}
In the proof of Theorem~\ref{them-finding recollements} above, we described how to ``glue'' together the right square to get the recollement. But there is an alternate yet equally valid way to glue together the right square. Indeed there is also a natural isomorphism from the composite functor $\class{F}_1/\sim \ \xrightarrow{\text{C}(\class{W}_2,\class{F}_2)} (\class{W}_2 \cap \class{F}_1)/\sim \ \xrightarrow{\text{E}(\class{W}_3,\class{F}_3)} \class{F}_3/\sim$ to the functor $\class{F}_1/\sim \ \xrightarrow{\text{E}(\class{W}_3,\class{F}_3)} \class{F}_3/\sim$. From this we deduce that the composite $(\class{W}_2 \cap \class{F}_1)/\sim \ \xrightarrow{\text{E}(\class{W}_3,\class{F}_3)} \class{F}_3/\sim \ \hookrightarrow \class{F}_1/\sim$ is right adjoint to $\class{F}_1/\sim \ \xrightarrow{\text{C}(\class{W}_2,\class{F}_2)} (\class{W}_2 \cap \class{F}_1)/\sim$. That is, we can restate Theorem~\ref{them-finding recollements} with a new recollement diagram emphasizing the quotient functor $\class{F}_1/\sim \ \xrightarrow{\text{C}(\class{W}_2,\class{F}_2)} (\class{W}_2 \cap \class{F}_1)/\sim$ with the inclusion as its left adjoint and with $\rho = \text{E}(\class{W}_3,\class{F}_3)$ as its right adjoint having essential image $\class{F}_3/\sim$. We will display all recollement diagrams in this paper however in the style appearing in the statement of Theorem~\ref{them-finding recollements}.
\end{remark}

We now record the projective version of the theorem.

\begin{theorem}\label{them-finding projective recollements}
Let $\cat{A}$ be an abelian category with enough projectives and suppose we have three projective cotorsion pairs $$\class{M}_1 = (\class{C}_1, \class{W}_1) , \ \ \ \class{M}_2 = (\class{C}_2, \class{W}_2) , \ \ \ \class{M}_3 = (\class{C}_3, \class{W}_3)$$ such that $\class{C}_2 , \class{C}_3 \subseteq  \class{C}_1$. Then if $\class{W}_3 \cap \class{C}_1 = \class{C}_2$ (or equivalently, $\class{W}_2 \cap \class{W}_3 = \class{W}_1$ and $\class{C}_2 \subseteq \class{W}_3$), then the \emph{left} localization $\class{M}_2\backslash\class{M}_1$ is Quillen equivalent to $\class{M}_3$ and $\class{M}_3\backslash\class{M}_1$ is Quillen equivalent to $\class{M}_2$. In fact, there is a recollement
\[
\begin{tikzpicture}[node distance=3.5cm]
\node (A) {$\mathcal{C}_2/\sim$};
\node (B) [right of=A] {$\mathcal{C}_1/\sim$};
\node (C) [right of=B] {$\mathcal{C}_3/\sim$};
\draw[<-,bend left=40] (A.20) to node[above]{\small E$(\class{C}_3, \class{W}_3)$} (B.160);
\draw[->] (A) to node[above]{\small $I$} (B);
\draw[<-,bend right=40] (A.340) to node [below]{\small C$(\class{C}_2, \class{W}_2)$} (B.200);
\draw[<-,bend left] (B.20) to node[above]{\small $I$} (C.160);
\draw[->] (B) to node[above]{\tiny C$(\class{C}_3, \class{W}_3)$} (C);
\draw[<-,bend right] (B.340) to node [below]{\small $\rho$=E$(\class{C}_2, \class{W}_2)$} (C.200);
\end{tikzpicture}
\]
Here, the notation such as E$(\class{C}_3, \class{W}_3)$ means to take a special $\class{W}_3$-preenvelope by using enough injectives of the cotorsion pair $(\class{C}_3, \class{W}_3)$. Similarly the notation C$(\class{C}_3,\class{W}_3)$ means to take a special $\class{C}_3$-precover. Moreover the right adjoint $\rho$ has essential image $(\class{C}_1 \cap \class{W}_2)/\sim$ yielding an equivalence $(\class{C}_1 \cap \class{W}_2)/\sim \ \cong \ \class{C}_3/\sim$.
\end{theorem}

\begin{example}\label{example-becker gets Krause}
In terms of Theorem~\ref{them-finding recollements}, Becker's method of recovering Krause's recollement boils down to the following observations. As in Example~\ref{example-degreewise model structures}, let $\dwclass{I}$ be the class of all complexes of injective $R$-modules, $\exclass{I}$ be the class of all exact complexes of injective $R$-modules and $\dgclass{I}$ be the class of all DG-injective complexes. Then we have the following injective cotorsion pairs: $$\class{M}_1 = (\class{W}_1, \dwclass{I}) , \ \ \ \class{M}_2 = (\class{W}_2, \exclass{I}) , \ \ \ \class{M}_3 = (\class{E}, \dgclass{I}).$$ It is easy to see that $\class{F}_2 , \class{F}_3 \subseteq  \class{F}_1$ and $\class{E} \cap \dwclass{I} = \exclass{I}$. So Theorem~\ref{them-finding recollements} gives a recollement which is Krause's recollement from~\cite{krause-stable derived cat of a Noetherian scheme}. See also the introduction to Section~\ref{sec-Gorenstein projective and injective models for the derived category}.

In an earlier version of this paper the author claimed without proof that this example extended to recover Krause's recollement for quasi-coherent sheaves over any scheme $X$. This is false and a counterexample has been constructed and communicated to the author by Amnon Neeman.

\end{example}

\subsection{A converse: Lifting a recollement to the level of model categories}\label{subsec-converse of main theorem}

Let $\class{M} = (\class{W}, \class{F})$ be a fixed injective cotorsion pair in an abelian category $\class{A}$ with enough injectives.
We have just shown above that if $(\class{W}', \class{F}')$, and $(\class{W}'', \class{F}'')$ are also injective cotorsion pairs and if $\class{F}'' \subseteq \class{F}$ and $\class{W}'' \cap \class{F} = \class{F}'$, then there is a recollement $\class{F}'/\sim \ \xrightarrow{} \class{F}/\sim \ \xrightarrow{} \class{F}''/\sim$. The author was very pleased to learn from an anonymous referee that there is a converse of sorts. That is, all recollements with $\class{F}/\sim$ in the middle come from injective cotorsion pairs in the way described in Theorem~\ref{them-finding recollements}! In particular, this shows that all such recollements arise from Quillen functors between model structures on $\class{A}$. The same is true for just localization or colocalization sequences. Since we are now going in a new direction - building model structures on $\class{A}$ starting with thick subcategories of $\Ho{\class{A}} = \class{F}/\sim$ we necessarily need different tools and theory in this section. In particular, we will apply the theory of torsion pairs and torsion triples in triangulated categories and we will relate these notions to cotorsion pairs (model structures). We refer the reader to~\cite{beligiannis-reiten} as a reference for torsion pairs and torsion triples in triangulated categories. We will only formulate the injective versions of our results here but of course there are projective versions as well.

For the remainder of this section, we let $\class{M} = (\class{W}, \class{F})$ be a fixed injective cotorsion pair in an abelian category $\class{A}$ with enough injectives.
We point out again (see Proposition~\ref{prop-classification of homotopic maps in injective models}) that the full subcategory $\class{F}$ of fibrant objects naturally inherits the structure of a Frobenius category with its projective-injective objects being precisely the injective objects of $\class{A}$. The classical stable category $\class{F}/\sim$ is a triangulated category and the inclusion $\class{F}/\sim \ \hookrightarrow  \Ho{\class{M}}$ is an equivalence of triangulated categories. The fibrant replacement functor $R : \class{A} \xrightarrow{} \class{F}/\sim \ \hookrightarrow \Ho{\class{M}}$ coincides with using enough injectives in the cotorsion pair $\class{M} = (\class{W}, \class{F})$. It takes short exact sequences in $\class{A}$ to exact triangles in $\class{F}/\sim$ by~\cite[Lemmas~1.4.3 and~1.4.4]{becker}. Note that for $F \in \class{F}/\sim$ its suspension $\Sigma F \in \class{F}/\sim$ is computed, up to a unique isomorphism in $\class{F}/\sim$, by finding a short exact sequence $0 \xrightarrow{} F \xrightarrow{} I \xrightarrow{} \Sigma F \xrightarrow{} 0$ in $\class{F}$ with $I$ injective.  On the other hand, the loop $\Omega F \in \class{F}/\sim$ is computed, up to a unique isomorphism in $\class{F}/\sim$, by using enough projectives of the cotorsion pair $(\class{W}, \class{F})$ to get a short exact sequence $0 \xrightarrow{} \Omega F \xrightarrow{} I \xrightarrow{} F \xrightarrow{} 0$ with $\Omega F \in \class{F}$ and $I \in \class{W} \cap \class{F}$ necessarily being injective.

Recall once again that a \emph{thick} subcategory of $\cat{A}$ is a class of objects that is closed under retracts and satisfies the two out of three property on short exact sequences in $\class{A}$. We have the following similar definitions which will be applied to the Frobenius category $\class{F}$ and the triangulated category $\class{F}/\sim$.

\begin{definition}
We call an additive subcategory $\class{T}$ of $\class{F}$ an \emph{$\class{F}$-thick subcategory} if it is closed under retracts and satisfies the two out of three property on short exact sequences in $\class{F}$. That is, on short exact sequences having all three terms in $\class{F}$. By a \emph{thick subcategory} of the triangulated category $\class{F}/\sim$ we mean a  triangulated subcategory (so closed under suspensions and satisfying the two out of three property on exact triangles) which is closed under retracts.
\end{definition}

$\class{F}$-thick subcategories of $\class{F}$ which contain the injectives correspond in a very straightforward way to thick subcategories of $\class{F}/\sim$. That is, if $\class{T}$ is $\class{F}$-thick and contains the injectives then its image under the projection $\class{F} \xrightarrow{} \class{F}/\sim$, which we denote by $\class{T}/\sim$, is thick. This uses in particular the fact that triangles in $\class{F}/\sim$ are obtainable from short exact sequences in $\class{F}$ whose terms are isomorphic in $\class{F}/\sim$ to those in the original triangle. Such an argument appears in our proof of Proposition~\ref{prop-torsion pairs and cotorsion pairs} ahead. Conversely, a thick subcategory of $\class{F}/\sim$ consists of a collection of objects $\class{T} \subseteq \class{F}$. It necessarily contains the injectives since all injectives are isomorphic to zero in $\class{F}/\sim$, and indeed $\class{T}$ is $\class{F}$-thick. This uses in particular the easier fact that short exact sequences in $\class{F}$ give rise to exact triangles in $\class{F}/\sim$. All told we get the following well-known Lemma which we won't prove in more detail.

\begin{lemma}\label{lemma-bijection between F-thick and F mod twiddle-thick}
There is a bijection  between the following two classes:

$$\{\, \class{F}\text{-thick subcategories of } \class{F} \text{ that contain the injectives} \,\} \longleftrightarrow$$
$$\{\, \text{thick subcategories of } \class{F}/\sim  \,\}             $$
\end{lemma}

\begin{remark}
Let $(\class{W}', \class{F}')$ be another injective cotorsion pair with $\class{F}' \subseteq \class{F}$. Then $\class{F}'$ is $\class{F}$-thick and contains the injectives. The key point to check is that for a given short exact sequence  $0 \xrightarrow{} F \xrightarrow{} F' \xrightarrow{} F'' \xrightarrow{} 0$ with $F \in \class{F}$ and $F', F'' \in \class{F}'$ we indeed have $F \in \class{F}'$ too. To see this, use that $(\class{W}', \class{F}')$ is an injective cotorsion pair to find another short exact sequence $0 \xrightarrow{} F''' \xrightarrow{} I \xrightarrow{} F'' \xrightarrow{} 0$ with $F''' \in \class{F}'$ and $I \in \class{W}' \cap \class{F}'$ necessarily injective. Taking the pullback $P$ of $F' \xrightarrow{} F'' \xleftarrow{} I$ yields another short exact sequence $0 \xrightarrow{} F \xrightarrow{} P \xrightarrow{} I \xrightarrow{} 0$ with $P \in  \class{F}'$ since it is an extension of $F'$ by $F'''$. But this last s.e.s must split since $(\class{W}, \class{F})$ is an injective cotorsion pair. Since $\class{F}'$ is closed under retracts we get $F \in \class{F}'$.
\end{remark}

\begin{setup}\label{setup}
We now fix a standard set-up along with notation that will be used throughout this section to prove Proposition~\ref{prop-torsion pairs and cotorsion pairs}. First, we have a fixed injective cotorsion pair $(\class{W}, \class{F})$ for which $\class{F}/\sim$ serves as our ambient triangulated category. By Lemma~\ref{lemma-bijection between F-thick and F mod twiddle-thick} we speak interchangeably of $\class{F}$-thick subcategories $\class{T} \subseteq \class{F}$ which contain the injectives and thick subcategories $\class{T}/\sim$ of $\class{F}/\sim$. We assume we are given such a class $\class{T}$ and we define the following class in $\class{A}$:
$$\class{W}_{\class{T}} = \{\, A \in \class{A} \ | \ \text{there exists a s.e.s. } 0 \xrightarrow{} A \xrightarrow{} T \xrightarrow{} W \xrightarrow{} 0  \text{ with } T \in \class{T} , W \in \class{W} \,\}$$

\end{setup}

\begin{lemma}\label{lemma-properties of the W class}
The following hold for any $\class{W}_{\class{T}}$ class.
\begin{enumerate}
\item $\class{W}_{\class{T}} = R^{-1}(\class{T}/\sim)$. That is, $\class{W}_{\class{T}}$ consists precisely of the objects of $\class{A}$ whose fibrant replacement in $\class{M} = (\class{W}, \class{F})$ lies in $\class{T}$.
\item $\class{W}_{\class{T}}$ is a thick subcategory of $\class{A}$ and contains both $\class{T}$ and $\class{W}$.
\item $\class{W}_{\class{T}} \cap \class{F} = \class{T}$
\end{enumerate}
\end{lemma}

\begin{proof}
For (1), recall that a fibrant replacement of $A$ is obtained by using enough injectives in the cotorsion pair $(\class{W},\class{F})$ to get a s.e.s. $0 \xrightarrow{} A \xrightarrow{} F \xrightarrow{} W \xrightarrow{} 0$ with $F \in \class{F}$ and $W \in \class{W}$. If $A \in \class{W}_{\class{T}}$, then there is another s.e.s. $0 \xrightarrow{} A \xrightarrow{} T \xrightarrow{} W' \xrightarrow{} 0$ with $T \in \class{T}$ and $W' \in \class{W}$. One can argue that $F$ is isomorphic to $T$ in the stable category $\class{F}/\sim$, and from this it follows that there are injectives $I$ and $J$ such that $F \oplus I \cong T \oplus J$ in $\class{F}$. Since $\class{T}$ is thick and contains the injectives we get that $F \in \class{T}$ too.

(2) The proof of the thickness of $\class{W}_{\class{T}}$ follows just like the proof of the thickness of the class $\class{W}$ given in~\cite[Proposition~1.4.2]{becker}.
Clearly $\class{T} \subseteq \class{W}_{\class{T}}$ because for any given $T \in \class{T}$, there is the s.e.s. $0 \xrightarrow{} T \xrightarrow{1} T \xrightarrow{} 0 \xrightarrow{} 0$. Also, for any $W \in \class{W}$, taking a fibrant replacement $0 \xrightarrow{} W \xrightarrow{} F \xrightarrow{} W' \xrightarrow{} 0$ will have $F \in \class{F} \cap \class{W}$. So $F$ is  injective and hence in $\class{T}$ since we are assuming $\class{T}$ contains the injectives. This shows $\class{W} \subseteq \class{W}_{\class{T}}$.

(3) We already know $\class{T}$ is contained in both $\class{F}$ and $\class{W}_{\class{T}}$. On the other hand, if $A \in \class{W}_{\class{T}} \cap \class{F}$ then there is a s.e.s. $0 \xrightarrow{} A \xrightarrow{} T \xrightarrow{} W \xrightarrow{} 0$ with $T \in \class{T}$ and $W \in \class{W}$. But with $A$ bing in $\class{F}$ we know that this sequence must split, making $A$ a retract of $T$. Hence $A \in \class{T}$ since $\class{T}$ is thick.

\end{proof}

\begin{lemma}\label{lemma-bijection between F-thick and thick W's}
Continuing with the notation in Set-Up~\ref{setup}, there is a bijection between the following two classes:

$$\{\, \class{F}\text{-thick subcategories of } \class{F} \text{ that contain the injectives} \,\} \longleftrightarrow$$
$$\{\, \text{thick subcategories of } \class{A} \text{ that contain } \class{W}  \,\}             $$ which acts by $\class{T} \hookrightarrow \class{W}_{\class{T}}$ and has inverse $\class{W}' \cap \class{F} \hookleftarrow \class{W}'$.
\end{lemma}

\begin{proof}
The mapping $\class{T} \hookrightarrow \class{W}_{\class{T}}$ is well-defined by Lemma~\ref{lemma-properties of the W class}. The alleged inverse mapping $\class{W}' \cap \class{F} \hookleftarrow \class{W}'$ is also well-defined, for if $\class{W}'$ is a thick subcategory of $\class{A}$ containing $\class{W}$, then $\class{W}' \cap \class{F}$ is clearly an $\class{F}$-thick subcategory of $\class{F}$, and it contains the injectives since both $\class{F}$ and $\class{W}$ do. To prove our claim that these are inverse mappings we must show $\class{W}_{\class{T}} \cap \class{F} = \class{T}$ and $\class{W}_{\class{W}' \cap \class{F}} = \class{W}'$. We already have $\class{W}_{\class{T}} \cap \class{F} = \class{T}$ from Lemma~\ref{lemma-properties of the W class} so we focus on $\class{W}_{\class{W}' \cap \class{F}} = \class{W}'$.

$\class{W}_{\class{W}' \cap \class{F}} \supseteq \class{W}'$. Suppose $W' \in \class{W}'$. We need to find a s.e.s $0 \xrightarrow{} W' \xrightarrow{} F \xrightarrow{} W \xrightarrow{} 0$ with $F \in \class{W}' \cap \class{F}$ and $W \in \class{W}$. We simply apply enough injectives of the cotorsion pair $(\class{W}, \class{F})$ to get such a s.e.s. $0 \xrightarrow{} W' \xrightarrow{} F \xrightarrow{} W \xrightarrow{} 0$. Since $\class{W}'$ is thick and contains $\class{W}$ we conclude that $F \in \class{W}'$. So $F \in \class{W}' \cap \class{F}$ and we have $W' \in \class{W}_{\class{W}' \cap \class{F}}$.

$\class{W}_{\class{W}' \cap \class{F}} \subseteq \class{W}'$. Let $A \in \class{W}_{\class{W}' \cap \class{F}}$. Then there exists a s.e.s. $0 \xrightarrow{} A \xrightarrow{} W' \xrightarrow{} W \xrightarrow{} 0$ with $W' \in \class{W}' \cap \class{F}$ and $W \in \class{W}$. Since $\class{W}'$ is thick and contains $\class{W}$ we conclude that $A \in \class{W}'$.
\end{proof}

\begin{notation}
Given a thick subcategory $\class{T}/\sim$ of $\class{F}/\sim$, then $\rightperp{(\class{T}/\sim)}$ will denote the Hom-orthogonal in the triangulated category $\class{F}/\sim$. That is, $\rightperp{(\class{T}/\sim)} = \{\, F \in \class{F}/\sim \,|\, \Hom_{\class{F}/\sim}(T,F) = 0 \text{ for all } T \in \class{T}/\sim  \,\}$. On the other hand, for a class $\class{W}'$ of objects in $\class{A}$, recall $\rightperp{\class{W}'}$ denotes the Ext-orthogonal in $\class{A}$.
\end{notation}

\begin{lemma}\label{lemma-right perps are equal}
Continuing with the notation in Set-Up~\ref{setup} we have $\rightperp{\class{W}_{\class{T}}} \subseteq \class{F}$ and $\rightperp{(\class{T}/\sim)} = (\rightperp{\class{W}_{\class{T}}})/\sim$.

\end{lemma}

\begin{proof}
Since $\class{W} \subseteq \class{W}_{\class{T}}$ (Lemma~\ref{lemma-properties of the W class}) we get $\class{F} = \rightperp{\class{W}} \supseteq \rightperp{\class{W}_{\class{T}}}$.

We show $(\rightperp{\class{W}_{\class{T}}})/\sim \ \subseteq \rightperp{(\class{T}/\sim)}$. Let $T \in \class{T}$ and $F \in \rightperp{\class{W}_{\class{T}}}$. We wish to show $\Hom_{\class{F}/\sim}(T,F) = 0$. Since $T,F \in \class{F}$ and $\class{F}$ is Frobenius we have $\Hom_{\class{F}/\sim}(T,F) = \Ext^1_{\class{F}}(\Sigma T,F)$. But $\Sigma T \in \class{T} \subseteq \class{W}_{\class{T}}$, so $\Ext^1_{\class{F}}(\Sigma T,F) = \Ext^1_{\class{A}}(\Sigma T,F) = 0$, proving $\Hom_{\class{F}/\sim}(T,F) = 0$.

Next we show $\rightperp{(\class{T}/\sim)} \subseteq (\rightperp{\class{W}_{\class{T}}})/\sim$. Let $F \in \rightperp{(\class{T}/\sim)}$,  so that $\Hom_{\class{F}/\sim}(T,F) = 0$ for all $T \in \class{T}$. Let $A \in \class{W}_{\class{T}}$ be given, so that it comes with a short exact sequence $0 \xrightarrow{} A \xrightarrow{} T \xrightarrow{} W  \xrightarrow{} 0$ with $T \in \class{T}$ and $W \in \class{W}$. We must show $\Ext^1_{\class{A}}(A,F) = 0$. We apply $\Hom_{\class{A}}(-,F)$ to the s.e.s and get exactness of $\Ext^1_{\class{A}}(T,F) \xrightarrow{} \Ext^1_{\class{A}}(A,F) \xrightarrow{} \Ext^2_{\class{A}}(W,F)$. But we have $0 =  \Hom_{\class{F}/\sim}(\Omega T,F) = \Ext^1_{\class{A}}(T,F)$, and also $\Ext^2_{\class{A}}(W,F) = 0$ since $(\class{W}, \class{F})$ is an hereditary cotorsion pair. It follows that $\Ext^1_{\class{A}}(A,F) = 0$.

\end{proof}

We refer the reader to~\cite{beligiannis-reiten} for the definition of a \emph{torsion pair} in a triangulated category.

\begin{proposition}\label{prop-torsion pairs and cotorsion pairs}
Continuing with the notation in Set-Up~\ref{setup}, the bijection $\class{T}/\sim \ \hookrightarrow \class{W}_{\class{T}}$ of Lemmas~\ref{lemma-bijection between F-thick and F mod twiddle-thick} and~\ref{lemma-bijection between F-thick and thick W's} restricts to a bijection between the following classes:
$$\{\, \class{T} \ |\ (\class{T}/\sim, \rightperp{(\class{T}/\sim)}) \text{is a torsion pair in } \class{F}/\sim \,\} \longleftrightarrow$$
$$\{\, \class{W}' \ |\ (\class{W}', \rightperp{\class{W}'}) \text{is an injective cotorsion pair in } \class{A} \text{ with } \class{W}' \supseteq \class{W} \,\} $$
In particular, given $\class{T}$ with $(\class{T}/\sim, \rightperp{(\class{T}/\sim)})$ a torsion pair in $\class{F}/\sim$ we have that $(\class{W}_{\class{T}}, \rightperp{\class{W}_{\class{T}}})$ is an injective cotorsion pair in $\class{A}$ with $\rightperp{\class{W}_{\class{T}}} \subseteq \class{F}$ and $\rightperp{(\class{T}/\sim)} = (\rightperp{\class{W}_{\class{T}}})/\sim$. Conversely, given an injective cotorsion pair $(\class{W}',\class{F}')$ in $\class{A}$ with $\class{F}' \subseteq \class{F}$ we have that $((\class{W}' \cap \class{F})/\sim, \class{F}'/\sim)$ is a torsion pair in $\class{F}/\sim$.

\end{proposition}

\begin{proof}
First suppose $(\class{T}/\sim, \rightperp{(\class{T}/\sim)})$ is a torsion pair in $\class{F}/\sim$. We wish to show $(\class{W}_{\class{T}}, \rightperp{\class{W}_{\class{T}}})$ is a complete cotorsion pair in $\class{A}$. Lets set $\class{F}' = \rightperp{\class{W}_{\class{T}}}$. By Lemma~\ref{lemma-right perps are equal} we have $\class{F}' \subseteq \class{F}$ and $(\class{T}/\sim, \class{F}'/\sim)$ is a torsion pair. We first show that this implies the following: For any $F \in \class{F}$, there is a short exact sequence $0 \xrightarrow{} F \xrightarrow{} F'_1 \xrightarrow{} T_1 \xrightarrow{} 0$ with $F'_1 \in \class{F}'$ and $T_1 \in \class{T}$ and another short exact sequence $0 \xrightarrow{} F'_2 \xrightarrow{} T_2 \xrightarrow{} F \xrightarrow{} 0$ with $F'_2 \in \class{F}'$ and $T_2 \in \class{T}$. To find the first short exact sequence for a given $F$, we use the definition of torsion pair to get an exact triangle $T \xrightarrow{} F \xrightarrow{} F' \xrightarrow{} \Sigma T$ in $\class{F}/\sim$ with $T \in \class{T}/\sim$ and $F' \in \class{F}'/ \sim$. Rotate the triangle to get an exact triangle $F \xrightarrow{f} F' \xrightarrow{} \Sigma T \xrightarrow{} \Sigma F$, and note that $\Sigma T \in \class{T}/\sim$ since the class is thick.
Let $0 \xrightarrow{} F \xrightarrow{i} I(F) \xrightarrow{} \Sigma F \xrightarrow{} 0$ be a short exact sequence with $I(F)$ injective and form the short exact sequence
$0 \xrightarrow{} F \xrightarrow{\alpha \,=\, (i, -f)} I(F) \oplus F' \xrightarrow{\beta \,=\, f' + i'} F'' \xrightarrow{} 0$.
Here, $f'$ and $i'$ are the maps in the pushout square below and $F'' \in \class{F}$ because $\class{F}$ is closed under extensions. (See Proposition~2.12 of~\cite{buhler-exact categories} for more details on this construction. In the current case the cokernel $F''$ is automatically in $\class{F}$ because the class is coresolving. But for constructing the dual short exact sequence $0 \xrightarrow{} F'_2 \xrightarrow{} T_2 \xrightarrow{} F \xrightarrow{} 0$ we need to use the dual of this construction.)
$$\begin{CD}
0 @>>> F @>i>> I(F) @>>> \Sigma F @>>> 0 \\
@. @VfVV @VVf'V @| @.\\
0 @>>> F' @>i' >> F'' @>>> \Sigma F @>>> 0 \\
\end{CD}$$
Now having in hand the short exact sequence
$0 \xrightarrow{} F \xrightarrow{\alpha \,=\, (i, -f)} I(F) \oplus F' \xrightarrow{\beta \,=\, f' + i'} F'' \xrightarrow{} 0$ in $\class{F}$ , recall that by~\cite[Lemma~2.7]{happel-triangulated} there must be a map $\gamma : F'' \xrightarrow{} \Sigma F$ in $\class{F}$ giving an exact triangle $F \xrightarrow{\alpha} I(F) \oplus F' \xrightarrow{\beta} F'' \xrightarrow{\gamma} \Sigma F$ in $\class{F}/\sim$. It is easy to check that the left square below commutes in $\class{F}/\sim$ and is an isomorphism (from $f$ to $\alpha$) in $\class{F}/\sim$, and so extends to an isomorphism of triangles as shown. $$\begin{CD}
F @> f >> F' @>>> \Sigma T @>>> \Sigma F \\
@| @V i_{F'} VV @VV \exists \, \cong V @| \\
F @> \alpha >> I(F) \oplus F' @> \beta >> F'' @> \gamma >> \Sigma F \\
\end{CD}$$ Since $\Sigma T \in \class{T}/\sim$ it follows that $F'' \in \class{T}$. Therefore we have shown that the short exact sequence $0 \xrightarrow{} F \xrightarrow{\alpha} I \oplus F' \xrightarrow{\beta} F'' \xrightarrow{} 0$ is the desired $0 \xrightarrow{} F \xrightarrow{} F'_1 \xrightarrow{} T_1 \xrightarrow{} 0$ with $F'_1 \in \class{F}'$ and $T_1 \in \class{T}$. A similar argument works to find the other short exact sequence $0 \xrightarrow{} F'_2 \xrightarrow{} T_2 \xrightarrow{} F \xrightarrow{} 0$ with $F'_2 \in \class{F}'$ and $T_2 \in \class{T}$. For this, first use the given torsion pair to get an exact triangle $T \xrightarrow{} \Omega F \xrightarrow{} F' \xrightarrow{} \Sigma T$ in $\class{F}/\sim$ with $T \in \class{T}/\sim$ and $F' \in \class{F}'/ \sim$. Then rotate it twice to get $F' \xrightarrow{} \Sigma T \xrightarrow{} \Sigma \Omega F \xrightarrow{} \Sigma F'$ and use the $(\class{F}/\sim)$-isomorphism $\Sigma \Omega F \cong F$ to replace this with an exact triangle $F' \xrightarrow{} \Sigma T \xrightarrow{g} F \xrightarrow{} \Sigma F'$. Now follow an argument dual to the above, replacing $g$ with an epimorphism, to find the s.e.s $0 \xrightarrow{} F'_2 \xrightarrow{} T_2 \xrightarrow{} F \xrightarrow{} 0$.

Now we can see why $(\class{W}_{\class{T}}, \class{F}')$ is a cotorsion pair. We have $\class{F}' = \rightperp{\class{W}_{\class{T}}}$ by definition, so $\class{W}_{\class{T}} \subseteq \leftperp{\class{F}'}$ is automatic, which means it is left to show $\class{W}_{\class{T}} \supseteq \leftperp{\class{F}'}$. So let $A \in \class{A}$ be in $\leftperp{\class{F}'}$. Using that $(\class{W}, \class{F})$ has enough injectives find a short exact sequence $0 \xrightarrow{} A \xrightarrow{} F \xrightarrow{} W \xrightarrow{} 0$ with $F \in \class{F}$ and $W \in \class{W}$. From the above paragraph we can also find a short exact sequence $0 \xrightarrow{} F' \xrightarrow{} T \xrightarrow{} F \xrightarrow{} 0$ with $F' \in \class{F}'$ and $T \in \class{T}$. Taking the pullback of $A \xrightarrow{} F \xleftarrow{} T$ we get a commutative diagram with a bicartesian square (i.e. a pullback and pushout square) as shown.
$$\begin{CD}
@. 0 @. 0 \\
@. @VVV @VVV @. @.\\
@.  F' @= F'  \\
@. @VVV @VVV @. @.\\
0 @>>> P @>>> T @>>> W @>>> 0 \\
@. @VVV @VVV @| @.\\
0 @>>> A @>>> F @>>> W @>>> 0 \\
@. @VVV @VVV @. @.\\
@. 0 @. 0 \\
\end{CD}$$ Note that by definition we have $P \in \class{W}_{\class{T}}$. Since $A \in \leftperp{\class{F}'}$, the left vertical column splits. We conclude $A \in \class{W}_{\class{T}}$ since $\class{W}_{\class{T}}$ is closed under retracts. This finishes the proof that  $(\class{W}_{\class{T}}, \class{F}')$ is a cotorsion pair. If one follows the exact argument we just gave for an \emph{arbitrary} $A \in \class{A}$ (rather than assuming $A \in \leftperp{\class{F}'}$), we see immediately that this cotorsion pair has enough projectives. To see that it has enough injectives we do a dual ``pushout'' argument. In detail, for an arbitrary $A \in \class{A}$ again use that $(\class{W}, \class{F})$ has enough injectives to find a short exact sequence $0 \xrightarrow{} A \xrightarrow{} F \xrightarrow{} W \xrightarrow{} 0$ with $F \in \class{F}$ and $W \in \class{W}$. From the last paragraph we also have a short exact sequence $0 \xrightarrow{} F \xrightarrow{} F' \xrightarrow{} T \xrightarrow{} 0$ with $F' \in \class{F}'$ and $T \in \class{T}$. Taking the pushout of $F' \xleftarrow{} F \rightarrow{} W$ leads to a commutative diagram with bicartesian square as shown.
$$\begin{CD}
@. @. 0 @. 0 \\
@. @. @VVV @VVV @. @.\\
0 @>>> A @>>> F @>>> W @>>> 0 \\
@. @| @VVV @VVV @.\\
0 @>>> A @>>> F' @>>> P @>>> 0 \\
@. @. @VVV @VVV @.\\
@. @. T @= T  \\
@. @. @VVV @VVV @. @.\\
@. @. 0 @. 0 \\
\end{CD}$$ Since $W, T \in \class{W}_{\class{T}}$ we get the extension $P \in \class{W}_{\class{T}}$. So the second row in the diagram shows that the cotorsion pair $(\class{W}_{\class{T}}, \class{F}')$ has enough injectives, and so it is complete.

On the other hand, say $(\class{W}_{\class{T}}, \class{F}')$ is an injective cotorsion pair. To see that $(\class{T}/\sim, \class{F}'/\sim)$ is a torsion pair in $\class{F}/\sim$ the only non-automatic thing required is to construct, for a given $F \in \class{F}/\sim$, an exact triangle $T \xrightarrow{} F \xrightarrow{} F' \xrightarrow{} \Sigma T$ with $T \in \class{T}/\sim$ and $F' \in \class{F}'/\sim$. But for a given $F \in \class{F}$ we use enough projectives of $(\class{W}_{\class{T}}, \class{F}')$ to get a short exact sequence $0 \xrightarrow{} F' \xrightarrow{} T \xrightarrow{} F \xrightarrow{} 0$ with $F' \in \class{F}'$ and $T \in \class{W}_\class{T} \cap \class{F} = \class{T}$. This gives rise to an exact triangle $F' \xrightarrow{} T \xrightarrow{} F \xrightarrow{} \Sigma F'$ in $\class{F}/\sim$ which rotates to give us the desired exact triangle $T \xrightarrow{} F \xrightarrow{} \Sigma F' \xrightarrow{} \Sigma T$. It follows that $(\class{T}/\sim, \class{F}'/\sim)$ is a torsion pair in $\class{F}/\sim$.

The remaining statements in the Proposition follow from the previous Lemmas.

\end{proof}

We now are ready to present the converse to both Becker's result in Proposition~\ref{prop-Beckers theorem} and to the main Theorem~\ref{them-finding recollements}. First, say $\class{T}' \xrightarrow{F} \class{T} \xrightarrow{G} \class{T}''$ is a localization sequence with right adjoints $F_{\rho}$ and $G_{\rho}$ as in Definition~\ref{def-localization sequence}. Then $F$ and $G_{\rho}$ are each fully faithful and we will identify $\class{T}'$ and $\class{T}''$ with the thick subcategories $\class{T}' = \im{F}$ and $\class{T}'' = \im{G_{\rho}}$. In this way $\class{T}'$ is identified as what we call a \emph{right admissible} subcategory of $\class{T}$ (meaning the inclusion has a right adjoint) while $\class{T}''$ is identified as a \emph{left admissible} subcategory of $\class{T}$ (its inclusion has a left adjoint). With these identifications we have $(\class{T}', \class{T}'')$ is a torsion pair in $\class{T}$. See~\cite[discussion before Lemma~3.2]{krause-stable derived cat of a Noetherian scheme}. Of course, if  $\class{T}' \xrightarrow{F} \class{T} \xrightarrow{G} \class{T}''$ were a colocalization sequence from the start, we get in a similar way a torsion pair $(\class{T}'', \class{T}')$. Indeed the left adjoints $F_{\lambda}$ and $G_{\lambda}$ form a localization sequence $\class{T}'' \xrightarrow{G_{\lambda}} \class{T} \xrightarrow{F_{\lambda}} \class{T}'$. So identifying  $\class{T}'' = \im{G_{\lambda}}$ and $\class{T}' = \im{F}$ we get the torsion pair $(\class{T}'', \class{T}')$. Conversely any torsion pair $(\class{X},\class{Y})$ in a triangulated category $\class{T}$ yields a (co)localization sequence. $\class{X}$ is right admissible while $Y$ is left admissible. See~\cite[Proposition~2.6]{beligiannis-reiten} and also~\cite[Lemma~3.2]{krause-stable derived cat of a Noetherian scheme}. Finally, consider the recollement diagram of Definition~\ref{def-recollement}. Then $F$, $G_{\rho}$, and $G_{\lambda}$ are all fully faithful. Identifying $\class{T}' = \im{F}$ and setting $\class{T}_{\lambda} = \im{G_{\lambda}}$ and $\class{T}_{\rho} = \im{G_{\rho}}$\,, we see that the colocalization sequence gives a torsion pair $(\class{T}_{\lambda}, \class{T}')$ while the localization sequence gives a torsion pair $(\class{T}', \class{T}_{\rho})$. We say $\class{T}' = \im{F}$ is an \emph{admissible subcategory} of $\class{T}$ (for it has both a left and right adjoint) and also call $(\class{T}_{\lambda}, \class{T}', \class{T}_{\rho})$ a \emph{torsion triple}. Conversely, an admissible subcategory gives rise to a torsion triple and a recollement. Finally we are able to state the promised converses as a corollary to Proposition~\ref{prop-torsion pairs and cotorsion pairs}.

\begin{corollary}\label{cor-bijective correspondences between admissible subcats and injective cot pairs}
Proposition~\ref{prop-torsion pairs and cotorsion pairs} gives the following.
\begin{enumerate}
\item Left admissible subcategories $\class{F}'/\sim$ of $\class{F}/\sim$ are in bijective correspondence with injective cotorsion pairs $(\class{W}', \class{F}')$ in $\class{A}$ with $\class{F}' \subseteq \class{F}$.

\item Admissible subcategories $\class{F}'/\sim$ of $\class{F}/\sim$ are in bijective correspondence with two injective cotorsion pairs $(\class{W}', \class{F}')$ and $(\class{W}'', \class{F}'')$ in $\class{A}$ satisfying $\class{F}'' \subseteq \class{F}$ and $\class{W}'' \cap \class{F} = \class{F}'$.
\end{enumerate}
\end{corollary}
We describe how the bijections work in the proof below.

\begin{proof}
As described above, a left admissible subcategory $\class{F}'/\sim$ of $\class{F}/\sim$ is equivalent to a colocalization sequence $\class{F}'/\sim \ \hookrightarrow \class{F}/\sim \ \xrightarrow{} \class{T}/\sim$ where $\class{T}/\sim$ is another thick subcategory of $\class{F}/\sim$ and this in turn is equivalent to a torsion pair $(\class{T}/\sim, \class{F}'/\sim)$. In particular, this is equivalent to the injective cotorsion pair $(\class{W}_{\class{T}}, \class{F}')$ characterized by $\class{W}_\class{T} \cap \class{F} = \class{T}$.

In a similar way, an admissible subcategory $\class{F}'/\sim$ of $\class{F}/\sim$ is equivalent to a colocalization sequence $\class{F}'/\sim \ \hookrightarrow \class{F}/\sim \ \xrightarrow{} \class{T}_{\lambda}/\sim$ along with a localization sequence $\class{F}'/\sim \ \hookrightarrow \class{F}/\sim \ \xrightarrow{} \class{T}_{\rho}/\sim$ where $\class{T}_{\lambda}/\sim$ and $\class{T}_{\rho}/\sim$ are each thick subcategories of $\class{F}/\sim$. That is, it is equivalent to a torsion triple $(\class{T}_{\lambda}/\sim, \class{F}'/\sim, \class{T}_{\rho}/\sim)$. Here, the torsion pair $(\class{T}_{\lambda}/\sim, \class{F}'/\sim)$ gives the injective cotorsion pair $(\class{W}', \class{F}') = (\class{W}_{\class{T}_{\lambda}}, \class{F}')$ with $\class{W}' \cap \class{F} = \class{T}_\lambda$ while the torsion pair $(\class{F}'/\sim, \class{T}_{\rho}/\sim)$ gives another injective cotorsion pair $(\class{W}'', \class{F}'') = (\class{W}_{\class{F}'}, \class{T}_{\rho})$ with $\class{W}'' \cap \class{F} = \class{F}'$.

\end{proof}

\begin{note}
The recollement diagram of Theorem~\ref{them-finding recollements} gives rise to the torsion triple $((\class{W}_2 \cap \class{F}_1)/\sim \,, \class{F}_2/\sim \,, \class{F}_3/\sim)$ in the triangulated category $\class{F}_1/\sim$. As a concrete example, using the notation of Example~\ref{example-becker gets Krause}, Krause's recollement corresponds to the torsion triple $((\class{W}_2 \cap \dwclass{I})/\sim,  \exclass{I}/\sim, \dgclass{I}/\sim)$ in $\dwclass{I}/\sim$.

\end{note}


\section{The Gorenstein injective cotorsion pair}\label{sec-the Gorenstein injective cotorsion pair}

Let $\cat{A}$ be any abelian category with enough injectives. It is clear that the canonical injective cotorsion pair $(\class{A},\class{I})$ satisfies $\class{I} \subseteq \class{F}$ whenever $(\class{W},\class{F})$ is another injective cotorsion pair. We now show two things. First, for any injective cotorsion pair $(\class{W},\class{F})$ we have $\class{F} \subseteq \class{GI}$, where $\class{GI}$ is the class of Gorenstein injective objects. Second, whenever the class $\class{GI}$ forms the right half of a complete cotorsion pair $(\leftperp{\class{GI}}, \class{GI})$, then this is automatically an injective cotorsion pair too. These facts have an immediate application (Corollary~\ref{cor-Goren inj chain complexes coincide with the degreewise Goren inj complexes}) to a simple characterization of the Gorenstein injective chain complexes in Ch($\cat{A}$) and leads us to consider a possible lattice structure on $\cat{A}$ in Section~\ref{sec-lattice}.

We start with the definition of a Gorenstein injective object in $\cat{A}$.

\begin{definition}\label{def-Gorenstein injective objects}
Let $\cat{A}$ be an abelian category with enough injectives and let $M \in \class{A}$. We call $M$ \emph{Gorenstein injective} if $M = Z_0J$ for some exact complex $J$ of injectives which remains exact after applying $\Hom_{\cat{A}}(I,-)$ for any injective object $I$. We will also call such a complex $J$ a \emph{complete injective resolution of $M$}.

\end{definition}

See~\cite{enochs-jenda-book} for a basic reference on Gorenstein injective $R$-modules. In~\cite{bravo-gillespie-hovey} it is shown that whenever $R$ is a (left) Noetherian ring then $(\leftperp{\class{G}\class{I}}, \class{G}\class{I})$ is a complete cotorsion pair, cogenerated by a set, and that it is in fact an injective cotorsion pair.

\begin{theorem}\label{them-Gorenstein injectives are at top of lattice}
Let $\cat{A}$ be an abelian category with enough injectives and let $\class{G}\class{I}$ denote the class of Gorenstein injectives in $\cat{A}$. Then we have $\class{F} \subseteq \class{GI}$ whenever $(\class{W},\class{F})$ is an injective cotorsion pair. Moreover, whenever $(\leftperp{\class{G}\class{I}}, \class{G}\class{I})$ is a complete cotorsion pair it is automatically an injective cotorsion pair too.

\end{theorem}

\begin{proof}
First, suppose $(\class{W},\class{F})$ is an injective cotorsion pair and let $F \in \mathcal{F}$. We must show that $F$ is Gorenstein injective. We will construct a complete
injective resolution of $F$. First, since $\cat{A}$ has enough injectives we may take a usual injective coresolution of $F$. But this coresolution will actually remain exact after applying $\Hom_{\cat{A}}(I,-)$ for any injective $I$ since  $(\class{W},\class{F})$ is hereditary and since $\class{W}$ contains any injective $I$. So this gives us the right half $F \hookrightarrow J^*$ of a complete injective resolution of $F$.

To build the left half of a complete injective resolution use that $(\class{W},\class{F})$ has enough injectives to get a short exact sequence

  \begin{displaymath}
    \tag{\text{$*$}}
    0 \longrightarrow F' \longrightarrow J \longrightarrow F
    \longrightarrow 0,
  \end{displaymath}
  where $J \in \class{W}$ and $F' \in \mathcal{F}$. As $\mathcal{F}$
  is closed under extensions, and since $F, F' \in \mathcal{F}$, it
  follows that $J \in \class{W} \cap \mathcal{F}$. That is, $J$ must be injective. Since $F' \in \class{F}$ we again have $\Ext^1_{\cat{A}}(I,F')=0$ for all injectives $I$. Hence the exact sequence ($*$) stays exact under
  application of the functor $\Hom_{\cat{A}}(I,-)$ whenever $I$ is injective. Iterating this process allows us to construct the left half of a complete injective
  resolution $J_* \twoheadrightarrow F$. We compose to set $J = J_* \twoheadrightarrow F \hookrightarrow J^*$ with $M = Z_0J$ making $J$ the desired complete injective resolution of $M$.

Next, set $\class{V} = \leftperp{\class{GI}}$ and suppose that $(\class{V}, \class{G}\class{I})$ is a complete cotorsion pair. If $M$ is Gorenstein injective then it follows from the definition that $\Ext^n_{\cat{A}}(I,M) = 0$ for any injective object $I$ and $n \geq 1$. So $\class{V}$ contains the injective objects. Next, we claim that the class $\class{GI}$ is cosyzygy closed. Indeed it is clear that if $M$ is Gorenstein injective with complete injective resolution $J$, then $Z_nJ$ are also all Gorenstein injective by definition. In particular we have the short exact sequence $0 \xrightarrow{} M \xrightarrow{} J_0 \xrightarrow{} Z_{-1}J \xrightarrow{} 0$ with $Z_{-1}J$ also Gorenstein injective. So if $0 \xrightarrow{} M \xrightarrow{} I \xrightarrow{} Z \xrightarrow{} 0$ is any other short exact sequence with $I$ injective we get $\Ext^1_{\cat{A}}(V,Z) \cong \Ext^2_{\cat{A}}(V,M) \cong \Ext^1_{\cat{A}}(V,Z_{-1}J) = 0$. So the class $\class{G}\class{I}$ is cosyzygy closed. Since $\cat{A}$ has enough injectives we conclude $(\class{V}, \class{G}\class{I})$ is hereditary from Lemma~\ref{lemma-hereditary test}. Now $(\class{V}, \class{G}\class{I})$ satisfies the hypotheses of Lemma~\ref{lem-Henriks lemma} and so $\class{V}$ is thick. Finally, since we have shown that $\class{V}$ is thick and contains the injectives we get from part~(2) of Proposition~\ref{prop-characterizing injective cotorsion pairs} that $(\class{V}, \class{G}\class{I})$ is an injective cotorsion pair.

\end{proof}

\subsection{The Gorenstein projective cotorsion pair}\label{subsec-Gorenstein projective cotorsion pair}

We now state the dual result concerning the Gorenstein projectives.

\begin{definition}\label{def-Gorenstein projective objects}
Let $\cat{A}$ be an abelian category with enough projectives and let $M \in \class{A}$. We call $M$ \emph{Gorenstein projective} if $M = Z_0Q$ for some exact complex $Q$ of projectives which remains exact after applying $\Hom_{\cat{A}}(-,P)$ for any projective object $P$. We will also call such a complex $Q$ a \emph{complete projective resolution of $M$}.

\end{definition}

Again, see~\cite{enochs-jenda-book} for a basic reference on Gorenstein projective $R$-modules. It is shown in~\cite{bravo-gillespie-hovey} that $(\class{GP}, \rightperp{\class{G}\class{P}})$ is in fact cogenerated by a set, so complete, whenever $R$ is a ring for which all \emph{level} modules have finite projective dimension. We define level modules in Section~\ref{sec-Gorenstein AC recollements}. But this includes all (left) coherent rings in which all flat modules have finite projective dimension. We pause now to comment on the extraordinarily large class of rings satisfying the condition that all flat modules have finite projective dimension.

In~\cite{simson}, Simson gives a short proof of the following: If a ring $R$ has
cardinality at most $\aleph_{n}$ then the maximal projective dimension of a
flat module is at most $n+1$. This amazing result doesn't depend on whether the ring is commutative or Noetherian or anything and so there is an abundance of rings with the property that all flat modules have finite projective dimension. However, there is the continuum hypothesis. Putting cardinality aside, Enochs, Jenda and L\'opez-Ramos have considered rings in~\cite{enochs-jenda-lopezramos-n-perfect} they call $n$-perfect. These are rings in which all flat modules have projective dimension at most $n$. In that paper and in other papers coauthored by Enochs they give numerous examples of $n$-perfect rings. In particular, a perfect ring is 0-perfect and any $n$-Gorenstein ring is $n$-perfect. In this language the above result of Simson says that any ring $R$ with cardinality at most $\aleph_{n}$ is $(n+1)$-perfect.
There is more. A main example of an $n$-perfect ring is a commutative  Noetherian ring of finite Krull dimension $n$. This follows from~\cite{jensen} and~\cite{Gruson-Raynaud}. Finally, we point out a generalization of this due to Peter J\o rgensen. In the article~\cite{jorgensen-finite flat dimension}, he shows that every flat module has finite
projective dimension whenever $R$ is right-Noetherian and has a dualizing
complex. Note that our condition of saying $R$ has finite projective dimension for each flat module is more general than saying $R$ is $n$-perfect because we are not assuming an upper bound on the projective dimensions.

\begin{theorem}\label{them-Gorenstein projectives are at top of lattice}
Let $\cat{A}$ be an abelian category with enough projectives and let $\class{G}\class{P}$ denote the class of Gorenstein projectives in $\cat{A}$. Then we have $\class{C} \subseteq \class{GP}$ whenever $(\class{C},\class{W})$ is a projective cotorsion pair. Moreover, whenever $(\class{GP}, \rightperp{\class{G}\class{P}})$ is a complete cotorsion pair it is automatically a projective cotorsion pair too.

\end{theorem}

\section{The semilattice of injective cotorsion pairs in $R$-Mod and $\textnormal{Ch}(R)$}\label{sec-lattice}

We now let $R$ be a ring and we assume for the remainder of the paper that $\cat{A}$ is either the category of $R$-modules or the category of chain complexes of $R$-modules. That is, now $\cat{A}$ denotes either $R\textnormal{-Mod}$ or $\textnormal{Ch}(R)$. In light of Theorem~\ref{them-Gorenstein injectives are at top of lattice}, Theorem~\ref{them-finding recollements} and Corollary~\ref{cor-bijective correspondences between admissible subcats and injective cot pairs}, it makes sense to consider whether or not the injective cotorsion pairs form a lattice. We now make a brief investigation, again for $\cat{A}$ being $R$-Mod or $\ch$. Ultimately we are only able to show that arbitrary suprema of injective cotorsion pairs exist and so we call the ordering a \emph{semilattice}. The question remains open as to whether or not infima of injective cotorsion pairs are again injective cotorsion pairs.

\begin{lemma}\label{lemma-sup and inf}
Suppose $\{\,(\class{C}_i,\class{D}_i)\,\}_{i \in I}$ is a collection of cotorsion pairs in $\cat{A}$, each cogenerated by some class $\class{S}_i$ and generated by some class $\class{T}_i$.
\begin{enumerate}
\item $({}^\perp(\cap_{i \in I} \class{D}_i) ,  \cap_{i \in I} \class{D}_i)$ is a cotorsion pair cogenerated by the class $\cup_{i \in I} \, \class{S}_i$. In particular, if each $\class{S}_i$ is a \emph{set}, then the cotorsion pair is also cogenerated by a set.
\item $(\cap_{i \in I} \class{C}_i ,  (\cap_{i \in I} \class{C}_i)^\perp)$ is a cotorsion pair generated by the class $\cup_{i \in I} \, \class{T}_i$. In particular, if each $\class{T}_i$ is a \emph{set}, then the cotorsion pair is also generated by a set.
\end{enumerate}
\end{lemma}

\begin{proof}
It is straightforward to check that $(\cup_{i \in I} \, \class{S}_i)^\perp = \cap_{i \in I} \class{D}_i$ which proves (1) and ${}^\perp(\cup_{i \in I} \, \class{T}_i) = \cap_{i \in I} \class{C}_i$ which proves (2).
\end{proof}

We can put a partial order on the class of all cotorsion pairs using containment of either the classes on the left side or the classes on the right side. For the theory of model categories it is best to change the ordering depending on whether we are focusing on injective or projective cotorsion pairs. But following either ordering, Lemma~\ref{lemma-sup and inf} guarantees suprema and infima so that a partial ordering gives a complete lattice on the class of all cotorsion pairs. It is nontrivial that restricting the partial ordering to just those cotorsion pairs which are cogenerated by a set forms a complete sublattice. But this follows from the fact that $\cap_{i \in I} \class{C}_i$ is \emph{deconstructible} (as defined in~\cite{Stovicek-Hill}) whenever each $(\class{C}_i,\class{D}_i)$ is cogenerated by a set~\cite[Proposition~2.9]{Stovicek-Hill}.

\subsection{The semilattice of injective cotorsion pairs}

Suppose we have two injective cotorsion pairs $\class{M}_1 = (\class{W}_1, \class{F}_1)$ and  $\class{M}_2 = (\class{W}_2, \class{F}_2)$ in $\cat{A}$. Then we define $\class{M}_2  \preceq_r \class{M}_1 $ if and only if $\class{F}_2 \subseteq \class{F}_1$. Note that this happens if and only if the inclusion functor $\class{M}_2 \to \class{M}_1$ is a right Quillen functor. With respect to $\preceq_r$ we have that the canonical injective cotorsion pair is the least element with respect to this ordering and from Theorem~\ref{them-Gorenstein injectives are at top of lattice} the Gorenstein injective cotorsion pair, whenever it exists, is the maximum element. We know it exists whenever $R$ is Noetherian by~\cite{bravo-gillespie-hovey}.

\begin{proposition}\label{prop-sublattice of injective cot pairs}
Let $\{\,(\class{W}_i, \class{F}_i)\,\}_{i \in I}$ be a collection of injective cotorsion pairs each cogenerated by a set $\class{S}_i$. Then its supremum $\bigvee_{i \in I} (\class{W}_i , \class{F}_i) = (\cap_{i \in I} \class{W}_i ,  (\cap_{i \in I} \class{W}_i)^\perp)$ is also an injective cotorsion pair cogenerated by a set.
\end{proposition}

\begin{proof}
Since each $(\class{W}_i, \class{F}_i)$ is cogenerated by a set, each class $\class{W}_i$ is deconstructible by~\cite{Stovicek-Hill}. Then by~\cite[Proposition~2.9~(2)]{Stovicek-Hill} it follows that $\cap_{i \in I} \class{W}_i$ is also deconstructible. This implies the cotorsion pair $(\cap_{i \in I} \class{W}_i ,  (\cap_{i \in I} \class{W}_i)^\perp)$ is cogenerated by a set, and so is complete. Since each $\class{W}_i$ is thick and contains the injective objects we get that $\cap_{i \in I} \class{W}_i$ is also thick and contains the injectives. So $(\cap_{i \in I} \class{W}_i ,  (\cap_{i \in I} \class{W}_i)^\perp)$ is an injective cotorsion pair by Proposition~\ref{prop-characterizing injective cotorsion pairs}.

\end{proof}

\begin{remark}
Proposition~\ref{prop-sublattice of injective cot pairs} says that the ordering on the class of all injective cotorsion pairs that are cogenerated by a set is a ``complete join-semilattice'' sitting inside the complete lattice of all cotorsion pairs that are cogenerated by a set. By a complete join-semilattice we mean that the supremum of any given set exists. The author doesn't know whether or not infima exist. That is, let $\{\,(\class{W}_i, \class{F}_i)\,\}_{i \in I}$ be a collection of injective cotorsion pairs each cogenerated by a set $\class{S}_i$. Then we know from Lemma~\ref{lemma-sup and inf} that $({}^\perp(\cap_{i \in I} \class{F}_i) ,  \cap_{i \in I} \class{F}_i)$ is the infimum in the lattice of all cotorsion pairs that are cogenerated by a set. We would like to know if this cotorsion pair is injective. (The only thing not clear is whether or not the left class is closed under taking cokernels of monomorphisms between its objects. Equivalently, does $[{}^\perp(\cap_{i \in I} \class{F}_i)] \cap [\cap_{i \in I} \class{F}_i]$ consist only of injective modules?)

Finally, suppose $R$ is a ring with the Gorenstein injective cotorsion pair $(\leftperp{\class{G}\class{I}}, \class{G}\class{I})$ being complete. Then $\class{G}\class{I}$ becomes a Frobenius category and we point out that the semilattice of injective cotorsion pairs in $R$-Mod embeds inside the lattice of thick subcategories of the stable category $\class{G}\class{I}/\sim$ by Corollary~\ref{cor-bijective correspondences between admissible subcats and injective cot pairs}.

\end{remark}

\subsection{The semilattice of projective cotorsion pairs}

 On the other hand we have the semilattice of projective cotorsion pairs in $\class{A}$. We give the simple proof of the dual result below. It is necessarily different due to the fact that we must always \emph{cogenerate} a cotorsion pair by a set to get completeness.

Suppose that $\class{M}_1 = (\class{C}_1, \class{W}_1)$ and  $\class{M}_2 = (\class{C}_2, \class{W}_2)$ are projective cotorsion pairs in $\cat{A}$. Then we define $\class{M}_2  \preceq_l \class{M}_1 $ if and only if $\class{C}_2 \subseteq \class{C}_1$. Note that this happens if and only if the inclusion functor $\class{M}_2 \to \class{M}_1$ is a left Quillen functor. With respect to $\preceq_l$ we have that the canonical projective cotorsion pair is the least element with respect to the ordering and from Theorem~\ref{them-Gorenstein projectives are at top of lattice} the Gorenstein projective cotorsion pair, whenever it exists, is the maximum element. We know it exists whenever $R$ is a coherent ring in which all flat modules have finite projective dimension by~\cite{bravo-gillespie-hovey}.

\begin{proposition}\label{prop-sublattice of projective cot pairs}
Let $\{\,(\class{C}_i, \class{W}_i)\,\}_{i \in I}$ be a collection of projective cotorsion pairs each cogenerated by a set $\class{S}_i$. Then its supremum, which is the cotorsion pair
$\bigvee_{i \in I} (\class{C}_i , \class{W}_i) =  ({}^\perp(\cap_{i \in I} \class{W}_i) ,  \cap_{i \in I} \class{W}_i)$, is also a projective cotorsion pair cogenerated by a set.
\end{proposition}

\begin{proof}
Here it is clear that $\cap_{i \in I} \class{W}_i$ is thick and contains the projectives since each $\class{W}_i$ does. Also, if $\{\,\class{S}_i\,\}_{i \in I}$ represents the cogenerating sets, then $\cup_{i \in I} \class{S}_i$ is a cogenerating set by Lemma~\ref{lemma-sup and inf}. So $\bigvee_{i \in I} (\class{C}_i , \class{W}_i)$ is complete and so a projective cotorsion pair by Proposition~\ref{prop-characterizing projective cotorsion pairs}.
\end{proof}

\subsection{Examples}\label{subsec-example of injective cot pairs in R-Mod}
In general, the semilattice of injective (or projective) cotorsion pairs on $\ch$ is always more interesting than the one on $R$-Mod and depends much on the global (Gorenstein) dimension of $R$. We will look at examples concerning $\ch$ at the end of Section~\ref{sec-lifting from modules to chain complexes}.

For a generic ring $R$, there seems to be just two injective cotorsion pairs on $R$-Mod that are typically of interest. The first is the canonical injective cotorsion pair $(\cat{A},\class{I})$ and the second the Gorenstein AC-injective cotorsion pair $(\class{W},\class{GI})$ discussed more in Section~\ref{sec-Gorenstein AC recollements}. When $R$ is Noetherian this is exactly the Gorenstein injective cotorsion pair.  But a ring can certainly have more than these two injective cotorsion pairs. For example, if $R$ is Noetherian then as we see in the next section there are several injective cotorsion pairs on $\ch$. But $\ch$ is really just a graded version of the category of $R[x]/(x^2)$-modules, and $R[x]/(x^2)$-Mod will have these analogous injective cotorsion pairs as well.

We note the following simplification for a ring of finite global dimension.

\begin{example}\label{example-finite global dim for modules}
Suppose $R$ is a ring with $\text{gl.dim}(R) < \infty$. Then the Gorenstein injective modules coincide with the injective modules. So in this case there is only one injective cotorsion pair in $R$-Mod, the categorical one. Similarly there is only one projective cotorsion pair in $R$-Mod, as the canonical projective cotorsion pair coincides with the Gorenstein projective cotorsion pair.

\end{example}

\section{Lifting from modules to chain complexes}\label{sec-lifting from modules to chain complexes}

The author has considered before the problem of lifting any given cotorsion pair $(\class{F},\class{C})$ in an abelian category $\class{A}$ to various cotorsion pairs in $\ch$. In particular, see~\cite{gillespie} and~\cite{gillespie-degreewise-model-strucs}.
We now show that when the ground pair $(\class{F},\class{C})$ is an injective (resp. projective) cotorsion pair, then these lifted pairs are also injective (resp. projective). We start with the following notations which were introduced in~\cite{gillespie} and~\cite{gillespie-degreewise-model-strucs}.

\begin{definition}\label{def-classes of complexes}
Given a class of $R$-modules $\class{C}$, we define the following classes of chain complexes in $\ch$.
\begin{enumerate}
\item $\dwclass{C}$ is the class of all chain complexes with $C_n \in \mathcal{C}$.
\item $\exclass{C}$ is the class of all exact chain complexes with $C_n \in \mathcal{C}$.
\item $\tilclass{C}$ is the class of all exact chain complexes with cycles $Z_nC \in \mathcal{C}$.
\end{enumerate}
\end{definition}
The ``dw'' is meant to stand for ``degreewise'' while the ``ex'' is meant to stand for ``exact''. When $\mathcal{C}$ is the class of projective (resp. injective, resp. flat) modules, then $\tilclass{C}$ are the categorical projective (resp. injective, resp. flat) chain complexes.

Moreover, if we are given any cotorsion pair $(\class{F}, \class{C})$ in $R$-Mod, then following~\cite{gillespie} we will denote $\rightperp{\tilclass{F}}$ by $\dgclass{C}$ and $\leftperp{\tilclass{C}}$ by $\dgclass{F}$.

\subsection{Injective cotorsion pairs in $\ch$ from one in $R$-Mod}\label{subsec-Injjective cotorsion pairs of complexes coming from one in R-Mod} Let $R$ be any ring and $(\class{W}, \class{F})$ be an injective cotorsion pair in $R$-Mod which is cogenerated by a set. We see below that this information allows for the construction of six injective cotorsion pairs, also each cogenerated by a set, in $\ch$. However, depending on the pair $(\class{W}, \class{F})$ that we start with and the particular ring $R$, some of these six will coincide. See the examples ahead in Section~\ref{subsec-examples of lattices in ch(R)}.

\begin{proposition}\label{prop-induced injective cot pairs on complexes}
Let $(\mathcal{W},\mathcal{F})$ be an injective cotorsion pair of $R$-modules cogenerated by some set. Then the following are also each injective cotorsion pairs in $\ch$, and cogenerated by sets.
\begin{enumerate}
\item $({}^\perp\dwclass{F}, \dwclass{F})$
\item $({}^\perp\exclass{F}, \exclass{F})$
\item $(\dgclass{W}, \tilclass{F})$

\

\item $(\dwclass{W}, (\dwclass{W})^\perp)$
\item $(\exclass{W}, (\exclass{W})^\perp)$
\item $(\tilclass{W}, \dgclass{F})$

\

\end{enumerate}
\end{proposition}

\begin{proof}
This time we give the proofs for the projective case and this appears in Proposition~\ref{prop-induced projective cot pairs on complexes} below. We note that the only difference between the projective and injective case is that we are always \emph{cogenerating} by a set to get completeness of the cotorsion pair. For (1)--(3) one can cogenerate by a set due to~\cite[Propositions~4.3, 4.4, and~4.6]{gillespie-degreewise-model-strucs} and for (4)--(6) one can use the theory of deconstructible classes in~\cite{Stovicek-Hill}.
\end{proof}

\subsection{Projective cotorsion pairs in $\ch$ from one in $R$-Mod}\label{subsec-Projective cotorsion pairs of complexes coming from one in R-Mod} Let $R$ be any ring and $(\class{C}, \class{W})$ be a projective cotorsion pair in $R$-Mod which is cogenerated by a set. We have the dual statement to Proposition~\ref{prop-induced injective cot pairs on complexes} which we now prove.

\begin{proposition}\label{prop-induced projective cot pairs on complexes}
Let $(\mathcal{C},\mathcal{W})$ be a projective cotorsion pair of $R$-modules cogenerated by some set. Then the following are also each projective cotorsion pairs in $\ch$, and cogenerated by sets.
\begin{enumerate}
\item $(\dwclass{C}, (\dwclass{C})^\perp)$
\item $(\exclass{C}, (\exclass{C})^\perp)$
\item $(\tilclass{C}, \dgclass{W})$

\

\item $({}^\perp\dwclass{W}, \dwclass{W})$
\item $({}^\perp\exclass{W}, \exclass{W})$
\item $(\dgclass{C}, \tilclass{W})$
\end{enumerate}
\end{proposition}

\begin{proof}
The proofs for (1)--(3) are all similar as are the proofs for (4)--(6).

Lets prove (6). First, it follows from~\cite[Section~3]{gillespie} that $({}^\perp\tilclass{W}, \tilclass{W})$ is a cotorsion pair. It is easy to show that $\tilclass{W}$ is thick because $\class{W}$ is thick. (In particular, note that any short exact sequence $0 \to W' \to W \to W'' \to 0$ of complexes in $\tilclass{W}$ gives rise to a short exact sequence $0 \to Z_nW' \to Z_nW \to Z_nW'' \to 0$ on the level of cycles because the complex $W'$ is exact.) Next, recall that a projective chain complex is one which is exact with projective cycles. Since $\class{W}$ contains the projective modules, we get that $\tilclass{W}$ contains the projective complexes. So since $\tilclass{W}$ is thick and contains the projectives it will follow from Proposition~\ref{prop-characterizing projective cotorsion pairs} that $({}^\perp\tilclass{W}, \tilclass{W})$ is a projective cotorsion pair once we know it is complete. Note here that if we had started with $({}^\perp\dwclass{W}, \dwclass{W})$ or $({}^\perp\exclass{W}, \exclass{W})$ then again we get that each of these is a cotorsion pair by~\cite[Propositions~3.2 and~3.3]{gillespie-degreewise-model-strucs} and similarly we argue that both $\dwclass{W}$ and $\exclass{W}$ are thick and contain the projectives. But finally, it was shown in~\cite[Propositions~4.3--4.6]{gillespie-degreewise-model-strucs} that each of the cotorsion pairs $({}^\perp\dwclass{W}, \dwclass{W})$, $({}^\perp\exclass{W}, \exclass{W})$ and $({}^\perp\tilclass{W}, \tilclass{W})$ are cogenerated by a set. This proves (4)--(6).

Now lets prove (1)--(3), It follows again from~\cite{gillespie} and~\cite{gillespie-degreewise-model-strucs} that these are all cotorsion pairs. Lets focus on (2) for example, since (1) and (3) will be similar. First, from~\cite[Proposition~3.3]{gillespie-degreewise-model-strucs} we note that for this cotorsion pair $(\exclass{C}, (\exclass{C})^\perp)$ the right class  $(\exclass{C})^\perp$ equals the class of all complexes $W$ for which $W_n \in \mathcal{W}$ and such that $\homcomplex(C,W)$ is exact whenever $C \in \exclass{C}$. ($\homcomplex$ is defined in Section~\ref{sec-preliminaries}. By Lemma~\ref{lemma-homcomplex-basic-lemma} we get that $\homcomplex(C,W)$ is exact if and only if any chain map $f : \Sigma^n C \to W$ is null homotopic if and only if any chain map $f : C \to \Sigma^n W$ is null homotopic.) In light of Proposition~\ref{prop-characterizing projective cotorsion pairs} we wish to show that this class $(\exclass{C})^\perp$ is thick and contains the projectives and that the cotorsion pair $(\exclass{C}, (\exclass{C})^\perp)$ is complete. The projective complexes are easily seen to be in $(\exclass{C})^\perp$ by its above description (recall $\class{W}$ contains the projectives and any chain map into a projective is null). Next, $(\exclass{C})^\perp$ is clearly closed under retracts since it is the right side of a cotorsion pair. To complete the thickness claim suppose that $0 \to W' \to W \to W'' \to 0$ is a short exact sequence of complexes. If any two of the three $W', W, W''$ are in $(\exclass{C})^\perp$ then note that since $\mathcal{W}$ is thick we get that all of the $W'_n, W_n, W'_n$ are in $\class{W}$. It now follows that for any $C \in \exclass{C}$ we will always get a short exact sequence of $\homcomplex$-complexes $0 \to \homcomplex(C,W') \to \homcomplex(C,W) \to \homcomplex(C,W') \to 0$ (because $\Ext$ vanishes degreewise). So now if any two of the three complexes $\homcomplex(C,W'), \homcomplex(C,W), \homcomplex(C,W')$ are exact then the third is automatically exact due to the long exact sequence in homology. This completes the proof that  $(\exclass{C})^\perp$ is thick. Finally it is left to show that $(\exclass{C}, (\exclass{C})^\perp)$ is complete. We use the results in~\cite{Stovicek-Hill} pointing out that a class which is the left side of a cotorsion pair is \emph{deconstructible} if and only if that cotorsion pair is cogenerated by set. It follows from~\cite[Theorem~4.2]{Stovicek-Hill} that $\dwclass{C}$ and $\tilclass{C}$ are deconstructible since $\class{C}$ is. So the only cotorsion pair left is $(\exclass{C}, (\exclass{C})^\perp)$. But here note that $\exclass{C} = \dwclass{C} \cap \mathcal{E}$ where $\mathcal{E}$ is the class of exact complexes. Since $\mathcal{E}$ is the left side of a cotorsion pair cogenerated by a set it is deconstructible and since $\dwclass{C}$ is also deconstructible, it follows from~\cite[Proposition~2.9]{Stovicek-Hill} that $\exclass{C}$ is also deconstructible.
\end{proof}

\subsection{Examples and homological dimensions}\label{subsec-examples of lattices in ch(R)}

We briefly discussed the semilattices of injective and projective cotorsion pairs on $R$-Mod in Section~\ref{subsec-example of injective cot pairs in R-Mod}. We continue that discussion now by looking in more detail at the basic examples of model structures on $\ch$ induced from the categorical injective (resp. projective) cotorsion pair and the Gorenstein injective (resp. Gorenstein projective) cotorsion pair. The basic theme is that for a Noetherian ring $R$, the semilattice for $\ch$ becomes more complicated as we move from $R$ having finite global dimension, to $R$ being Gorenstein, to a general Noetherian $R$. In the next section we look in more detail at the models induced by the Gorenstein injective and Gorenstein projective pairs.

We will give projective examples here and leave the obvious dual statements concerning injective cotorsion pairs to the reader. Let $\class{P}$ denote the class of projective modules and $(\class{P},\class{A})$ the canonical projective cotorsion pair. Consider the six projective cotorsion pairs on $\ch$ induced by $(\class{P},\class{A})$ using Proposition~\ref{prop-induced projective cot pairs on complexes}. We see that they are in fact only four distinct pairs. They are $(\dwclass{P}, (\dwclass{P})^\perp)$, $(\exclass{P}, (\exclass{P})^\perp)$, $(\tilclass{P}, \ch)$, and $(\dgclass{P}, \mathcal{E})$. We have the following observation.

\begin{proposition}\label{prop-rings have finite global dim iff exclass equals tilclass}
Let $R$ be any ring with $\text{gl.dim}(R) < \infty$. Then we have $\exclass{P} = \tilclass{P}$ and $\dwclass{P} = \dgclass{P}$.
\end{proposition}

\begin{proof}
$\dgclass{P} \subseteq \dwclass{P}$ is automatic. To show $\dwclass{P} \subseteq \dgclass{P}$ let $P \in \dwclass{P}$ and let $E \in \mathcal{E}$. We must show $\Ext^1(P,E) = 0$. Since $\text{gl.dim}(R) = d < \infty$, it is easy to argue that $E$ has finite projective dimension in $\ch$ [because any $d$th syzygy must be exact and inherit projective cycles.] Letting $$0 \to P^d \to P^{d-1} \to \cdots \to P^1 \to P^0 \to E \to 0$$ be a finite projective resolution we conclude by dimension shifting that $\Ext^1(P,E) = \Ext^{d+1}(P,P^d)$. But $\Ext^{d+1}(P,P^d) = 0$ because $(\dwclass{P}, (\dwclass{P})^\perp)$ is a projective cotorsion pair. Thus $\dwclass{P} = \dgclass{P}$. Also, $\exclass{P} = \tilclass{P}$ because given any $P \in \exclass{P}$ and an $R$-module $N$, we can dimension shift $\Ext^1(Z_nP,N) = \Ext^{d+1}(Z_{n-d}P,N) = 0$.
\end{proof}

So the next two examples clarify further the models obtained on $\ch$ from $(\class{P},\class{A})$ using Proposition~\ref{prop-induced projective cot pairs on complexes}.

\begin{example}\label{example-lattice on ch(R) when R has finite global dim}
Say $R$ has finite global dimension. Then we saw in Example~\ref{example-finite global dim for modules} that $(\class{P},\class{A}) = (\class{GP},\class{W})$ is the only projective cotorsion pair on $R$-Mod. Proposition~\ref{prop-rings have finite global dim iff exclass equals tilclass} tells us that we only get two distinct projective cotorsion pairs coming from Proposition~\ref{prop-induced projective cot pairs on complexes} in this case. This first is $(\dwclass{P},\class{E})$ which is the projective model for the derived category of $R$ and the other is the trivial projective model structure $(\tilclass{P},\class{A})$. It follows from Corollary~\ref{cor-Goren proj chain complexes coincide with the degreewise Goren proj complexes} that $(\dwclass{P},\class{E})$ is actually the Gorenstein projective pair in $\ch$, sitting on top of the semilattice while $(\tilclass{P},\class{A})$ is the canonical projective pair sitting at the bottom of the semilattice.  The author doesn't know whether or not there are others on $\ch$ besides these two when $\text{gl.dim}(R) < \infty$.
 \end{example}

\begin{example}\label{example-four possible cot pairs}
Now suppose that $R$ has infinite global dimension. Then we have all four generally distinct pairs $(\dwclass{P}, (\dwclass{P})^\perp)$, $(\exclass{P}, (\exclass{P})^\perp)$, $(\tilclass{P}, \ch)$, and $(\dgclass{P}, \mathcal{E})$. We note $\tilclass{P}$ is the class of categorical projective complexes and so $(\tilclass{P}, \tilclass{P}^\perp)$ is trivial as a model structure, and not of interest. Next, $\dgclass{P}$ is the class of DG-projective complexes and $(\dgclass{P}, \mathcal{E})$ is the usual \emph{projective model structure} on $\ch$ having homotopy category the usual derived category $\class{D}(R)$.  The model structure associated to $(\dwclass{P},(\dwclass{P})^\perp)$ appears in~\cite{bravo-gillespie-hovey} where it is called the \emph{Proj model structure} on $\ch$. This model structure has also appeared in~\cite{positselski} where its homotopy category was called the \emph{contraderived category} of $R$. The model structure  $(\exclass{P},(\exclass{P})^\perp)$ also appears in~\cite{bravo-gillespie-hovey} where it is called the \emph{exact Proj model structure} on $\ch$. Its homotopy category is the (projective) stable derived category $S(R)$ introduced in~\cite{bravo-gillespie-hovey}. For a general ring $R$ we have the portion of the semilattice shown below.
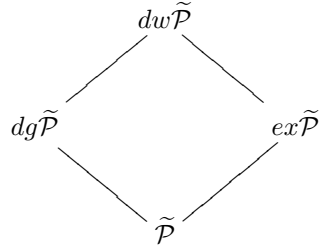
\begin{center}
\begin{displaymath}
\xymatrix{
& & \dwclass{P}  \ar@{-}[dr]  \ar@{-}[dl] & & \\
& \dgclass{P} \ar@{-}[dr] &   & \exclass{P}  \ar@{-}[dl] & \\
& & \tilclass{P}  & &  \\
}
\end{displaymath}
\end{center}
We point out that there are two more model structures on $\ch$ distinct from the above that exist whenever $R$ is a coherent ring in which all flat modules have finite projective dimension. These will appear in~\cite{bravo-gillespie-hovey}.

\end{example}

Having considered projective models on $\ch$ induced from the canonical projective pair $(\class{P},\class{A})$ via Proposition~\ref{prop-induced projective cot pairs on complexes}, we now turn to models induced from the Gorenstein projective cotorsion pair $(\class{GP},\class{W})$. More on this appears in the next Section~\ref{sec-Gorenstein projective and injective models for the derived category}. We now just look at the case when $R$ is a Gorenstein ring.

\begin{example}\label{example-four possible Goren cot pairs}
Let $R$ be a Gorenstein ring and again assume $\text{gl.dim}(R) = \infty$. Recall that $R$ must have finite injective dimension when considered as a module over itself and that these dimensions must coincide. Here assume id($R$) = $d$. Also recall that a module $M$ over $R$ has finite injective dimension if and only if it has finite projective dimension if and only if it has finite flat dimension and that if this is the case then all these dimensions must be $\leq d$. Call these the modules of \emph{finite $R$-dimension} and let $\class{W}$ denote the class of all these modules. Then it was shown in~\cite{hovey} that $(\class{W},\class{GI})$ is a complete cotorsion pair where $\class{GI}$ are the Gorenstein injective complexes and that $(\class{GP},\class{W})$ is a complete cotorsion pair where $\class{GP}$ are the Gorenstein projective complexes. Then we have the following Gorenstein version of Proposition~\ref{prop-rings have finite global dim iff exclass equals tilclass}

\begin{proposition}\label{prop-rings have finite Gorenstein global dim iff exclass equals tilclass}
Let $R$ be a Gorenstein ring with id($R$) = $d$ and let $(\class{GP},\class{W})$ be the Gorenstein projective cotorsion pair. Then $\dwclass{GP} = \dgclass{GP}$ and $\exclass{GP} = \tilclass{GP}$.

\end{proposition}

\begin{proof}
Let $G \in \exclass{GP}$. We wish to show $Z_nG \in \class{GP}$. So let $W \in \class{W}$ and we must show $\Ext^1_R(Z_nG, W) = 0$. But since $\Ext^i_R(C,W) = 0$ for all $C \in \class{GP}$ we can dimension shift to get $\Ext^1_R(Z_nG,W) \cong \Ext^{d+1}_R(Z_{n-d}G, W)$. But since $W$ must have finite injective dimension less than or equal to $d$ we get that this last group equals 0. Therefore $G \in \tilclass{GP}$. So $\exclass{GP} = \tilclass{GP}$. The fact that $\dwclass{GP} = \dgclass{GP}$ is true by combining~\cite[Theorem~3.11]{gill-hovey-generalized derived cats} with Corollary~\ref{cor-Goren inj chain complexes coincide with the degreewise Goren inj complexes} below.

\end{proof}

We conclude that when $R$ is Gorenstein of infinite global dimension, applying Proposition~\ref{prop-induced projective cot pairs on complexes} to both the categorical projective and the Gorenstein projective pairs generally leads to 8 model structures on $\ch$. This is illustrated concretely in the next example where all 8 model structures are distinct.

\end{example}

\begin{example}\label{example-projective cotorsion pairs of complexes over Z mod 4}
Let $R = \mathbb{Z}_4$, the ring of integers mod 4 and consider $\ch$. Then as described in Example~\ref{example-four possible cot pairs}, the projective cotorsion pair $(\mathcal{P},\mathcal{A})$ on $R$-Mod gives rise to the four projective cotorsion pairs $(\dwclass{P}, (\dwclass{P})^\perp)$ , $(\exclass{P}, (\exclass{P})^\perp)$ , $(\tilclass{P}, \ch)$, and $(\dgclass{P}, \mathcal{E})$ in $\ch$. These classes are indeed distinct because the complex $\cdots \mathbb{Z}/4 \xrightarrow{\times 2} \mathbb{Z}/4 \xrightarrow{\times 2} \mathbb{Z}/4 \cdots$ is in $\exclass{P}$ but not $\tilclass{P}$, and so also $\dwclass{P}$ but not $\dgclass{P}$. Recall that $R$ is quasi-Frobenius meaning that the class of injective modules coincides with the class of projective modules. It follows that the Gorenstein projective cotorsion pair on $R$-Mod is $(\class{A},\class{I})$ where $\class{I}$ is the class of injective/projective $R$-modules. The four associated projective cotorsion pairs from Example~\ref{example-four possible Goren cot pairs} turn out to be $(\ch, \tilclass{I})$ and $(\mathcal{E}, \dgclass{I})$ and $({}^\perp\exclass{I}, \exclass{I})$ and $({}^\perp\dwclass{I}, \dwclass{I})$. These eight classes of cofibrant objects are distinct and are related to each other as shown in the cube shaped lattice below.

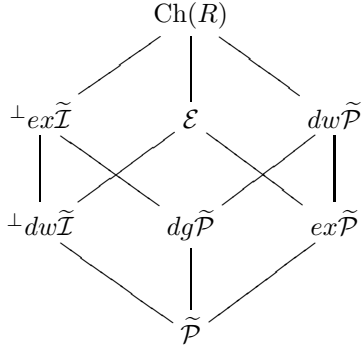
\begin{center}
\begin{displaymath}
\xymatrix{
& & \ch  \ar@{-}[dl] \ar@{-}[d] \ar@{-}[dr] & & \\
& {}^\perp\exclass{I}  \ar@{-}[d] \ar@{-}[dr] & \class{E} \ar@{-}[dl] \ar@{-}[dr] & \dwclass{P} \ar@{-}[dl]  \ar@{-}[d] & \\
& {}^\perp\dwclass{I} \ar@{-}[dr] & \dgclass{P}  \ar@{-}[d] & \exclass{P} \ar@{-}[dl] & \\
& & \tilclass{P}  & &  \\
}
\end{displaymath}
\end{center}

\end{example}

\section{Gorenstein models for the derived category and recollements}\label{sec-Gorenstein projective and injective models for the derived category}

Assume $R$ is any Noetherian ring and let $\class{GI}$ denote the class of Gorenstein injective modules. From~\cite{bravo-gillespie-hovey} we know that the Gorenstein injectives are part of an injective cotorsion pair $(\class{W},\class{GI})$ giving rise to a model structure on $R$-Mod. The resulting model structure on $R$-Mod is called the \emph{Gorenstein injective} model structure on $R$-Mod and coincides with the one in~\cite{hovey} when $R$ is a Gorenstein ring. Now applying Proposition~\ref{prop-induced injective cot pairs on complexes} we potentially get 6 injective model structures on $\ch$. As illustrated by the examples in the previous section, the number of cotorsion pairs on $\ch$ induced by $(\class{W},\class{GI})$ and the canonical $(\class{A},\class{I})$ increases as we consider more general rings. In particular, when $R$ is non-Gorenstein there are indeed many injective model structures on $\ch$. Three particular model structures appearing in~\cite{bravo-gillespie-hovey} are the following:
\begin{enumerate}

\item $\class{M}_1 = (\leftperp{\dwclass{I}}, \dwclass{I})$ = The injective model for the homotopy category of all complexes of injectives (or coderived category in the language of~\cite{positselski}).

\item $\class{M}_2 = (\leftperp{\exclass{I}}, \exclass{I})$ = The injective model for the stable derived category.

\item $\class{M}_3 = (\class{E},\dgclass{I})$ = The injective model for the usual derived category.

\end{enumerate}

These are linked through Krause's recollement~\cite{krause-stable derived cat of a Noetherian scheme} indicated below.
\[
\begin{tikzpicture}[node distance=3.5cm]
\node (A) {$K_{ex}(Inj)$};
\node (B) [right of=A] {$K(Inj)$};
\node (C) [right of=B] {$K(\textnormal{DG-}Inj)$};
\draw[<-,bend left=40] (A.20) to node[above]{\small E$(\class{M}_2)$} (B.160);
\draw[->] (A) to node[above]{\small $I$} (B);
\draw[<-,bend right=40] (A.340) to node [below]{\small C$(\class{M}_3)$} (B.200);
\draw[<-,bend left] (B.20) to node[above]{\small $\lambda$=C$(\class{M}_2)$} (C.160);
\draw[->] (B) to node[above]{\tiny E$(\class{M}_3)$} (C);
\draw[<-,bend right] (B.340) to node [below]{\small $I$} (C.200);
\end{tikzpicture}
\]
Recall that $K(\textnormal{DG-}Inj) \cong \class{D}(R)$ and here the notation such as E$(\class{M}_3)$ and C$(\class{M}_3)$ respectively represent using enough injectives or enough projectives with respect to that cotorsion pair. This corresponds to taking special preenvelopes or precovers.

It is not our purpose at this point to make a detailed study of all of the analogous Gorenstein derived categories. We simply wish to illustrate the usefulness of Theorem~\ref{them-finding recollements} by presenting three Gorenstein analogs of Krause's recollement. These appear in Theorem~\ref{them-three Gorenstein version of Krause recollement} and the Gorenstein projective analogs appear in Theorem~\ref{them-three Gorenstein projective versions of Krause recollement}. In the next section we point out two new and interesting recollement situations involving these derived categories and in Section~\ref{sec-Gorenstein AC recollements} we see an extension to arbitrary rings $R$. We now set some language following the language used in~\cite{bravo-gillespie-hovey}.

\begin{itemize}
\item $\dwclass{GI}$ is the class of (categorical) \emph{Gorenstein injective} complexes by Corollary~\ref{cor-Goren inj chain complexes coincide with the degreewise Goren inj complexes} below. We call the model structure $(\leftperp{\dwclass{GI}}, \dwclass{GI})$ the \emph{Gorenstein injective model structure} on $\ch$.

\item $\exclass{GI}$ is the class of \emph{exact Gorenstein injective} complexes. We call the model structure $(\leftperp{\exclass{GI}}, \exclass{GI})$ the \emph{exact Gorenstein injective model structure} on $\ch$.

\item $\dgclass{GI} = \rightperp{\tilclass{W}}$ is the class of \emph{DG-Gorenstein injective} complexes. We call the model structure $(\tilclass{W}, \dgclass{GI})$ the \emph{DG-Gorenstein injective model structure} on $\ch$.

\item $\tilclass{GI}$ is the class of \emph{exact DG-Gorenstein injective} complexes. We call the model structure $(\dgclass{W}, \tilclass{GI})$ the \emph{exact DG-Gorenstein injective model structure} on $\ch$.

\end{itemize}

By Proposition~\ref{prop-classification of homotopic maps in injective models}, two chain maps $f,g : X \xrightarrow{} F$ in $\ch$ (where $F$ is fibrant) are formally homotopic in any of these model structures if and only if their difference factors through an injective complex. But injective complexes are contractible and it follows that the two maps are homotopic if and only if they are chain homotopic in the usual sense.
So, for example, the homotopy category of the Gorenstein injective model structure is equivalent to $K (GInj)$, the
chain homotopy category of the Gorenstein injective complexes. Similarly, the homotopy
category of the exact Gorenstein injective model structure will be denoted $K_{ex}(GInj)$. Then we have the DG-versions which we will denote by $K(\textnormal{DG-}GInj)$ and $K_{ex}(\textnormal{DG-}GInj)$.

\begin{remark}
We resist any urge to give names to the complexes in $\dwclass{W}$ and $\exclass{W}$ and $\dgclass{W}$ at this point. However, note the following. Since $\class{W}$ contains all the projective modules, $\dwclass{W}$ (resp. $\exclass{W}$, resp. $\dgclass{W}$) contains all complexes of projectives (resp. exact complexes of projectives, resp. DG-projective complexes). Moreover, since $\class{W}$ contains the injectives, $\dwclass{W}$ contains all complexes of injectives and $\exclass{W}$ contains the exact complexes of injectives.

\end{remark}

We have the following portion of the semilattice of injective cotorsion pairs in $\ch$. But note that we have not even included the model structures corresponding to $\rightperp{\dwclass{W}}$ and  $\rightperp{\exclass{W}}$, which are still interesting especially in light of Theorem~\ref{them-lovely recollements}.
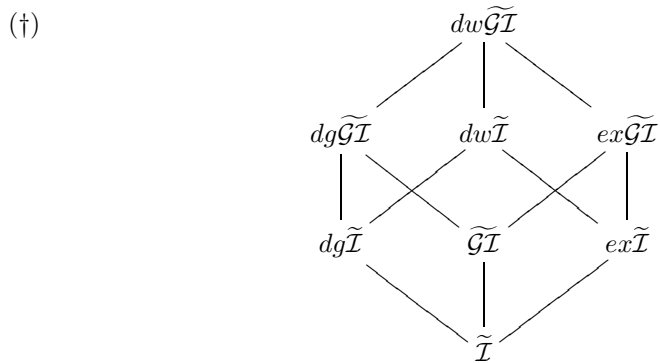
\begin{center}
\begin{displaymath}
\tag{\text{$\dag$}}
\xymatrix{
& & \dwclass{GI}  \ar@{-}[dl] \ar@{-}[d] \ar@{-}[dr] & & \\
& \dgclass{GI}  \ar@{-}[d] \ar@{-}[dr] & \dwclass{I} \ar@{-}[dl] \ar@{-}[dr] & \exclass{GI} \ar@{-}[dl]  \ar@{-}[d] & \\
& \dgclass{I} \ar@{-}[dr] & \tilclass{GI}  \ar@{-}[d] & \exclass{I} \ar@{-}[dl] & \\
& & \tilclass{I}  & &  \\
}
\end{displaymath}
\end{center}

\subsection{Gorenstein injective complexes}

We now wish to characterize the categorical Gorenstein injective complexes, show that these are the fibrant objects in a model structure on $\ch$ having $\class{D}(R)$ as its homotopy category, and embed these homotopy categories in a Gorenstein version of Krause's recollement.

It has been known for some time that over certain rings, especially Gorenstein rings, that the Gorenstein injective (resp. Gorenstein projective) complexes are precisely those complexes $X$ for which each $X_n$ is Gorenstein injective (resp. Gorenstein projective). For example, see~\cite{garcia-rozas}. Enochs, Estrada, and Iacob have shown this for Gorenstein projective complexes over a commutative Noetherian ring admitting a dualizing complex~\cite{enochs-estrada-iacob}. Moreover, we see from~\cite{Yang-Liu-Gorenstein complexes} that this is true for any ring $R$! But Theorem~\ref{them-Gorenstein injectives are at top of lattice} says that the Gorenstein injective complexes sit on the top of the semilattice of injective cotorsion pairs. So the following is a quick and elegant proof of the above fact which works at this point for any Noetherian ring $R$.

\begin{corollary}\label{cor-Goren inj chain complexes coincide with the degreewise Goren inj complexes}
Let $R$ be any ring for which we know that the Gorenstein injective cotorsion pair $(\mathcal{W}, \mathcal{G}\mathcal{I})$ is cogenerated by a set.  Then $(\leftperp{\dwclass{GI}}, \dwclass{GI})$ is also cogenerated by a set and $\dwclass{GI}$ is exactly the class of Gorenstein injective complexes.
\end{corollary}

\begin{proof}
We choose to prove the projective version this time. See proof of Corollary~\ref{cor-Goren proj chain complexes coincide with the degreewise Goren proj complexes}.

\end{proof}

\begin{theorem}\label{them-three Gorenstein version of Krause recollement}
Let $R$ to be any Noetherian ring. Then for the three choices of $\class{M}_1$ and $\class{M}_2$ as indicated below these are injective cotorsion pairs in $\ch$ with $\class{M}_2 \preceq_r \class{M}_1$ having right localization $\class{M}_1/\class{M}_2$ a model for the derived category $\class{D}(R)$.

\begin{enumerate}
\item $\class{M}_1 = ({}^\perp\dwclass{GI}, \dwclass{GI})$ and $\class{M}_2 = ({}^\perp\exclass{GI}, \exclass{GI})$.

\

\item $\class{M}_1 = (\tilclass{W}, \dgclass{GI})$ and $\class{M}_2 = (\dgclass{W}, \tilclass{GI})$.

\

\item $\class{M}_1 = (\exclass{W}, (\exclass{W})^\perp)$ and $\class{M}_2 = (\dwclass{W}, (\dwclass{W})^\perp)$.
\end{enumerate}

Furthermore, each case gives a recollement with the usual derived category as indicated by Theorem~\ref{them-finding recollements}. For example, the first gives the recollement
\[
\begin{tikzpicture}[node distance=3.5cm]
\node (A) {$K_{ex}(GInj)$};
\node (B) [right of=A] {$K(GInj)$};
\node (C) [right of=B] {$K(\textnormal{DG-}Inj)$};
\draw[<-,bend left=40] (A.20) to node[above]{\small E$(\class{M}_2)$} (B.160);
\draw[->] (A) to node[above]{\small $I$} (B);
\draw[<-,bend right=40] (A.340) to node [below]{\small C$(\class{M}_3)$} (B.200);
\draw[<-,bend left] (B.20) to node[above]{$\lambda$=C$(\class{M}_2)$} (C.160);
\draw[->] (B) to node[above]{\tiny E$(\class{M}_3)$} (C);
\draw[<-,bend right] (B.340) to node [below]{\small $I$} (C.200);
\end{tikzpicture}
\]
Recall that $K(\textnormal{DG-}Inj) \cong \class{D}(R)$ and the notation E$(\class{M}_3)$ and C$(\class{M}_3)$ respectively represent using enough injectives or enough projectives with respect to the cotorsion pair $\class{M}_3 = (\class{E},\dgclass{I})$. (So taking DG-injective preenvelopes or Exact precovers.)

\end{theorem}

\begin{proof}
We set $\class{M}_3 = (\class{E},\dgclass{I})$. According to Theorem~\ref{them-finding recollements}, in each case we just need to show $\class{F}_2, \class{F}_3 \subseteq \class{F}_1$ and one of the other two conditions. For statements~(1) and~(2) we will show $\class{E} \cap \class{F}_1 = \class{F}_2$. Then we use the other condition in Theorem~\ref{them-finding recollements} to prove~(3).

Take the first pair, $\class{M}_1 = (\leftperp{\dwclass{GI}}, \dwclass{GI})$ and $\class{M}_2 = (\leftperp{\exclass{GI}}, \exclass{GI})$. They are injective cotorsion pairs in $\ch$ by Proposition~\ref{prop-induced injective cot pairs on complexes}. Since $\class{F}_1 = \dwclass{GI}$ lies at the top of the semilattice it clearly contains both $\class{F}_2 = \exclass{GI}$ and $\class{F}_3 = \dgclass{I}$. It is also clear that $\class{E} \cap \dwclass{GI} = \exclass{GI}$. So we get the recollement for~(1).

Now take the second pair, $\class{M}_1 = (\tilclass{W}, \dgclass{GI})$ and $\class{M}_2 = (\dgclass{W}, \tilclass{GI})$. They are injective cotorsion pairs in $\ch$ by Proposition~\ref{prop-induced injective cot pairs on complexes}. We have $\class{E} \cap \dgclass{GI} = \tilclass{GI}$ since this in general holds from~\cite{gillespie}. So all that is left is to show $\dgclass{I} \subseteq \dgclass{GI}$. That is, we wish to show that every DG-injective complex is a DG-Gorenstein injective. Indeed if $I$ is DG-injective then it is injective in each degree, so Gorenstein injective in each degree. Furthermore any map $f : E \xrightarrow{} I$ where $E$ is exact must be null homotopic. In particular, any map $f : W \xrightarrow{} I$ where $W \in \tilclass{W}$ must be null homotopic. This completes the proof of statement~(2).

Next take the third pair, $\class{M}_1 = (\exclass{W}, (\exclass{W})^\perp)$ and $\class{M}_2 = (\dwclass{W}, (\dwclass{W})^\perp)$. In this setup we first need to see that the DG-injective complexes are in $(\exclass{W})^\perp$. But $(\exclass{W})^\perp$ is the class of all complexes $Y$ of Gorenstein injective modules with the property that any map $f : W \xrightarrow{} Y$ is null homotopic whenever $W \in \exclass{W}$. So we see that the DG-injectives are in this class. So following the setup to Theorem~\ref{them-finding recollements} we do have $\class{F}_2 , \class{F}_3 \subseteq \class{F}_1$ and we will now use the second condition of Theorem~\ref{them-finding recollements}. That is, we will finish by showing $\class{W}_2 \cap \class{W}_3 = \class{W}_1$ and $\class{F}_2 \subseteq \class{W}_3$. But $\class{W}_2 \cap \class{W}_3 = \class{W}_1$ is clear in this case so lets look at $\class{F}_2 \subseteq \class{W}_3$. It is required to show that every complex in $\class{F}_2 = (\dwclass{W})^\perp$ is exact. But as mentioned in the above Remark, the DG-projective complexes must be in $\dwclass{W}$, and so $(\dwclass{W})^\perp$ must be contained in the class of exact complexes. This completes the proof and we get the three recollements.

\end{proof}

\begin{remark}
Theorem~\ref{them-three Gorenstein version of Krause recollement} holds for any ring $R$ for which we know that the Gorenstein injective cotorsion pair $(\class{W}, \class{GI})$ is cogenerated by a set. In fact, whenever $(\class{W},\class{F})$ is an injective cotorsion pair in $R$-Mod which is cogenerated by a set, then the analogous six lifted cotorsion pairs will give rise to three (not necessarily distinct or nontrivial) recollements. For example, starting with the canonical injective cotorsion pair $(\class{A},\class{I})$, the analog of recollement (1) recovers Krause's recollement. However, the analogs to recollements~(2) and~(3) are trivial!

\end{remark}

\subsection{The Gorenstein projective complexes}\label{subsec-Gorenstein projective complexes}

We now look at the dual situation. Here we assume that $R$ is any coherent ring in which each flat module has finite projective dimension. (See Subsection~\ref{subsec-Gorenstein projective cotorsion pair} for examples of such coherent rings. But as we have already alluded to, one can replace this with the more general hypothesis of any ring $R$ in which the so-called \emph{level} modules have finite projective dimension. This is explained in Section~\ref{sec-Gorenstein AC recollements}.)  With this hypothesis on $R$ we have from~\cite{bravo-gillespie-hovey} that if $\class{GP}$ denotes the class of all Gorenstein projective modules then there is a projective cotorsion pair $(\class{GP},\class{W})$ which is cogenerated by a set. (The class $\class{W}$ is in general different than the class $\class{W} = \leftperp{\class{GI}}$, unless we know $R$ is Gorenstein or some other nice conditions on the ring.) Applying Proposition~\ref{prop-induced projective cot pairs on complexes} to this projective cotorsion pair results in 6 lifted projective model structures on $\ch$, all analogous to the Gorenstein injective versions considered above. We record the duals of the injective results, starting with the promised fact that the Gorenstein projective complexes are precisely the complexes having a Gorenstein projective module in each degree.

\begin{corollary}\label{cor-Goren proj chain complexes coincide with the degreewise Goren proj complexes}
Let $R$ be any ring for which we know that the Gorenstein projective cotorsion pair $(\class{GP}, \class{W})$ is cogenerated by a set.  Then $(\dwclass{GP}, \rightperp{(\dwclass{GP})})$ is also cogenerated by a set and $\dwclass{GP}$ is exactly the class of Gorenstein projective complexes.
\end{corollary}

\begin{proof}
$(\dwclass{GP}, \rightperp{(\dwclass{GP})})$ is a projective cotorsion pair by Proposition~\ref{prop-induced projective cot pairs on complexes}. So by the ``lattice theorem'' Theorem~\ref{them-Gorenstein projectives are at top of lattice} we have that $\dwclass{GP}$ is contained in the class of all Gorenstein projective complexes. To finish we just need to show that if $G$ is Gorenstein projective then each $G_n$ must be Gorenstein projective. For this suppose $G$ is a Gorenstein projective chain complex so that there is an exact sequence of projective chain complexes $$\class{P} \equiv \cdots \xrightarrow{} P^2 \xrightarrow{} P^{1} \xrightarrow{} P^{0} \xrightarrow{} P^{-1} \xrightarrow{} \cdots $$ for which $G = \ker{(P^0 \xrightarrow{} P^{-1})}$ and which will remain exact after application of $\Hom_{\ch}(-,Q)$ for any projective chain complex $Q$. Then for any projective \emph{module} $M$ taking $Q$ to be the projective complex $D^{n+1}(M)$ we get exactness of the complex $\Hom_{\ch}(\class{P},D^{n+1}(M)) \cong \Hom_{R}(P^{*}_n,M)$, where $P^{*}_n$ denotes the complex
 $$P^{*}_n \equiv \cdots \xrightarrow{} P^2_n \xrightarrow{} P^{1}_n \xrightarrow{} P^{0}_n \xrightarrow{} P^{-1}_n \xrightarrow{} \cdots $$ Of course $P^{*}_n$ is an exact complex of projective modules and since we also have $G_n = \ker{(P^0_n \xrightarrow{} P^{-1}_n)}$ we conclude that $G_n$ is a Gorenstein projective module.

\end{proof}

\begin{remark}
We use the notation $\class{M}_2\backslash\class{M}_1$ from~\cite{becker} to denote the left localization of a projective cotorsion pair $\class{M}_1$ by another $\class{M}_2$ having $\class{C}_2 \subseteq \class{C}_1$. This notation emphasizes a \emph{left} localization.
\end{remark}

\begin{theorem}\label{them-three Gorenstein projective versions of Krause recollement}
Let $R$ be a coherent ring in which all flat modules have finite projective dimension. Then for the three choices of $\class{M}_1$ and $\class{M}_2$ as indicated below these are projective cotorsion pairs in $\ch$ with $\class{M}_2 \preceq_l \class{M}_1$ having left localization $\class{M}_2\backslash\class{M}_1$ a model for the derived category $\class{D}(R)$.

\begin{enumerate}
\item $\class{M}_1 = (\dwclass{GP}, \rightperp{(\dwclass{GP})})$ and $\class{M}_2 = (\exclass{GP}, \rightperp{(\exclass{GP})})$.

\

\item $\class{M}_1 = (\dgclass{GP}, \tilclass{W})$ and $\class{M}_2 = (\tilclass{GP}, \dgclass{W})$.

\

\item $\class{M}_1 = (\leftperp{\exclass{W}}, \exclass{W})$ and $\class{M}_2 = (\leftperp{\dwclass{W}}, \dwclass{W})$.
\end{enumerate}

Furthermore, each case gives a recollement with the usual derived category as indicated by Theorem~\ref{them-finding projective recollements}. For example, the first gives the recollement
\[
\begin{tikzpicture}[node distance=3.5cm]
\node (A) {$K_{ex}(GProj)$};
\node (B) [right of=A] {$K(GProj)$};
\node (C) [right of=B] {$K(\textnormal{DG-}Proj)$};
\draw[<-,bend left=40] (A.20) to node[above]{\small E$(\class{M}_3)$} (B.160);
\draw[->] (A) to node[above]{\small $I$} (B);
\draw[<-,bend right=40] (A.340) to node [below]{\small C$(\class{M}_2)$} (B.200);
\draw[<-,bend left] (B.20) to node[above]{\small $I$} (C.160);
\draw[->] (B) to node[above]{\tiny C$(\class{M}_3)$} (C);
\draw[<-,bend right] (B.340) to node [below]{$\rho$=E$(\class{M}_2)$} (C.200);
\end{tikzpicture}
\]
Recall that $K(\textnormal{DG-}Proj) \cong \class{D}(R)$ and the notation E$(\class{M}_3)$ and C$(\class{M}_3)$ respectively represent using enough injectives or enough projectives with respect to the cotorsion pair $\class{M}_3 = (\dgclass{P},\class{E})$. (So taking Exact preenvelopes or DG-projective precovers.)

\end{theorem}

\begin{remark}
Theorem~\ref{them-three Gorenstein projective versions of Krause recollement} holds for any ring $R$ for which we know that the Gorenstein projective cotorsion pair $(\class{GP}, \class{W})$ is cogenerated by a set. So as described in Section~\ref{sec-Gorenstein AC recollements}, it holds for any ring $R$ for which all level modules have finite projective dimension.

\end{remark}

\section{More recollement situations}\label{sec-more recollements}

In this section we point out two more interesting recollements, both injective and projective versions, arising from the classes of Gorenstein complexes introduced in Section~\ref{sec-Gorenstein projective and injective models for the derived category}. Note that since the class $\dwclass{GI}$ of Gorenstein injective complexes sits on the top of the semilattice of injective cotorsion pairs, Theorem~\ref{them-finding recollements} says we get a recollement situation whenever we notice $\class{W}_3 \cap \dwclass{GI} = \class{F}_2$.

\begin{theorem}\label{them-lovely recollements}
Let $R$ be a Noetherian ring. Then we have the two recollement situations below.
\[
\begin{tikzpicture}[node distance=3.5cm]
\node (A) {$K(Inj)$};
\node (B) [right of=A] {$K(GInj)$};
\node (C) [right of=B] {$K((\dwclass{W})^\perp)$};
\draw[<-,bend left=40] (A.20) to node[above]{\small E$(\class{M}_2)$} (B.160);
\draw[->] (A) to node[above]{\small $I$} (B);
\draw[<-,bend right=40] (A.340) to node [below]{\small C$(\class{M}_3)$} (B.200);
\draw[<-,bend left] (B.20) to node[above]{$\lambda$=C$(\class{M}_2)$} (C.160);
\draw[->] (B) to node[above]{\tiny E$(\class{M}_3)$} (C);
\draw[<-,bend right] (B.340) to node [below]{\small $I$} (C.200);
\end{tikzpicture}
\]

AND

\[
\begin{tikzpicture}[node distance=3.5cm]
\node (A) {$K_{ex}(Inj)$};
\node (B) [right of=A] {$K(GInj)$};
\node (C) [right of=B] {$K((\exclass{W})^\perp))$};
\draw[<-,bend left=40] (A.20) to node[above]{\small E$(\class{M}_2)$} (B.160);
\draw[->] (A) to node[above]{\small $I$} (B);
\draw[<-,bend right=40] (A.340) to node [below]{\small C$(\class{M}_3)$} (B.200);
\draw[<-,bend left] (B.20) to node[above]{$\lambda$=C$(\class{M}_2)$} (C.160);
\draw[->] (B) to node[above]{\tiny E$(\class{M}_3)$} (C);
\draw[<-,bend right] (B.340) to node [below]{\small $I$} (C.200);
\end{tikzpicture}
\]

\end{theorem}

\begin{proof}
Apply Theorem~\ref{them-finding recollements}. For the first, we take $\class{M}_1 = (\leftperp{\dwclass{GI}}, \dwclass{GI})$, $\class{M}_2 = (\leftperp{\dwclass{I}}, \dwclass{I})$ and $\class{M}_3 = (\dwclass{W}, \rightperp{(\dwclass{W})})$. It is clear that $\dwclass{W} \cap \dwclass{GI} = \dwclass{I}$ since $\class{W} \cap \class{GI}$ is the class of injective modules.

For the second, we take $\class{M}_1 = (\leftperp{\dwclass{GI}}, \dwclass{GI})$, $\class{M}_2 = (\leftperp{\exclass{I}}, \exclass{I})$ and $\class{M}_3 = (\exclass{W}, \rightperp{(\exclass{W})})$. Again, it is clear that $\exclass{W} \cap \dwclass{GI} = \exclass{I}$ since $\class{W} \cap \class{GI}$ is the class of injective modules.
\end{proof}

\begin{remark}\label{remark-loevely recollements dual}
The dual to Theorem~\ref{them-lovely recollements} holds whenever $R$ is a ring for which we know that the Gorenstein projective cotorsion pair is cogenerated by a set. In particular, we have the dual whenever $R$ is coherent and all flat modules have finite projective dimension.

\end{remark}

\begin{example}\label{example-recollement of Z mod 4}
Let $R$ be the quasi-Frobenius ring $\Z/4$ as in Example~\ref{example-projective cotorsion pairs of complexes over Z mod 4}. Then the second recollement in Theorem~\ref{them-lovely recollements} above gives a recollement $S(R) \xrightarrow{} K(R) \xrightarrow{} K(\class{W}_{\text{proj}})$ where $S(R) = K_{ex}(Inj)$ is the stable derived category of $R$, $K(R)$ is the usual homotopy category and $\class{W}_{\text{proj}} = \rightperp{\exclass{P}} = \rightperp{\exclass{I}}$ are the trivial objects in the model for the \emph{projective stable derived category}, $S_{\text{proj}}(R)$, of~\cite{bravo-gillespie-hovey}. As shown in~\cite{bravo-gillespie-hovey}, we have $\class{W}_{\text{proj}} \neq \class{W}_{\text{inj}}$ but $S_{\text{proj}}(R) \cong S(R)$ through a natural Quillen equivalence.

\end{example}

\section{Recollements involving complexes of Gorenstein AC-injectives}\label{sec-Gorenstein AC recollements}

In this section we extend the results of Sections~\ref{sec-Gorenstein projective and injective models for the derived category} and~\ref{sec-more recollements} to general rings. This depends on a generalization of Gorenstein homological algebra to arbitrary rings, provided by~\cite{bravo-gillespie-hovey}. We now briefly explain this, but we refer the reader to~\cite{bravo-gillespie-hovey} for the full details.

A main idea of~\cite{bravo-gillespie-hovey} is that Gorenstein homological algebra only seems to work well over Noetherian rings. We see from~\cite{gillespie-Ding-Chen rings} that the theory extends to coherent rings if we alter the definition of Gorenstein injective (resp. Gorenstein projective) to get the so-called \emph{Ding injective} (resp. \emph{Ding projective}) modules. In~\cite{bravo-gillespie-hovey} we see that the generalization goes beyond coherent rings to general rings $R$. Here we get what we call the \emph{Gorenstein AC-injective} and \emph{Gorenstein AC-projective} modules which we define shortly. When $R$ is coherent, they agree with the Ding modules of~\cite{gillespie-Ding-Chen rings}. For Noetherian rings, the Gorenstein AC-injective modules coincide with the usual Gorenstein injectives. The Gorenstein AC-projective modules are a bit more subtle. They do coincide with the usual Gorenstein projectives whenever $R$ is Noetherian and has a dualizing complex. But as we indicate below, it is not because the ring is Noetherian. It is because certain modules we call \emph{level}, which for Noetherian rings are the flat modules, have finite projective dimension. It is the existence of a dualizing complex which forces this~\cite{jorgensen-finite flat dimension}.

In more detail, a module $N$ is said to be of type $FP_{\infty}$ if it has a projective resolution by finitely generated projectives. Over coherent rings these are precisely the finitely presented modules, and over Noetherian rings these are precisely the finitely generated modules. Loosely speaking, finitely generated modules are well-behaved over Noetherian rings, while finitely presented modules are well-behaved over coherent rings. So as we relax the ring from Noetherian to coherent we need to sharpen the notion of a ``finite'' module. So the idea is that for a general ring $R$, it is the type $FP_{\infty}$ modules that are the well-behaved ``finite'' modules.

Now let $I$ be an R-module. Note that $I$ is injective if and only if $\Ext^1_R(N,I) = 0$ for all finitely generated modules $N$. More generally, $I$ is called \emph{absolutely pure} (or \emph{FP-injective}) if $\Ext^1_R(N,I) = 0$ for all finitely presented modules $N$. Still more generally, for reasons described in~\cite{bravo-gillespie-hovey} we call $I$ \emph{absolutely clean} (or \emph{$FP_{\infty}$-injective}) if $\Ext^1_R(N,I) = 0$ for all modules $N$ of type $FP_{\infty}$. For coherent rings, we have absolutely clean = absolutely pure. For Noetherian rings, absolutely clean = absolutely pure = injective. We point out that $FP_{\infty}$-modules have been studied by Livia Hummel in~\cite{miller-livia}.

Recall that for a left $R$-module $N$, its character module is the right $R$-module $N^+ = \Hom_{\Z}(N,\Q)$. For Noetherian rings $R$, it is known that $N$ is flat if and only if $N^+$ is injective and also $N$ is injective if and only if $N^+$ is flat. (See~\cite{enochs-jenda-book}.) For a coherent ring $R$ it is known that $N$ is flat if and only if $N^+$ is absolutely pure and $N$ is absolutely pure if and only if $N^+$ is flat. (See~\cite{fieldhouse}.) In order for this duality to extend to non-coherent rings we need the following definition, which is an extensions of the notion of flatness: A left $R$-module $N$ is called \emph{level} if $\Tor^R_1(M,N) = 0$ for all right $R$-modules $M$ of type $FP_{\infty}$. Now for a general ring $R$, we get $N$ is level if and only if $N^+$ is absolutely clean and $N$ is absolutely clean if and only if $N^+$ is level. This duality captures the above dualities for the special cases of coherent and Noetherian rings. Moreover there are interesting characterizations of coherent rings in terms of level modules and Noetherian rings in terms of absolutely clean modules. For example, a ring $R$ is (right) coherent if and only if all level (left) modules are flat.

Finally, this duality is the key to finding a well-behaved notion of Gorenstein modules for general rings $R$.
We now define a module $M$ to be \emph{Gorenstein AC-injective} if $M = Z_0 X$ for some exact complex of injectives $X$ for which $\Hom_R(I,X)$ remains exact for any absolutely clean module $I$. We see that when $R$ is coherent, the Gorenstein AC-injective modules coincide with the Ding injective modules of~\cite{gillespie-Ding-Chen rings} and when $R$ is Noetherian, the Gorenstein AC-injective modules coincide with the usual Gorenstein injective modules. On the other hand we define a module $M$ to be \emph{Gorenstein AC-projective} if $M = Z_0X$ for some exact complex of projectives $X$ for which $\Hom_R(X,L)$ remains exact for all level modules $L$. Then for coherent rings, these coincide with the Ding projective modules of~\cite{gillespie-Ding-Chen rings}. If every level module has finite projective dimension then they coincide with the usual Gorenstein projectives.  In particular, when $R$ is coherent and satisfies that all flat modules have finite projective dimension, then the Gorenstein AC-projectives are exactly the Gorenstein projectives. A main result of~\cite{bravo-gillespie-hovey} is the following, which appear there as Theorem~5.5/Proposition~5.10 and Theorem~8.5/Proposition~8.10.

\begin{fact}\label{fact-Hovey-Bravo paper}
Let $R$ be any ring. Denote the class of Gorenstein AC-injective modules by $\class{GI}$ and the class of Gorenstein AC-projective modules by $\class{GP}$. Then $(\leftperp{\class{GI}},\class{GI})$ is an injective cotorsion pair and $(\class{GP},\rightperp{\class{GP}})$ is a projective cotorsion pair. Each are cogenerated by a set.

\end{fact}

This result allows for the following extension of the recollements in Sections~\ref{sec-Gorenstein projective and injective models for the derived category} and~\ref{sec-more recollements} to arbitrary rings.

\begin{theorem}\label{them-Gorenstein AC-injective recollements}
Let $R$ be any ring. Denote the class of all Gorenstein AC-injective modules by $\class{GI}$, and set $\class{W} = \leftperp{\class{GI}}$. Then all of the pairs below are injective cotorsion pairs in $\ch$. For each choice of $\class{M}_1$, $\class{M}_2$, and $\class{M}_3$ shown below, Theorem~\ref{them-finding recollements} yields a recollement.
\begin{enumerate}
\item $\class{M}_1 = ({}^\perp\dwclass{GI}, \dwclass{GI}) , \ \ \ \class{M}_2 = ({}^\perp\exclass{GI}, \exclass{GI}) , \ \ \ \class{M}_3 = (\class{E}, \dgclass{I})$

\

\item $\class{M}_1 = (\tilclass{W}, \dgclass{GI}) , \ \ \ \class{M}_2 = (\dgclass{W}, \tilclass{GI}) , \ \ \ \class{M}_3 = (\class{E}, \dgclass{I})$

\

\item $\class{M}_1 = (\exclass{W}, (\exclass{W})^\perp) , \ \ \ \class{M}_2 = (\dwclass{W}, (\dwclass{W})^\perp) , \ \ \ \class{M}_3 = (\class{E}, \dgclass{I})$

\

\item $\class{M}_1 = ({}^\perp\dwclass{GI}, \dwclass{GI}) , \ \ \ \class{M}_2 = ({}^\perp\dwclass{I}, \dwclass{I}) , \ \ \ \class{M}_3 = (\dwclass{W}, (\dwclass{W})^\perp)$

\

\item $\class{M}_1 = ({}^\perp\dwclass{GI}, \dwclass{GI}) , \ \ \ \class{M}_2 = ({}^\perp\exclass{I}, \exclass{I}) , \ \ \ \class{M}_3 = (\exclass{W}, (\exclass{W})^\perp)$
\end{enumerate}
\end{theorem}
\begin{remark}
Note that the first three recollements correspond to those in Theorem~\ref{them-three Gorenstein version of Krause recollement} while the fourth and fifth correspond to Theorem~\ref{them-lovely recollements}. When $R$ is Noetherian they coincide with those recollements. Moreover, we have the dual statements corresponding to the three recollements of Theorem~\ref{them-three Gorenstein projective versions of Krause recollement} and the two in Remark~\ref{remark-loevely recollements dual}. They coincide with those recollements whenever all level modules have finite projective dimension. In particular, whenever $R$ is coherent and all flat modules have finite projective dimension.
\end{remark}

\begin{proof}
Since $(\class{W}, \class{GI})$ is an injective cotorsion pair in $R$-Mod which is cogenerated by a set, it follows then from Proposition~\ref{prop-induced injective cot pairs on complexes} that all the above are injective cotorsion pairs in $R$-Mod. Now the proof of the first three recollements follows exactly as in Theorem~\ref{them-three Gorenstein version of Krause recollement} while the fourth and fifth follows just as in~Theorem~\ref{them-lovely recollements}. It all works the same way since $\class{W} \cap \class{GI}$ equals the class of injective modules, making it easy to verify the hypotheses of Theorem~\ref{them-finding recollements}.

\end{proof}

\end{document}